%% file: branch-structure-harmonic-arxiv.tex
\documentclass[a4paper,11pt]{amsart}
\usepackage{amssymb, amsmath, amsthm}
\usepackage{mdwlist}

\input{mymacros}

\begin{document}
\setlength\parskip{5pt}
\title[Fine properties of branch point singularities]{Fine properties of branch point singularities: two-valued harmonic functions}
\author{Brian Krummel \& Neshan Wickramasekera} 
\thanks{Department of Pure Mathematics and Mathematical Statistics\\
\indent University of Cambridge\\ \indent Cambridge, CB3 0WB, United Kingdom.\\ \indent 
B.~Krummel@dpmms.cam.ac.uk,  N.Wickramasekera@dpmms.cam.ac.uk}

\begin{abstract}
In the 1980's, Almgren developed a theory of multi-valued Dirichlet energy minimizing  functions on an $n$ dimensional domain and used it, in an essential way, to bound the Hausdorff dimension  of the singular sets of area minimizing rectifiable currents of dimension $n$ and codimension $\geq 2$. Recent work of  the second author shows that  two-valued $C^{1, \mu}$ harmonic functions on $n$ dimensional domains, which are typically-non-minimizing stationary points of Dirichlet energy,  play an essential role in the study of multiplicity 2 branch points of stable codimension 1 rectifiable currents of dimension $n$. In all of these cases (of multi-valued harmonic functions and minimal currents), it is known that the branch sets have Hausdorff dimension $\leq n-2.$ 
In this paper we initiate a study of the local structure of branch sets. We show that the branch set of a two-valued Dirichlet energy minimizing function or  a two-valued $C^{1, \mu}$ harmonic function on an $n$-dimensional domain is, in each closed ball of its domain, either empty or has positive $(n-2)$-dimensional Hausdorff measure and is equal  to the union of  a finite number of locally compact, locally $(n-2)$-rectifiable sets. Our method is inspired by the work of L.~Simon on the structure of singularities of minimal submanifolds in compact, multiplicity 1 classes. 
\end{abstract}

\newcommand{\op}[1]{\operatorname{\text{\rm #1}}} 
\renewcommand\div{\op{div}}

\maketitle
\newtheorem{defn}[theorem]{Definition}
\renewcommand{\thedefn}{\arabic{section}.\arabic{defn}}
\newtheorem{rmk}[theorem]{Remark}
\renewcommand{\thermk}{\arabic{section}.\arabic{rmk}}

\tableofcontents
\section{Introduction} 
This is the first in a series of papers in which we study the local structure of sets of branch point singularities of certain minimal submanifolds, multi-valued harmonic functions and multi-valued solutions to more general second order elliptic PDEs. In the present paper,  we consider two of the simplest classes of objects in which the branching phenomenon naturally occurs: (i) two-valued Dirichlet energy minimizing functions and 
(ii)  two-valued $C^{1, \mu}$ harmonic functions, defined on an open subset of ${\mathbb R}^{n}$ and taking values in the space of unordered pairs of elements in ${\mathbb R}^{m}$, for arbitrary $n, m \geq 1$. 
Given a function $u$ belonging to either of these two classes, its (interior) \emph{branch set} ${\mathcal B}_{u}$ is defined to be the set of points in its domain about which there is no neighborhood 
in which the values of the function are given by two single-valued harmonic functions.

Here we establish countable $(n-2)$-rectifiability together with a
certain local-finiteness-of-measure property for ${\mathcal B}_{u}$ whenever $u$ belongs to one of the above two classes; see Theorems $A$, $B$ and $C$  below for the precise statements of what we prove. Since ${\mathcal B}_{u}$ has the property that ${\mathcal H}^{n-2} \, ({\mathcal B}_{u} \cap B_{\r}(y)) > 0$ for each $y \in {\mathcal B}_{u}$ and each  ball $B_{\r}(y)$ contained in the domain of $u$, our theorems  always have non-trivial content whenever ${\mathcal B}_{u} \neq \emptyset.$ Moreover, an immediate implication of our work (here and forthcoming)---one that has important consequences for the \emph{regularity theories} for the relevant classes of objects (cf.\ \cite{KW2}, \cite{H})---is that branch sets have zero $2$-capacity. Prior to our work, it has only been known that ${\mathcal B}_{u}$ 
has Hausdorff dimension $(n-2)$.

Let us now give the context of our work in the present paper, followed by a description of our specific results in some detail. 

Case (i), i.e.\ the two-valued Dirichlet energy minimizing functions,  is the special two-valued case of the 
multi-valued Dirichlet energy minimizing functions introduced by Almgren in his fundamental work, published posthumously in \cite{Almgren},  on interior regularity of area minimizing rectifiable currents of codimension $\geq 2.$ In area minimizing currents of higher codimension, in contrast to those of codimension 1, the occurrence of branch point singularities, where tangent cones are higher multiplicity planes, is a well known fact; and as shown by Almgren,  the local analysis of branch points of area minimizers requires multi-valued Dirichlet energy minimizing functions, a certain generalization of classical harmonic functions introduced by Almgren himself, which themselves typically carry branch points. In the first part of the monumental work \cite{Almgren}, Almgren established an existence theory (in the Sobolev space $W^{1, 2}$) for multi-valued Dirichlet energy minimizing functions, and obtained sharp regularity estimates for them, including the sharp upper bound on the Hausdorff dimension of their singular sets (which includes branch points if any exist). This regularity result says that a $q$-valued Dirichlet energy minimizing function on a domain $\Omega \subset {\mathbb R}^{n}$ is H\"older continuous in the interior of $\Omega$, and moreover, away from a closed set $\Sigma \subset \Omega$ of Hausdorff dimension $\leq n-2$, its values are locally given by $q$ single valued harmonic functions, no two of which have a common value unless they are  identical. 
Almgren then used this result, in the second part of \cite{Almgren}, to establish via an intricate approximation argument the fundamental regularity result   that the Hausdorff dimension of the interior singular set of an $n$-dimensional area minimizing rectifiable current is at most $n-2.$ This is the sharp dimension bound on the singular set of an area minimizing rectifiable current of codimension $\geq 2$. (See also the very recent series of articles by De~Lellis and Spadaro (\cite{DS1})).

In a similar vein,  case (ii) with $m=1$, i.e.\  $C^{1, \mu}$ harmonic functions taking values in the space of unordered pairs of real numbers, arises in the study of multiplicity 2 branch points of stable minimal hypersurfaces. Here, by definition, a multi-valued $C^{1, \mu}$ harmonic function is a multi-valued $C^{1, \mu}$ function whose values away from its branch set are given locally by single valued harmonic functions. These functions are not Dirichlet energy minimizing in general, but they satisfy an important stationarity property, namely   the variational identities (\ref{monotonicity_identity1}) and (\ref{monotonicity_identity2}) below.   In the second author's work \cite{Wic08} and \cite{Wic},  it is shown that two-valued $C^{1, \mu}$ harmonic functions, for a fixed $\mu \in (0, 1)$ independent  of the functions, approximate, in a certain precise sense,  a stable minimal hypersurface near any point $z$ where the hypersurface has a varifold tangent cone equal to a multiplicity 2 hyperplane; estimates for these two-valued $C^{1, \mu}$ harmonic functions then yield, via this approximation scheme, corresponding $C^{1, \mu^{\prime}}$ regularity  for the stable hypersurface near $z,$ for some fixed positive $\mu^{\prime} < \mu$. Results of Simon and the second author (\cite{SW}) can then be applied directly to conclude, both for multi-valued $C^{1, \mu}$ harmonic functions on a domain in ${\mathbb R}^{n}$ and for stable codimension 1 rectifiable $n$-currents,  that the Hausdorff dimension of the set of multiplicity 2 branch points is $n-2$.

In all of the situations described above, it has remained open, in general dimensions $n \geq 3,$ what one can say concerning the \emph{fine properties} of branch sets, namely, those properties  beyond their Hausdorff dimension. This is the question we are interested in here and in the sequels to this paper.  In the present paper  we initiate a study of fine properties of branch sets by establishing several results, the first of which (deduced from Theorem~\ref{theorem2} and Theorem~\ref{theorem3} below) is the following:

\noindent
{\bf Theorem A.} \emph{Let $u$ be a two-valued locally $W^{1, 2}$ Dirichlet energy minimizing function or a two-valued $C^{1, \mu}$ harmonic function on an open subset $\Omega$ of ${\mathbb R}^{n},$ and let ${\mathcal B}_{u}$ be the branch set of $u$. If $n=2$ then ${\mathcal B}_{u}$ is discrete. If $n \geq 3$ then for each  closed ball $B \subset \Omega$, either ${\mathcal B}_{u} \cap B = \emptyset$ or ${\mathcal B}_{u} \cap B$ has positive $(n-2)$-dimensional Hausdorff measure and is equal to the union of a finite number of 
pairwise disjoint, locally compact sets each of which is locally $(n-2)$-rectifiable (and has in particular locally finite $(n-2)$-dimensional  Hausdorff measure)}.  

A primary source of examples of two-valued functions of the type considered in this paper is the set of functions 
on ${\mathbb R}^{2}$ of the form $u(x_{1}, x_{2}) = \{\pm{\rm Re} \,\left(c(x_{1}+ ix_{2})^{k/2}\right)\}$, for 
constants $c \in {\mathbb C}^{m}$ and constant positive integers $k$.  One can of course regard these as functions on ${\mathbb R}^{n}$ independent of $(n-2)$-variables, and as such they are Dirichlet energy minimizing for appropriate choices of $c$, and $C^{1, \mu}$ harmonic for any $c \in {\mathbb C}^{m}$ and integer $k \geq 2.$ The following result, which is a consequence of our main estimates leading to Theorem A, says that any two-valued Dirichlet energy minimizing function or a two-valued $C^{1, \mu}$ harmonic function is asymptotic to one of these examples near 
${\mathcal H}^{n-2}$-a.e.\ branch point:

\noindent
{\bf Theorem B.} \emph{ If 
$u$ is  a two-valued locally $W^{1, 2}$ Dirichlet energy minimizing function or a two-valued $C^{1, \mu}$ harmonic function on an open subset of ${\mathbb R}^{n}$, then for ${\mathcal H}^{n-2}$-a.e.\ point $Z$ in the branch set ${\mathcal B}_{u}$ of $u$ there exist an orthogonal rotation $Q_{Z}$ of ${\mathbb R}^{n}$, a positive integer $k_{Z}$ ($\geq 3$ in case $u \in C^{1, \mu}),$ a constant $c_{Z} \in {\mathbb C}^{m} \setminus \{0\}$ and a number $\rho_{Z} >0$ such that for each $X \in B_{\r_{Z}}(0)$, 
\begin{eqnarray*}
&\hspace{-5.6in}u(Z + Q_{Z}X) = \\ 
&\hspace{-.05in} \left\{h(Z+ Q_{Z}X) +{\rm Re} \, \left(c_{Z}(x_{1}+ ix_{2})^{k_{Z}/2}\right) + \e_{Z}(X), 
h(Z+ Q_{Z}X) -{\rm Re} \, \left(c_{Z}(x_{1} + ix_{2})^{k_{Z}/2}\right) - \e_{Z}(X)\right\},
\end{eqnarray*} 
where 
$h$ is a single-valued harmonic function independent of $Z$ (in fact $h=$ the average of the two values of $u$) and $\{\pm \e_{Z}\} \, : \, B_{\rho_{Z}} (0) \to {\mathcal A}_{2}({\mathbb R}^{m})$ is a symmetric two-valued function with 
$$\s^{-n}\int_{B_{\s}(0)}|\e_{Z}|^{2} \leq C_{Z} \, \s^{k_{Z}+ \g_{Z}}$$ 
for all $\s \in (0, \r_{Z})$ and some constants $C_{Z}, \g_{Z}>0$ independent of $\s.$ 
Here we have used the notation $X = (x_{1}, x_{2}, \ldots, x_{n})$.}

\noindent
{\bf Remark:} In case $u$ is Dirichlet energy minimizing, it follows from the work of Micallef and White (\cite{MW},  Theorem~3.2) that the constant $c_{Z} \in {\mathbb C}^{m} \setminus \{0\}$ in the conclusion of Theorem~B satisfies $c_{Z} \cdot c_{Z} = 0$ whenever $k_{Z}$ is odd.

Much more is true near any branch point $Z$ where the expansion as in Theorem B is valid with the value of $k_{Z}$ equal to the minimum possible value for the relevant class (which is 1 for two-valued Dirichlet energy minimizing functions and 3 for two-valued $C^{1,\mu}$ harmonic functions):

\noindent 
{\bf Theorem C.} \emph{If either  $u$ is two-valued locally $W^{1, 2}$ Dirichlet energy minimizing and $k_{Z_{0}} = 1$ or $u$ is two-valued $C^{1, \mu}$ harmonic and $k_{Z_{0}} = 3$ for some point $Z_{0} \in {\mathcal B}_{u}$ at which the asymptotic expansion as in Theorem B holds, then the expansion above is valid for \emph{each} $Z \in {\mathcal B}_{u} \cap B_{\r_{Z_{0}}/2}(Z_{0})$ (where $\r_{Z_{0}}$ is as in Theorem B) with $k_{Z} = k_{Z_{0}}$, $C_{Z} = C_{Z_{0}}$, $\g_{Z} = \g_{Z_{0}}$ and $\r_{Z} = \r_{Z_{0}}/4.$ Furthermore, letting $\{\pm \e_{Z}\}$ be as in Theorem B, we have in this case that
$$\sup_{B_{\s}(0)} |\e_{Z}|^{2} \leq C_{0} \, \s^{k_{Z_{0}} + \frac{\g_{Z_{0}}}{2n}}$$
for each $Z \in {\mathcal B}_{u} \cap B_{\r_{Z_{0}}/2}(Z_{0})$ and $\s \in (0, \r_{Z_{0}}/4)$, 
where  $C_{0}$ is independent of $\s$ and $Z$; we also have  that  
${\mathcal B}_{u} \cap B_{\r_{Z_{0}}/2}(Z_{0})$ is an $(n-2)$-dimensional $C^{1, \alpha}$ submanifold for some $\alpha = \alpha_{Z_{0}} \in (0, 1)$.}  

\emph{In fact, the asymptotic expansion as in Theorem B  is valid at any point $Z_{0} \in {\mathcal B}_{u}$ at which one ``blow-up'' of $v = u - h$, after composing with an orthogonal rotation of ${\mathbb R}^{n}$,  has the form $\{\pm {\rm Re} \, \left(c(x_{1} + ix_{2})^{k/2}\right)\}$ with $c \in {\mathbb C}^{m} \setminus \{0\}$ a constant and with $k=1$ if $u$ is Dirichlet energy minimizing and $k = 3$ 
if $u$ is $C^{1, \mu}$ harmonic.}

In case dimension $n=2$, the fact that branch points are isolated was previously known. Indeed, this has been shown for area minimizing rectifiable currents in \cite{C}, for multi-valued Dirichlet energy minimizing functions in \cite{DS} and for two-valued $C^{1, \mu}$ harmonic functions in \cite{SW};  for two-valued $C^{1,\mu}$ solutions to the minimal surface system this fact follows essentially from the work of \cite{SW} and is pointed out explicitly in our forthcoming work \cite{KW1}; and finally, by applying this last result, it is shown in \cite{Wic} that multiplicity 2 branch points of stable two-dimensional rectifiable currents with no boundary in an open set in ${\mathbb R}^{3}$ are isolated. 

In forthcoming work, we shall show, based heavily on our work here, that Theorems A, B and C hold for two-valued $C^{1, \mu}$ solutions $u$ to the minimal surface system. This in turn implies, in view of the results of \cite{Wic}, that if $T$ is a stable integral $n$-current with no boundary in an open subset of ${\mathbb R}^{n+1}$ (or, more generally, of an $(n+1)$-dimensional Riemannian manifold),  and if ${\mathcal B}_{T}$ denotes the set of branch points of $T$ (which is the set of points 
$Z \in {\rm spt} \, \|T\|$ at which $T$ has a tangent cone supported on a union of hyperplanes but ${\rm spt} \, \|T\|$ is not immersed at $Z$),  then Theorem A and analogues of Theorems B and C hold with ${\mathcal B}_{T} \cap 
\{X \, : \, \Theta \, (\|T\|, X) < 3\}$ in place of ${\mathcal B}_{u}$  (and with $k_{Z_{0}} = 3$ in Theorem~C). Here $\|T\|$ denotes the weight measure  and $\Theta \, (\|T\|, \cdot)$ is the density function associated with $T$. Also in forthcoming work, we shall treat the general multi-valued cases of Theorems A, B and C; these extensions require both new ideas  and more-or-less-direct analogues/generalizations of some of the key  a priori estimates established in the present work. 

\noindent
{\bf A brief outline of the method:} In the pioneering work \cite{SimonCTC}, L.~Simon established rectifiability properties (as in Theorem A, but without the local positivity of measure) for the singular sets of minimal submanifolds belonging to certain compact, ``multiplicity 1'' classes, i.e.\  classes in which occurrence of higher multiplicity in weak  limits is  a priori ruled out. In that work, Simon employed a ``blowing-up'' argument to establish asymptotic decay estimates near a.e.\ ``top dimensional'' singularity for minimal submanifolds in such classes. The key to his method is the usual monotonicity formula for minimal submanifolds. 

In this paper we follow Simon's strategy. However, since our work in contrast to Simon's concerns singular sets in the presence of higher multiplicity, direct application of his results or methods is not possible. We employ a blowing-up argument similar to that of Simon's,  relying firmly on both the celebrated frequency function monotonicity formula due to Almgren (\cite{Almgren}; identity (\ref{monotonicity_eqn3}) below) and on a certain variant of it (Lemma~\ref{new-monotonicity}). The latter, which appears not to have been utilised in the analysis of branch points before, says that if $u$ is (say) a two-valued Dirichlet energy minimizing function 
on a domain $\Omega \subset {\mathbb R}^{n}$ and $\alpha \in {\mathbb R}$, then for each $Y \in \Omega$, 
$$\hspace{.5in}\frac{d}{d\rho}\left(\rho^{-2\alpha} (D_{u,Y}(\rho) - \alpha H_{u,Y}(\rho))\right)
= 2\rho^{2-n}\int_{\partial\,B_{\rho}(Y)} \left| \frac{\partial (u/R^{\alpha})}{\partial R} \right|^2 \hspace{1in} (\star)$$ 
for every $\rho \in (0, {\rm dist}\, (Y, \partial \, \Omega)),$ where $R(X) = |X - Y|$, $D_{u, Y}(\r) = \r^{2-n} \int_{B_{\r}(y)}|Du|^{2}$ and $H_{u, Y}(\r) = \r^{1-n}\int_{\partial \, B_{\r}(Y)} |u|^{2}$. Note that the right hand side of this identity can be regarded as measuring how far $u$ is from being homogeneous of degree $\a$, and the expression being differentiated on the left hand side is, by the frequency function monotonicity (\ref{monotonicity_eqn3}), non-negative for all $\r \in (0, {\rm dist} \, (Y, \partial \, \Omega))$ and appropriate values of $\a$. 

Using ($\star$) and the two variational identities (\ref{monotonicity_identity1}), (\ref{monotonicity_identity2}) underpinning both ($\star$) and the Almgren frequency function monotonicity, we establish a key new a priori estimate for $u$ (Theorem~\ref{thm6_2}(a)) which says the following: Let $\varphi_{0}$ be a non-zero two-valued cylindrical homogeneous Dirichlet energy minimizing function  with zero average and degree of homogeneity $\alpha >0.$ If $\varphi$ is a non-zero two-valued cylindrical homogeneous Dirichlet energy minimizing function  with zero average and degree of homogeneity $\alpha$, and $u$ is a two-valued Dirichlet energy minimizing function on $B_{2}(0)$ with zero average and frequency at the origin $\geq \alpha$, then 
$$\hspace{1.4in}\int_{B_{1/2}(0)} R^{2-n}\left| \frac{\partial (u/R^{\alpha})}{\partial R} \right|^2 \leq  C\int_{B_{1}(0)} {\mathcal G}(u, \varphi)^{2} \hspace{1.5in} (\star\star)$$
provided $\int_{B_{1}(0)} {\mathcal G}(\varphi, \varphi_{0})^{2}$ and $\int_{B_{1}(0)} {\mathcal G}(u, \varphi_{0})^{2}$ are both sufficiently small (depending on $\varphi_{0}$), 
where $C = C(n, m, \varphi_{0}) \in (0, \infty)$ and ${\mathcal G}$ denotes the standard metric on the space of unordered pairs of vectors in ${\mathbb R}^{m}$ (defined in  Section~\ref{multivalfun_section}). 

Given these ingredients, our argument can briefly be summarised as follows: The Almgren monotonicity implies the (well-known) fact that  if $\varphi$ is any blow-up of $u$ at a branch point, 
then $\varphi$ is homogeneous of some degree $\alpha >0$ and is non-trivial (but possibly not unique); that is to say, if $Z$ is a branch point of $u$ and if $\r_{j} \to 0^{+}$ as $j \to \infty$, then after passing to a subsequence, 
$$\frac{u(Z + \r_{j}(\cdot))}{\|u(Z + \r_{j}(\cdot))\|_{L^{2}(B_{1}(0))}} \to \varphi_{0}$$
locally in $W^{1, 2}$ for some non-zero two-valued function $\varphi$ homogeneous of degree $\alpha$ for some   $\alpha >0$. In particular, near $Z$, the function $u$ rescaled sufficiently is close to $\varphi_{0}$ to render $(\star\star)$ valid. On the other hand, by implementing Simon's blow-up argument staring with $(\star\star)$,  we arrive at the following conclusion: Let $\varphi_{0}$ be a homogeneous cylindrical two-valued Dirichlet energy minimizing function with zero average and degree of homogeneity $\alpha$, and let $u$ be a two-valued Dirichlet energy minimizing function with zero average, frequency at the origin $\geq \alpha$ and sufficiently close (depending on $\varphi_{0}$) to $\varphi_{0}$ at unit scale. If $u$ satisfies a further hypotheses  concerning the distribution of branch points (one that is a posteriori satisfied  at ${\mathcal H}^{n-2}$-a.e.\ branch point of a minimizer at all sufficiently small scales), then $\r^{-\a}u(\r(\cdot))$ 
decays, as $\r \to 0^{+}$,  to a unique homogeneous cylindrical limit $\widetilde{\varphi}$ lying near  $\varphi_{0}$, at a rate given by a fixed positive power of $\r$; that is, whenever $u$ satisfies appropriate hypotheses including being sufficiently close at unit scale to $\varphi_{0}$, there exists a (unique) homogeneous cylindrical two-valued harmonic function $\widetilde\varphi$ such that 
$$\|\r^{-\alpha} u(\r(\cdot)) - \widetilde\varphi\|_{L^{2}(B_{1}(0))} \leq C\r^{\g}$$  
for all $\r \in (0, 1/2]$, where $C, \g >0$ are independent of $\r$. (To be even more precise, this is the conclusion one draws by iteratively applying Lemma~\ref{lemma1} below after choosing $\theta$ suitably, under the assumption that the rescaled functions $u_{j} = u(\th^{j}(\cdot))$ fail to satisfy alternative (i) of the lemma for all $j=1, 2, 3, \ldots$.) This and its counterpart for $C^{1, \mu}$ two-valued harmonic functions are the key estimates that lead to Theorems A , B and C.

\section{Preliminaries: two classes of two-valued harmonic functions} \label{multivalfun_section}
\setcounter{equation}{0}
\noindent
{\bf Two-valued functions}.
Let $\mathcal{A}_2(\mathbb{R}^m)$ denote the space of unordered pairs $\{a_1,a_2\}$ with $a_1,a_2 \in \mathbb{R}^m$ not necessarily distinct.  We equip $\mathcal{A}_2(\mathbb{R}^m)$ with the metric
\begin{equation*}
	\mathcal{G}(a,b) = \min \{\sqrt{|a_1-b_{1}|^2 + |a_{2} - b_{2}|^{2}}, \sqrt{|a_{1} - b_{2}|^{2} + |a_{2} - b_{1}|^{2}}\}
\end{equation*}
for $a = \{a_1, a_2\} \in {\mathcal A}_{2}({\mathbb R}^{m})$ and  $b = \{b_1, b_2\} \in {\mathcal A}_{2}({\mathbb R}^{m})$. For $a = (a_{1}, a_{2}) \in {\mathcal A}_{2}({\mathbb R}^{m})$, we let 
\begin{equation*}
	|a| = \mathcal{G}(a,\{0,0\}) = \sqrt{|a_1|^2 + |a_{2}|^{2}}.
\end{equation*}
A two-valued function on a set $\Omega \subseteq \mathbb{R}^n$ is a map $u : \Omega \rightarrow \mathcal{A}_{2}(\mathbb{R}^m)$. 

Let $\Omega \subseteq \mathbb{R}^n$ be open.  Since $\mathcal{A}_2(\mathbb{R}^m)$ is a metric space, we can define the space of continuous two-valued functions $C^0(\Omega;\mathcal{A}_2(\mathbb{R}^m))$ on $\Omega$ in the usual way.  For each $\mu \in (0,1]$, we define $C^{0,\mu}(\Omega;\mathcal{A}_2(\mathbb{R}^m))$ to be the space of two-valued functions $u : \Omega \rightarrow \mathcal{A}_2(\mathbb{R}^m)$ such that 
\begin{equation*}
	[u]_{\mu,\Omega} = \sup_{X,Y \in \Omega, X \neq Y} \frac{\mathcal{G}(u(X),u(Y))}{|X-Y|^{\mu}} < \infty. 
\end{equation*}
We say a two-valued function $u : \Omega \rightarrow \mathcal{A}_2(\mathbb{R}^m)$ is \textit{differentiable} at $Y \in \Omega$ if there exists a two-valued affine function $\ell_{Y} : \mathbb{R}^n \rightarrow \mathcal{A}_2(\mathbb{R}^m)$, i.e.\ a two-valued function $\ell_{Y}$ of the form $\ell_{Y}(X) = \{ A_1^{Y} X + b_1^{Y}, A_2^{Y} X + b_2^{Y} \}$ for some $m \times n$ real-valued constant matrices $A_1^{Y},A_2^{Y}$ and constants $b_1^{Y}, b_2^{Y} \in \mathbb{R}^m$, such that 
\begin{equation*} 
	\lim_{X \rightarrow Y} \frac{\mathcal{G}(u(X),\ell_{Y}(X))}{|X-Y|} = 0. 
\end{equation*}
$\ell_{Y}$ is unique if it exists. The derivative of $u$ at $Y$ is the unordered pair of $m \times n$ matrices $Du(Y) = \{A_1^{Y},A_2^{Y}\}$.  

We define $C^1(\Omega;\mathcal{A}_2(\mathbb{R}^m))$ to be the space of two-valued functions $u : \Omega \rightarrow \mathcal{A}_2(\mathbb{R}^m)$ such that the derivative of $u$ exists and is continuous on $\Omega$.  For $\mu \in (0,1]$, we define $C^{1,\mu}(\Omega;\mathcal{A}_2(\mathbb{R}^m))$ as the set of $u \in C^1(\Omega;\mathcal{A}_2(\mathbb{R}^m))$ such that $Du \in C^{0,\mu}(\Omega;\mathcal{A}_2(\mathbb{R}^{nm}))$.

$L^2(\Omega;\mathcal{A}_2(\mathbb{R}^m))$ denotes the space of Lebesgue measurable two-valued functions $u : \Omega \rightarrow \mathcal{A}_2(\mathbb{R}^m)$ such that $\|u\|_{L^{2} \, (\Omega)} \equiv \|\mathcal{G}(u(X),0)\|_{L^2(\Omega)} < \infty$.  We equip $L^2(\Omega;\mathcal{A}_2(\mathbb{R}^m))$ with the metric 
\begin{equation*}
	d(u,v) = \left( \int_{\Omega} \mathcal{G}(u(X),v(X))^2 \right)^{1/2}
\end{equation*}
for $u,v \in L^2(\Omega;\mathcal{A}_q(\mathbb{R}^m)).$ 

The Sobolev space $W^{1,2}(\Omega;\mathcal{A}_2(\mathbb{R}^m))$ of two-valued functions is defined and discussed in~\cite{Almgren}, Chapter 2. (See \cite{DS} for an equivalent characterization of $W^{1,2}(\Omega;\mathcal{A}_2(\mathbb{R}^m))$.)

 We say that a two-valued function $u : \Omega \rightarrow \mathcal{A}_2(\mathbb{R}^m)$ is \textit{symmetric} if for each $X \in \Omega$, $u(X) = \{-u_1(X),u_1(X)\}$ for some $u_1(X) \in \mathbb{R}^m$.  Given $u : \Omega \rightarrow \mathcal{A}_2(\mathbb{R}^m)$ with $u(X) = \{u_{1}(X), u_{2}(X)\}$ for each $X \in \Omega$, we can write $u = u_a + u_s$ where $u_a(X) = (u_1(X)+u_2(X))/2$ and $u_s : \Omega \rightarrow \mathcal{A}_2(\mathbb{R}^m)$ is the symmetric two-valued function given by $u_s(X) = \{ \pm (u_1(X) - u_2(X))/2 \}$ for $X \in \Omega$.\\

\noindent
{\bf Two-valued Dirichlet energy minimizing  functions}.
\begin{defn}
A two-valued function $u \in W^{1,2}_{\rm loc}(\Omega;\mathcal{A}_2(\mathbb{R}^m))$ is \textit{Dirichlet energy minimizing} (on $\Omega$) if 
\begin{equation*}
	\int_{\Omega} |Du|^2 \leq \int_{\Omega} |Dv|^2
\end{equation*}
whenever $v \in W^{1,2}_{\rm loc}(\Omega;\mathcal{A}_2(\mathbb{R}^m))$ and $v = u$ a.e. on $\Omega \setminus K$ for some compact set $K \subset \Omega$.
\end{defn}

Following Almgren~[\cite{Almgren}, Theorem~2.14], we define the \emph{singular set} $\Sigma_{u}$ of a two-valued Dirichlet energy minimizing function $u \in W^{1, 2} \, (\Omega; {\mathcal A}_{2}({\mathbb R}^{m}))$ to be the set of points $Y \in \Omega$ such that there is no $\r> 0$ with the following property:  
There exist single-valued harmonic functions $u_1, u_{2} \, : \, B_{\r}(Y) \to \mathbb{R}^m$ with either $u_1 \equiv u_2$ on $B_{\r}(Y)$ or $u_1(X) \neq u_2(X)$ for each $X \in B_{\r}(Y)$ such that  $u(X) = \{u_1(X),u_2(X)\}$ for ${\mathcal H}^{n}$-a.e.\ $X \in B_{\r}(Y)$. (Note that such $u_{1}$, $u_{2}$ are unique if they exist.) It is clear that $\Sigma_{u}$ is a closed subset of $\Omega$. 

Almgren developed an existence and regularity theory for Dirichlet energy minimizing multi-valued functions as an essential ingredient in his fundamental work \cite{Almgren} on regularity of area minimizing rectifiable currents of arbitrary dimension and codimension $\geq 2$. Almgren's theory, contained in [\cite{Almgren}, Chapter 2],  in particular establishes the following  results (see \cite{DS} for a nice, concise  exposition of Almgren's theory): 
\begin{itemize}
\item[(i)] [\cite{Almgren}, Theorem~2.2], which asserts the 
existence, in the Sobolev space $W^{1, 2}$, of Dirichlet energy minimizing functions with prescribed boundary values on bounded $C^{1}$ domains.
\item[(ii)] [\cite{Almgren}, Theorem~2.6], which implies that if $u  = \{u_{1}, u_{2}\} \in W^{1,2}(\Omega,\mathcal{A}_2(\mathbb{R}^m))$ is Dirichlet energy minimizing, then the symmetric two-valued function $v$ given by $v(X) = \{ \pm (u_1(X)-u_2(X))/2 \}$ for $X \in \Omega$ is Dirichlet energy minimizing; and of course $\Sigma_u = \Sigma_v \subseteq \{ X : v(X) = \{0,0\} \}$ (note that by (iii) below, $u, v$ are continuous). 
\item[(iii)] [\cite{Almgren}, Theorem~2.13],  which establishes a sharp interior regularity theory for Dirichlet energy minimizers, according to which such a minimizer is locally uniformly 
H\"older continuous in the interior of its domain. In the case of two-valued functions,  [\cite{Almgren}, Theorem~2.13] implies the following: If $u \in W^{1, 2} \, (\Omega; {\mathcal A}_{2}({\mathbb R}^{m}))$ is Dirichlet energy minimizing, then  
$u \in C^{0,1/2}(\Omega^{\prime};\mathcal{A}_2(\mathbb{R}^m))$ for each open set $\Omega^{\prime}$ with $\Omega^{\prime} \subset\subset \Omega$, and satisfies the estimate 
\begin{equation*}
	\r^{1/2} [u]_{1/2;B_{\r/2}(X_0)} \leq C\r^{1-n/2} \|Du\|_{L^2(B_\r(X_0))} 
\end{equation*}
for each ball $B_{\r}(X_{0})$ with $\overline{B_{\r}(X_{0})} \subset \Omega,$  where $C = C(n,m) \in (0,\infty)$;  for such $u$, it then follows from standard interpolation inequalities and (\ref{monotonicity_identity1}) below that 
\begin{equation} \label{schauder_dm}
	\sup_{B_{\r/2}(X_0)} |u| + \r^{1/2} [u]_{\mu;B_{\r/2}(X_0)} + \r^{1-n/2} \|Du\|_{L^2(B_{\r/2}(X_0))} \leq C\r^{-n/2} \|u\|_{L^2(B_\r(X_0))}  
\end{equation}
for each ball $B_{\r}(X_{0})$ with $\overline{B_{\r}(X_{0})} \subset \Omega,$  where $C = C(n,m) \in (0,\infty)$.  
\item[(iv)] [\cite{Almgren}, Theorem~2.14], which implies  the following: If $u \in W^{1, 2} \, (\Omega; {\mathcal A}_{2}({\mathbb R}^{m}))$ is Dirichlet energy minimizing, then ${\rm dim}_{\mathcal H} \, (\Sigma_{u}) \leq n-2.$
\end{itemize}

We define the branch set $\mathcal{B}_u$ of a Dirichlet energy minimizing two-valued function 
$u$ belonging to $W^{1,2}(\Omega;\mathcal{A}_2(\mathbb{R}^m))$ to be the set of points $Y \in \Omega$ with the property that there is no $\r \in (0,\op{dist}(Y,\partial \Omega))$ such that $u(X) = \{u_1(X), u_{2}(X)\}$ for two single-valued functions $u_1, u_{2} \in W^{1, 2} \, (B_{\r}(Y),\mathbb{R}^m)$ and ${\mathcal H}^{n}$-a.e.\ $X \in B_{\r}(Y)$.  It follows that ${\mathcal B}_{u}$ is a closed subset of $\Omega$ and that ${\mathcal B}_{u} \subset \Sigma_{u}.$ In view of Theorem~\ref{theorem2} which implies that $\Sigma_{u}$ has zero $2$-capacity, it is easy to see that ${\mathcal B}_{u}$ can equivalently be defined as the set of points $Y \in \Omega$ with the property that there is no $\r \in (0,\op{dist} \, (Y,\partial \Omega))$ such that $u(X) = \{u_1(X), u_{2}(X)\}$ for two single-valued  harmonic (and hence real analytic) functions $u_1, u_{2} \, : \, B_{\r}(Y) \to \mathbb{R}^m$ and ${\mathcal H}^{n}$-a.e.\ $X \in B_{\r}(Y)$.

A main goal of this paper is to understand the local structure of $\mathcal{B}_u.$ Rather than focusing directly on ${\mathcal B}_{u},$ it is more convenient to study the entire singular set $\Sigma_{u}$  (which contains ${\mathcal B}_{u}$) since $\Sigma_{u}$ satisfies the following property in relation to convergent sequences of Dirichlet energy minimizers: If $u_j,u \in W^{1,2}(\Omega;\mathcal{A}_2(\mathbb{R}^m))$ are Dirichlet energy minimizing two-valued functions such that $u_j \rightarrow u$ locally in $L^{2}$ (or, equivalently, in view of (\ref{schauder_dm}), uniformly on compact subsets of $\Omega$), then for each $\Omega' \subset \subset \Omega,$ either $\Sigma_{u_j} \cap \Omega' = \emptyset$ for all sufficiently large $j$ or 
$\Sigma_u \cap \Omega' \neq \emptyset$. Note that this property does not hold if we replace $\Sigma_{u_{j}}$ with  ${\mathcal B}_{u_{j}}$ and $\Sigma_{u}$ with ${\mathcal B}_{u}$. (For example, consider $u_j,u : \mathbb{R}^2 \rightarrow \mathbb{C} \cong \mathbb{R}^2$ defined by $u_j(x_1,x_2) = \{ \pm ((x_1+ix_2)^2 - 1/j^2)^{1/2} \}$, $u(x_1,x_2) = \{ \pm (x_1+ix_2) \}$ for $(x_1,x_2) \in \mathbb{R}^2$.)\\

\noindent
{\bf Two-valued $C^{1, \mu}$ harmonic functions.}
For $u \in C^{1} (\Omega; {\mathcal A}_{2}({\mathbb R}^{m}))$, we define the \emph{branch set}
${\mathcal B}_{u}$ of $u$ to be the set of points $Y \in \Omega$ such that there exists no $\r>0$ with the property that $u(X) = \{u_{1}(X), u_{2}(X)\}$ for each $X \in B_{\r}(Y)$ and two single-valued functions $u_{1}, u_{2} \in C^{1} \, (B_{\r}(Y); {\mathbb R}^{m}).$ It is clear that ${\mathcal B}_{u}$ is a closed subset of $\Omega.$ 

\begin{defn}
A two-valued function $u \in C^1(\Omega;\mathcal{A}_2(\mathbb{R}^m))$ is said to be harmonic on $\Omega$ if for every open ball $B_{\r}(Y) \subset \Omega \setminus \mathcal{B}_u$, $u(X) = \{ u_1(X), u_2(X) \}$ on $B_{\r}(Y)$ for two single-valued harmonic (hence real analytic) functions $u_1,u_2 \, : \, B_{\r}(Y) \to {\mathbb R}^{m}$.  Clearly if such $u_{1}$, $u_{2}$ exist, they are unique. 
\end{defn}

Given a two-valued harmonic function $u \in C^{1,\mu}(\Omega;\mathcal{A}_2(\mathbb{R}^m))$, we define 
\begin{align*}
	\mathcal{Z}_u &= \{ X \in \Omega : u_1(X) = u_2(X) \}, \\
	\mathcal{K}_u &= \{ X \in \Omega : u_1(X) = u_2(X), \, Du_1(X) = Du_2(X) \},
\end{align*}
where we write $u(X) = \{u_1(X),u_2(X)\}$ and $Du(X) = \{Du_1(X),Du_2(X)\}$ as unordered pairs for each $X \in \Omega$. It is clear that ${\mathcal B}_{u} \subset {\mathcal K}_{u}$. 

As shown in \cite{Wic08} and \cite{Wic}, $C^{1,\mu}$ harmonic two-valued functions with values in 
${\mathcal A}_{2}({\mathbb R})$ play an essential role in the regularity theory of stable minimal hypersurfaces; and  it is shown in~\cite{SW} that if $\Omega$ is connected (and open), $u \in C^{1,\mu}(\Omega;\mathcal{A}_2(\mathbb{R}^m))$ for some $\mu >0$  and is harmonic in $\Omega$, then 
\begin{itemize}
\item[(i)] either ${\mathcal K}_{u} = \Omega$ or ${\rm dim}_{\mathcal H} \, ({\mathcal K}_{u}) \leq n-2;$ 
\item[(ii)] $u \in C^{1,1/2}(\Omega;\mathcal{A}_2(\mathbb{R}^m)) \cap W_{\text{loc}}^{2,2}(\Omega;\mathcal{A}_2(\mathbb{R}^m)),$ and satisfies the following: 
\begin{equation} \label{schauder_h}
	\sup_{B_{\r/2}(X_0)} |u| + \r \sup_{B_{\r/2}(X_0)} |Du| + \r^{3/2} [Du]_{1/2,B_\r(X_0)} + \r^{2-n/2} \|D^2 u\|_{L^2(B_{\r/2}(X_0))} \leq C\r^{-n/2} \|u\|_{L^2(B_\r(X_0))}
\end{equation} 
for any open ball $B_{\r}(X_{0})$ with $\overline{B_\r(X_0)} \subseteq \Omega,$ where $C = C(m,n) \in (0,\infty)$.  
\end{itemize}

Indeed, if we let $u_a(X) = (u_1(X)+u_2(X))/2$ and $v = \{ \pm (u_1(X) - u_2(X))/2 \}$ for $X \in \Omega$, where $u(X) = \{u_1(X),u_2(X)\}$,  then by Lemma 4.1 of~\cite{SW} applied to $v$, $\mathcal{B}_u$ has Hausdorff dimension at most $n-2$.  Consequently $u_a$ is harmonic on all of $\Omega$. It then follows from standard elliptic estimates and Lemmas 2.1 and 4.1 of~\cite{SW} that $u \in C^{1,1/2}(\Omega;\mathcal{A}_2(\mathbb{R}^m)) \cap W_{\text{loc}}^{2,2}(\Omega;\mathcal{A}_2(\mathbb{R}^m))$ and (\ref{schauder_h}) holds.

To understand the structure of $\mathcal{B}_u,$ it is more convenient to study the larger set $\mathcal{K}_u$ since we have the following property: If $\mu \in (0, 1)$ and $u_j,u \in C^{1,\mu}(\Omega;\mathcal{A}_2(\mathbb{R}^m))$ are harmonic on $\Omega$ for each $j=1, 2, \ldots,$ and if $u_j \rightarrow u$ in locally in $L^{2}$ (or, equivalently, by (\ref{schauder_h}), in $C^1(\Omega';\mathcal{A}_q(\mathbb{R}^m))$ for all $\Omega' \subset \subset \Omega$), then for each $\Omega' \subset \subset \Omega,$ either $\mathcal{K}_{u_j} \cap \Omega' = \emptyset$ for all sufficiently large $j$ or $\mathcal{K}_u \cap \Omega' \neq \emptyset$.  

Finally, note that if $u  = \{u_{1}, u_{2}\} \in C^{1,\mu}(\Omega;\mathcal{A}_2(\mathbb{R}^m))$ is harmonic on $\Omega$, then the symmetric function $v = \{ \pm (u_1-u_2)/2 \}$ is in $C^{1,\mu}(\Omega;\mathcal{A}_2(\mathbb{R}^m))$ and is harmonic on $\Omega$ with 
\begin{align*}
         \mathcal{B}_v &= \mathcal{B}_u,\\
	\mathcal{Z}_v &= \mathcal{Z}_u = \{ X \in \Omega : v(X) = \{0,0\} \} \;\; {\rm and}\\
	\mathcal{K}_v &= \mathcal{K}_u = \{ X \in \Omega : v(X) = \{0,0\}, \, Dv(X) = \{0,0\} \}.  
\end{align*}

\section{Monotonicity of frequency and its preliminary consequences}
\setcounter{equation}{0}
We first consider the case of Dirichlet minimizing functions.  Let $u \in W^{1,2}(\Omega;\mathcal{A}_2(\mathbb{R}^m))$ be Dirichlet energy minimizing, $Y \in \Omega$ and $0 < \rho < \op{dist}(Y, \partial \Omega)$. Following Almgren, define the \emph{frequency function} associated with $u$ and $Y$ by 
\begin{equation} \label{frequency_defn}
	N_{u,Y}(\rho) = \frac{D_{u,Y}(\rho)}{H_{u,Y}(\rho)} 
\end{equation}	
whenever $H_{u, Y}(\r) \neq 0$, where
\begin{equation}\label{energy-height}
	D_{u,Y}(\rho) = \rho^{2-n} \int_{B_{\rho}(Y)} |Du|^2 \hspace{3mm} \text{and} \hspace{3mm}
	H_{u,Y}(\rho) = \rho^{1-n} \int_{\partial B_{\rho}(Y)} |u|^2. 
\end{equation}

A fundamental discovery of Almgren ([\cite{Almgren}, Theorem~2.6]) is that for any multi-valued $W^{1, 2}$ function $u$ which is a stationary point of the Dirichlet energy functional with respect to two types of deformations (the ``squash'' and ``squeeze'' deformations; see \cite{Almgren}, Sections 2.4 and 2.5), and for any point $Y$ in its domain, the associated frequency  function $N_{u, Y} (\cdot)$, if defined for $\r \in I$ for some open interval $I$, is monotone increasing in $I.$ Almgren's argument, briefly, is as follows: The squash and squeeze deformations lead to the identities 

\begin{align} 
	\label{monotonicity_identity1} &\int_{\Omega} |Du|^2 \zeta = -\int_{\Omega} u^{\kappa} D_i u^{\kappa} D_i \zeta, \\
	\label{monotonicity_identity2} &\int_{\Omega} \left( \tfrac{1}{2} |Du|^2 \delta_{ij} - D_i u^{\kappa} D_j u^{\kappa} \right) D_i \zeta^j = 0, 
\end{align}
for all $\zeta, \zeta^1,\ldots,\zeta^n \in C^1_c(\Omega;\mathbb{R})$, where  
$u^{\kappa} = u \cdot e_{\k}$ is the (two-valued) $\kappa$-th coordinate of $u$ (thus 
$u^{\kappa}(X) = \{u_{1}(X) \cdot e_{\k}, u_{2}(X) \cdot e_{\k}\}$), with 
$\{e_{1}, e_{2}, \ldots, e_{m}\}$ denoting the standard basis for ${\mathbb R}^{m}$, and we use the convention of summing over repeated indices;  these identities in turn imply 

\begin{align} 
	\label{energy1} &\int_{B_{\r}(Y)} |Du|^2 = \int_{\partial \, B_{\r}(Y)} u \cdot D_{R}u, \\
	\label{energy2} &\frac{d}{d\r}\left( \r^{2-n}\int_{B_{\r}(Y)}|Du|^{2}\right) = 2\r^{2-n}\int_{\partial \, B_{\r}(Y)} |D_{R} \, u|^{2}, 
\end{align}
for a.e.\ $\r \in (0, {\rm dist} \, (Y, \partial \, \Omega))$, which yield, by direct computation,
\begin{equation} \label{monotonicity_eqn3}
	N'_{u,Y}(\rho) = \frac{2\rho^{1-2n}}{H_{u,Y}(\rho)^2} \left( \left( \int_{\partial B_{\rho}(Y)} |u|^2 \right) 
		\left( \int_{\partial B_{\rho}(Y)} R^2 |D_R u|^2 \right) - \left( \int_{\partial B_{\rho}(Y)} R u \cdot D_R u \right)^2 \right) . 
\end{equation}
Thus $N'_{u,Y}(\rho) \geq 0$ for a.e.\ $\r \in (0, {\rm dist} \, (Y, \partial \, \Omega))$ by the Cauchy-Schwartz inequality. Hence for any $Y \in \Omega$ and any $s$, $t$ with $0 < s < t < {\rm dist} \, (Y, \partial \, \Omega)$, $N_{u, Y} \, (\r)$ is monotonically increasing for $\r \in (s, t)$ whenever $H_{u, Y}(\r) \neq 0$ 
for $\r \in (s, t)$. (See alternatively \cite[Propositions 3.1, 3.2]{DS} or the proof of Lemma 2.2 in~\cite{SW} for details.)

So if for some $t > 0$, $H_{u, Y}(\r) \neq 0$ for all $\r \in (0, t)$, then the limit 
$${\mathcal N}_{u}(Y)  = \lim_{\r \to 0^{+}} \, N_{u, Y}(\r)$$ exists. We call 
${\mathcal N}_{u}(Y)$ the \emph{frequency} of $u$ at $Y$ whenever it exists. 

The monotonicity of $N_{u,Y}$ has the following  standard consequences.  

\begin{lemma} \label{monotonicity_cor} 
Let $u \in W^{1, 2} \, (\Omega; {\mathcal A}_{2}({\mathbb R}^{m}))$ be a non-zero Dirichlet energy minimizing function. Then 
\begin{enumerate}
\item[(a)]  $H_{u, Y}(\r) \neq 0$ for each $Y \in \Omega$ and $\r \in (0, {\rm dist} \, (Y, \partial \, \Omega)); $ ${\mathcal N}_{u}(Y)$ exists for each $Y \in \Omega$, and   
\begin{equation} \label{doubling_estimate}
	\left( \frac{\sigma}{\rho} \right)^{2N_{u,Y}(\rho)} H_{u,Y}(\rho) \leq H_{u,Y}(\sigma) 
	\leq \left( \frac{\sigma}{\rho} \right)^{2\mathcal{N}_u(Y)} H_{u,Y}(\rho) 
\end{equation}
for $0 < \sigma \leq \rho < \op{dist}(Y,\partial \Omega)$. 
\item[(b)] For each $Y \in \Omega$ and $0 < \sigma \leq \rho < \op{dist}(Y,\partial \Omega)$,  
\begin{equation} \label{doubling_estimate2}
	\left( \frac{\sigma}{\rho} \right)^{2N_{u,Y}(\rho)} \rho^{-n} \int_{B_{\rho}(Y)} |u|^2 \leq \sigma^{-n} \int_{B_{\sigma}(Y)} |u|^2 
	\leq \left(  \frac{\sigma}{\rho} \right)^{2\mathcal{N}_u(Y)} \rho^{-n} \int_{B_{\rho}(Y)} |u|^2. 
\end{equation}
\item[(c)] $\mathcal{N}$ is upper semicontinuous in the sense that if $u_j \in W^{1,2}(\Omega;\mathcal{A}_2(\mathbb{R}^m))$ are Dirichlet energy minimizing functions such that $u_j \rightarrow u$ in $L^{2}(\Omega';\mathcal{A}_2(\mathbb{R}^m))$ for all $\Omega' \subset \subset \Omega$, and if $Y_j, Y \in \Omega$ such that $Y_j \rightarrow Y$, then $\mathcal{N}_u(Y) \geq \limsup_{j \rightarrow 0} \mathcal{N}_{u_j}(Y_j)$.

\item[(d)] If $u = \{u_{1}, u_{2}\}$ is symmetric (i.e.\ if $u_{1} + u_{2}= 0$ on $\Omega$), then $\mathcal{N}_u(Y) \geq 1/2$ for every $Y \in \Sigma_u$. 
\item[(e)] For every $X \in \Omega$ such that $0 < \op{dist}(X,\Sigma_u) \leq \op{dist}(X,\partial \Omega)/2$, 
\begin{equation} \label{dist2sing_estimate} 
	|u(X)| + d(X) |Du(X)| + d(X)^2 |D^2 u| \leq 2^{{\mathcal N}_{u}(Y)} C d(X)^{\mathcal{N}_u(Y)} \rho^{-n/2-\mathcal{N}_u(Y)} \|u\|_{L^2(\Omega)}, 
\end{equation} 
where $d(X) = \op{dist}(X,\Sigma_u),$ $\r = {\rm dist} \, (X, \partial \, \Omega)$ and $C = C(n, m) \in (0, \infty)$. 
\end{enumerate} \end{lemma}
\begin{proof}
 (a) follows from the monotonicity of $N_{u,Y}$ and the fact that $N_{u,Y}(\rho) = \rho H'_{u,Y}(\rho)/2 H_{u,Y}(\rho);$ (b) follows from (a) via integration;  $\left(c\right)$ follows from monotonicity of $N_{u, Y};$ (d) follows from (b) and the fact that $u \in C^{0,1/2}(\Omega;\mathcal{A}_2(\mathbb{R}^m))$ for every Dirichlet minimizing two-valued function $u \in W^{1,2}(\Omega;\mathcal{A}_2(\mathbb{R}^m));$  (e) is a consequence of (\ref{doubling_estimate2}) and the Schauder estimates for single-valued harmonic functions: 
\begin{align*}
	&|u(X)| + d(X) |Du(X)| + d(X)^2 |D^2 u| 
	\leq C d(X)^{-n/2} \|u\|_{L^2(B_{d(X)/2}(X))} 
	\leq C d(X)^{-n/2} \|u\|_{L^2(B_{2d(X)}(Y))} 
	\\&\hspace{.5in}\leq 2^{{\mathcal N}_{u}(Y)}C d(X)^{\mathcal{N}_u(Y)} \rho^{-n/2-\mathcal{N}_u(Y)} \|u\|_{L^2(B_{\rho/2}(Y))} 
	\leq 2^{{\mathcal N}_{u}(Y)}C d(X)^{\mathcal{N}_u(Y)} \rho^{-n/2-\mathcal{N}_u(Y)} \|u\|_{L^2(\Omega)} 
\end{align*} 
where $\rho = \op{dist}(X,\partial \Omega)$, $Y \in \Sigma_u$ is such that $d(X) = |X-Y|$ and $C = C(n,m) \in (0,\infty)$. 
\end{proof}

Now fix a non-zero, symmetric, Dirichlet energy minimizing function $u \in W^{1,2}(\Omega;\mathcal{A}_2(\mathbb{R}^m))$ and $Y \in \Sigma_u$.  Let 
\begin{equation*}
	u_{Y,\rho}(X) = \frac{u(Y+\rho X)}{\rho^{-n/2} \|u\|_{L^2(B_{\rho}(Y))}} 
\end{equation*}
for $0 < \rho < \op{dist}(X,\partial \Omega)$.  If $\r_{j}$ is a sequence of numbers with $\r_{j} \to 0^{+}$, it follows from (\ref{schauder_dm}) that there exists a symmetric two-valued function $\varphi \in W^{1, 2}_{\rm loc} \, ({\mathbb R}^{n}; {\mathcal A}_{2}({\mathbb R}^{m})) \cap C^{0,1/2} \, ({\mathbb R}^{n}; {\mathcal A}_{2}({\mathbb R}^{m}))$ such that after passing to a subsequence, 
$$u_{Y, \r_{j}} \to \varphi$$ 
uniformly on $B_{\s}(0)$ for every $\s > 0.$ We say $\varphi$ is a \emph{blow up of $u$ at $Y$}. It follows from 
(\ref{doubling_estimate2}) that $\varphi$ is non-zero. By [\cite{Almgren}, Theorem~2.13], $\varphi$
is Dirichlet energy minimizing on ${\mathbb R}^{n}$,  $N_{\varphi, 0}(\r) = \mathcal{N}_{\varphi}(0) = \mathcal{N}_u(Y) \geq 1/2$ for each $\r > 0$ and $\varphi$ is homogeneous of degree $\mathcal{N}_u(Y)$, i.e.\ $\varphi(\lambda X) = \{ \pm \lambda^{\mathcal{N}_u(Y)} \varphi_1(X) \}$ for every $X \in {\mathbb R}^{n}$ and $\lambda > 0$, where we use the notation $\varphi(X) = \{ \pm \varphi_1(X) \}$.  

Suppose $\varphi \in W^{1,2}(\mathbb{R}^n;\mathcal{A}_2(\mathbb{R}^m))$ is any homogeneous degree $\alpha$ ($\geq 1/2$), symmetric, Dirichlet energy minimizing function.  For each $Y \in {\mathbb R}^{n}$, it follows from  Lemma~\ref{monotonicity_cor} $\left(c\right)$ (applied with $u_{j} = u = \varphi$ and $Y_{j}= t_{j}Y$ for some sequence $t_{j} \to 0^{+}$) that $\mathcal{N}_{\varphi}(Y) \leq \mathcal{N}_{\varphi}(0) = \a$.  Let 
\begin{equation*}
	S(\varphi) = \{ Y \in {\mathbb R}^{n} : \mathcal{N}_{\varphi}(Y) = \mathcal{N}_{\varphi}(0) \}. 
\end{equation*}
It follows from [\cite{Almgren}, Theorem~2.14]  that $S(\varphi)$ is a linear subspace in $\mathbb{R}^n$ and that $\varphi(X) = \varphi(X + ty)$ for all $Y \in S(\varphi)$ and $t \in {\mathbb R}$.  Since the only Dirichlet minimizing symmetric two-valued functions in $W_{\rm loc}^{1,2}(\mathbb{R};\mathcal{A}_2(\mathbb{R}^m))$ are constant functions, $\dim S(\varphi) \leq n-2$.  For $j = 0,1,2,\ldots,n-2$, define $\Sigma^{(j)}_u$ to be the set of points $Y \in \Sigma_u$ such that $\dim S(\varphi) \leq j$ for every blow up $\varphi$ of $u$ at $Y$.  Observe that 
\begin{equation*}
	\Sigma_u = \Sigma^{(n-2)}_u \supseteq \Sigma^{(n-3)}_u \supseteq \cdots \supseteq \Sigma^{(1)}_u \supseteq \Sigma^{(0)}_u. 
\end{equation*}

\begin{lemma} \label{stratification_lemma}
Let $u \in W^{1,2}(\Omega;\mathcal{A}_2(\mathbb{R}^m))$ be a symmetric, Dirichlet energy minimizing function.  For each $j = 1,2,\ldots,n$, $\Sigma^{(j)}_u$ has Hausdorff dimension at most $j$.  For $\alpha > 0$, $\Sigma^{(0)}_u \cap \{ Y : \mathcal{N}_u(Y) = \alpha \}$ is discrete. 
\end{lemma}
\begin{proof}
The well-known argument, due to Almgren ([\cite{Almgren}, Theorem~2.26]) in the context of stationary integral varifolds, is based on upper semi-continuity of frequency and the fact that ${\mathcal N}_{\varphi}(0) = {\mathcal N}_{u}(Y)$ for any blow-up of $u$ at $Y$.  See the proof of Lemma 1 in Section 3.4 of~\cite{harmonicmaps} for a nice, concise presentation of the argument  in the context of energy minimizing maps. 
\end{proof}

Thus in order to establish, for a symmetric Dirichlet minimizing function 
$u \in W^{1,2}(\Omega;\mathcal{A}_2(\mathbb{R}^m)),$ $(n-2)$-rectifiability properties of $\Sigma_{u}$,  it suffices to consider the set $\Sigma_u \setminus \Sigma_u^{(n-3)}$.  Each point $Y \in \Sigma_u \setminus \Sigma_u^{(n-3)}$ has the property that there is a blow up $\varphi$ of $u$ at $Y$ such that, after an orthogonal change of coordinates on $\mathbb{R}^n$, $\varphi$ has the form 
\begin{equation*}
	\varphi(x_1,x_2,x_3,\ldots,x_n) = \{ \pm \op{Re}(c (x_1+ix_2)^{\alpha}) \} 
\end{equation*}
for some constant $c \in \mathbb{C}^m$ and some $\alpha = k/2$ where $k \geq 1$ is an integer.  

We shall later need the following lemma, which is the analogue of Lemma 2.4 in~\cite{SimonCTC}. 

\begin{lemma} \label{lemma2_4}
Let $0 < \alpha \leq K$ given.  There are functions $\delta : (0,1) \rightarrow (0,1)$ and $R : (0,1) \rightarrow (2,\infty)$ depending on $K$ such that if  $u \in W^{1,2}(\Omega;\mathcal{A}_2(\mathbb{R}^m))$ is a symmetric Dirichlet energy minimizing function on a domain 
$\Omega$ with $\overline{B_{R(\varepsilon)}(0)} \subset \Omega$ such that $0 \in \Sigma_u$ and 
\begin{equation*}
	N_{u,0}(R(\varepsilon)) - \alpha < \delta(\varepsilon), 
\end{equation*}
and if $X_1 \in \Sigma_u \cap B_1(0)$ with $\mathcal{N}_u(X_1) \geq \alpha$, then the following hold: 
\begin{enumerate}
\item[(i)] $0 < N_{u,X_1}(\rho) - \alpha < \varepsilon^2$ for $\rho < R(\varepsilon) - 1$.
\item[(ii)] For every $\rho \in (0,1]$, there is a Dirichlet energy minimizing symmetric  harmonic function $\varphi \in W^{1,2}(\mathbb{R}^n;\mathcal{A}_2(\mathbb{R}^m))$ (depending on $\rho$) that is homogeneous of degree $\mathcal{N}_{\varphi}(0)$ such that $|\mathcal{N}_{\varphi}(0) - \alpha| < \varepsilon^2$ and 
\begin{equation*}
	\int_{B_1(0)} \mathcal{G}(u_{X_1,\rho},\varphi)^2 < \varepsilon^2. 
\end{equation*}
\item[(iii)] For every $\rho \in (0,1]$, either there is a Dirichlet energy minimizing symmetric  harmonic function $\varphi \in W^{1,2}(\mathbb{R}^n;\mathcal{A}_2(\mathbb{R}^m))$ (depending on $\rho$) such that $\varphi$ is homogeneous degree $\mathcal{N}_{\varphi}(0)$,  $|\mathcal{N}_{\varphi}(0) - \alpha| \leq \varepsilon^2$, $\dim S(\varphi) = n-2$, and 
\begin{equation*}
	\int_{B_1(0)} \mathcal{G}(u_{X_1,\rho},\varphi)^2 < \varepsilon^2 
\end{equation*}
or 
\begin{equation*}
	\{ X \in \Sigma_{u_{X_1,\rho}} \cap \overline{B_1(0)} : \mathcal{N}_{u_{X_1,\rho}}(X) \geq \alpha \} \subset \{ X : \op{dist}(X,L) < \varepsilon \} 
\end{equation*}
for some $(n-3)$-dimensional subspace $L$. 
\end{enumerate} 
\end{lemma}
\begin{proof} Given (i), the assertion of 
(ii) is an easy consequence of the monotonicity of frequency functions and the compactness property of Dirichlet minimizing two-valued functions,   and the proof of (iii) is similar to the proof of Lemma 2.4(iii) in~\cite{SimonCTC}.  

To prove (i), first observe that $N_{u,X_1}(\rho) \geq \alpha$ by monotonicity of $N_{u,X_1}(\cdot)$.  Clearly 
\begin{equation} \label{lemma2_4_eqn1a}
	D_{u,X_1}(\rho) \leq \left( 1 + \frac{|X_1|}{\rho} \right)^{n-2} D_{u,0}(\rho+|X_1|) 
\end{equation}
for all $\rho \in (0,R(\varepsilon)-1)$.  By integration by parts and using $\int_{\partial B_{\sigma}(0)} uD_R u = \int_{B_{\sigma}(0)} |Du|^2$, 
\begin{align*}
	&H_{u,Y}(\rho) = \rho^{1-n} \int_{\partial B_{\rho}(Y)} |u|^2 = \rho^{-n} \int_{\partial B_{\rho}(Y)} |u|^2 X \cdot \frac{X}{R} = \rho^{-n} \int_{B_{\rho}(Y)} (n |u|^2 + 2R u D_R u) 
	\\&\hspace{3in}= n \rho^{-n} \int_{B_{\rho}(Y)} |u|^2 + 2\rho^{-n} \int_0^{\rho} \sigma \int_{B_{\sigma}(Y)} |Du|^2  
\end{align*}
for all $Y \in B_1(0)$ and $\rho \in (0,R(\varepsilon)-|Y|)$.  Since for $\r > |X_{1}|,$ 
\begin{eqnarray*}
&&\int_{0}^{\r} \s \int_{B_{\s}(X_{1})} |Du|^{2} \geq \int_{|X_{1}|}^{\r} \s \int_{B_{\s - |X_{1}|}(0)} |Du|^{2} = 
\int_{0}^{\r - |X_{1}|}(\s + |X_{1}|)\int_{B_{\s}(0)} |Du|^{2} d\s\nonumber\\ 
&&\hspace{4in}\geq \int_{0}^{\r - |X_{1}|} \s \int_{B_{\s}(0)}|Du|^{2},
\end{eqnarray*}
we deduce using also (\ref{doubling_estimate}) that  
\begin{equation} \label{lemma2_4_eqn1b}
	H_{u,X_1}(\rho) \geq \left( 1 - \frac{|X_1|}{\rho} \right)^n H_{u,0}(\rho-|X_1|) 
	\geq \left( 1 - \frac{|X_1|}{\rho} \right)^n \left( \frac{\rho-|X_1|}{\rho+|X_1|} \right)^{2\alpha+2\delta(\varepsilon)} H_{u,0}(\rho+|X_1|) 
\end{equation}
for all $\rho \in (|X_{1}|,R(\varepsilon)-1)$.  By (\ref{lemma2_4_eqn1a}) and (\ref{lemma2_4_eqn1b}), 
\begin{align*}
	N_{u,X_1}(\rho) &\leq \left( 1 + \frac{|X_1|}{\rho} \right)^{n-2} \left( 1 - \frac{|X_1|}{\rho} \right)^{-n} 
		\left( \frac{\rho-|X_1|}{\rho+|X_1|} \right)^{-2\alpha-2\delta(\varepsilon)} N_{u,0}(\rho+|X_1|) 
	\\&\leq \left( 1 + \frac{1}{\rho} \right)^{n-2} \left( 1 - \frac{1}{\rho} \right)^{-n} 
		\left( \frac{\rho-1}{\rho+1} \right)^{-2K-2} N_{u,0}(\rho+|X_1|) 
\end{align*}
for $\r \in (|X_{1}|, R(\e) - 1)$, and hence $N_{u,X_1}(\rho) - \alpha < \varepsilon^2$ for all $\rho \in [R(\varepsilon) - 2, R(\varepsilon)-1)$ provided $R(\varepsilon)$ is sufficiently large.  By  monotonicity of $N_{u,X_1}$, it then follows that $N_{u,X_1}(\rho) - \alpha < \varepsilon^2$ for all $\rho \in (0, R(\varepsilon)-1)$. 
\end{proof}

Let us now turn to the case of two-valued $C^{1, \mu}$ harmonic functions. 

Fix $\mu \in (0, 1)$. Given a symmetric two-valued harmonic function $u \in C^{1,\mu}(\Omega;\mathcal{A}_2(\mathbb{R}^m))$, we define $D_{u,Y}(\rho)$, $H_{u,Y}(\rho)$ and $N_{u,Y}(\rho)$ for $0 < \rho < \op{dist}(Y,\partial \Omega)$ by (\ref{frequency_defn}).  By~\cite[Remark 2.3(2)]{SW}, $u$ satisfies the identities in (\ref{monotonicity_identity1}) and (\ref{monotonicity_identity2}).  Consequently the frequency function $N_{u,Y}(\r)$ is monotone increasing, with $N'_{u,Y}(\rho)$ given by (\ref{monotonicity_eqn3}), for $\r \in I$ for any open interval $I$ with $H_{u, Y}(\r) > 0$ for $\r \in I.$ 

Lemma~\ref{monotonicity_cor} holds with the hypothesis that $u, u_j \in C^{1, \mu} \, (\Omega; {\mathcal A}_{2}({\mathbb R}^{m}))$ are harmonic with $u$ non-zero,  taken in place of the hypothesis that $u, u_j \in W^{1, 2} \, (\Omega; {\mathcal A}({\mathbb R}^{m}))$ are Dirichlet energy minimizing, and with the following additional changes: ${\mathcal K}_{u}$ taken in place of $\Sigma_{u}$ in parts (d) and (e), and the conclusion  
${\mathcal N}_{u}(Y) \geq 3/2$ in place of ${\mathcal N}_{u}(Y) \geq 1/2$ in (d).  The proof only requires obvious modifications.

Given a symmetric harmonic function $u \in C^{1,\mu}(\Omega;\mathcal{A}_2(\mathbb{R}^m)),$ 
$Y \in \mathcal{K}_u$, a sequence of numbers $\r_{j}$ with $\r_{j} \to 0^{+}$, the estimates 
(\ref{schauder_h}) guarantee the existence of $\varphi \in C^{1, \mu} \, (\Omega; {\mathcal A}_{2}({\mathbb R}^{m}))$ (a blow-up of $u$ at $Y$) such that after passing to a sub sequence, $u_{Y, \r_{j}} \to \varphi$ in 
$C^{1}(B_{\s}(0); {\mathcal A}_{2}({\mathbb R}^{m}))$ for every $\s > 0,$ where
 $u_{Y,\rho}(X) = u(Y+\rho X)/\rho^{-n/2} \|u\|_{L^2(B_{\rho}(Y))}.$ Furthermore, each such blow-up $\varphi$ is non-zero (by the Lemma~\ref{monotonicity_cor} (b)), symmetric,  and harmonic on $\mathbb{R}^n.$ (Indeed, it is clear, by the $C^{1}$ convergence $u_{Y, \r_{j}} \to \varphi$,  that 
 $\varphi$ is locally given by two harmonic functions in ${\mathbb R}^{n} \setminus {\mathcal K}_{\varphi},$ and by standard facts concerning single-valued harmonic functions and the fact that $\varphi$ is in $C^{1}$, 
 the same is true near every point in ${\mathcal K}_{\varphi} \setminus {\mathcal B}_{\varphi}$);    
 moreover, $\varphi$ is homogeneous of degree ${\mathcal N}_{\varphi}(0) = \mathcal{N}_u(Y) \geq 3/2$.  
 
If for each $j = 0,1,2,\ldots,n-2$, we define $\mathcal{K}^{(j)}_u$ to be the set of points $Y \in \mathcal{K}_u$ such that ${\rm dim} \, S(\varphi) \leq j$ for every blow up $\varphi$ of $u$ at $Y$, where $S(\varphi)$ is as above,  then we have 
\begin{equation*}
	\mathcal{K}_u = \mathcal{K}^{(n-2)}_u \supseteq \mathcal{K}^{(n-3)}_u \supseteq \cdots \supseteq \mathcal{K}^{(1)}_u \supseteq \mathcal{K}^{(0)}_u 
\end{equation*}
and that Lemma~\ref{stratification_lemma} with ${\mathcal K}_{u}^{j}$ in place of $\Sigma_{u}^{j}$ holds, so that in particular, ${\rm dim}_{\mathcal H} \, {\mathcal K}^{(j)} \leq j$ for $j=0, 1, 2, \ldots, n-2.$ 

Thus in order to establish $(n-2)$-rectifiability properties for  $\mathcal{B}_u$ and $\mathcal{K}_u$, it suffices to consider $\mathcal{K}_u \setminus \mathcal{K}_u^{(n-3)}$.  Note that for each $Y \in \mathcal{K}_u \setminus \mathcal{K}_u^{(n-3)}$, there is a blow up $\varphi$ of $u$ at $Y$ such that, after an orthogonal change of coordinates, $\varphi$ has the form 
\begin{equation*}
	\varphi(x_1,x_2,x_3,\ldots,x_n) = \{ \pm \op{Re}(c (x_1+ix_2)^{\alpha}) \} 
\end{equation*}
for some constant $c \in \mathbb{C}^m$ and some $\alpha = k/2$ where $k \geq 3$ is an integer.  

Lemma \ref{lemma2_4} (with the same proof) holds with $u, \varphi  \in C^{1, \mu} \, (\Omega; {\mathcal A}_{2}({\mathbb R}^{m}))$ and harmonic, in place of the  hypothesis that $u \in W^{1, 2} \, (\Omega; {\mathcal A}_{2}({\mathbb R}^{m}))$ and  Dirichlet energy minimizing.

Finally, note that in the case of $C^{1,\mu}$ harmonic two-valued functions, we can assume without loss of generality that  $m = 1,$ since for each harmonic function $u \in C^{1,\mu}(\Omega;\mathcal{A}_2(\mathbb{R}^m))$, the coordinate functions $u^{\k}$, defined  for $\kappa = 1,\ldots,m$ by $u^{\kappa}(X) = \{ u_1(X) \cdot e_{\kappa}, u_2(X) \cdot e_{\kappa} \}$ for $X \in \Omega$, where $u(X) = \{u_1(X),u_2(X)\}$ and $e_1,\ldots,e_m$ is the standard orthonormal basis for $\mathbb{R}^m$, satisfy $u^{\kappa} \in C^{1,\mu}(\Omega;\mathcal{A}_2(\mathbb{R})),$

\begin{equation} \label{singset_coord}
	\mathcal{B}_{u^{\kappa}} \subseteq \mathcal{B}_u \subseteq \mathcal{K}_u = \bigcap_{\lambda=1}^m \mathcal{K}_{u^{\lambda}} \subseteq \mathcal{K}_{u^{\kappa}}
\end{equation}
so $u^{\k}$ is given by two single-valued harmonic functions locally on $\Omega \setminus {\mathcal K}_{\k}$ and hence also on $\Omega \setminus \mathcal{B}_{u^{\kappa}}.$ Thus to establish rectifiability properties for $\mathcal{B}_u$ and $\mathcal{K}_u$, it suffices to show such properties for $\mathcal{K}_{u^{\kappa}}$  for each $\kappa = 1,\ldots,m$.  (We cannot suppose $m = 1$ by passing to coordinate functions in this manner for Dirichlet minimizing two-valued functions since the coordinate functions need not be Dirichlet minimizing; for example, consider $u : \mathbb{R}^2 \rightarrow \mathbb{C} \cong \mathbb{R}^2$ defined by $u(x_1,x_2) = \{ \pm (x_1+ix_2)^{1/2} \}$).

\section{Statement of main results}
\setcounter{equation}{0}
In what follows, we will fix a number $\alpha$ such that $\alpha = k/2$ for some integer $k \geq 1$ and fix a non-zero function $\varphi^{(0)} : \mathbb{R}^n \rightarrow \mathcal{A}_2(\mathbb{R}^m)$ given by 
\begin{equation} \label{varphi0}
	\varphi^{(0)}(x_1,x_2,\ldots,x_n) = \{ \pm \op{Re}(c^{(0)} (x_1+ix_2)^{\alpha}) \}
\end{equation}
for some $c^{(0)} \in \mathbb{C}^m \setminus \{0\}$.  In the context of Dirichlet energy minimizing two-valued functions, we will assume $c^{(0)}$ and $\alpha$ are chosen so that $\varphi^{(0)}$ is Dirichlet energy minimizing.  In the context of $C^{1,\mu}$ two-valued harmonic functions, we will assume that $\varphi^{(0)} \in C^{1, \mu} \, ({\mathbb R}^{n}; {\mathcal A}_{2}({\mathbb R}^{m}))$ for $\mu >0,$ 
and that $D\varphi^{(0)}(0) = \{0, 0\},$ which is equivalent to either the requirement that $\varphi^{(0)} \in C^{1,1/2}(\mathbb{R}^n;\mathcal{A}_2(\mathbb{R}^m))$ and $D\varphi^{(0)}(0) = \{0, 0\}$ or to the requirement that $\alpha \geq 3/2$.  

\begin{defn}
Let $\varepsilon \in (0,1)$.  For $\varphi^{(0)}$ as above and Dirichlet energy minimizing, let $\mathcal{F}^{\text{Dir}}_{\varepsilon}(\varphi^{(0)})$ denote the set of Dirichlet energy minimizing symmetric functions $u \in W^{1,2}(B_1(0);\mathcal{A}_2(\mathbb{R}^m))$ such that 
\begin{equation} \label{defn_fnspace}
	\int_{B_1(0)} \mathcal{G}(u(X),\varphi^{(0)}(X))^2 dX \leq \varepsilon^2. 
\end{equation}

For $\varphi^{(0)}$ as above with $\alpha \geq 3/2$, let $\mathcal{F}^{\text{Harm}}_{\varepsilon}(\varphi^{(0)})$ denote the space of symmetric harmonic functions $u \in C^{1,1/2}(B_1(0);\mathcal{A}_2(\mathbb{R}^m))$ such that  (\ref{defn_fnspace}) holds.  (Recall that if $u \in C^{1,\mu}(B_1(0);\mathcal{A}_2(\mathbb{R}^m))$ and harmonic for some $\mu \in (0,1)$, then $u \in C^{1,1/2}(B_1(0);\mathcal{A}_2(\mathbb{R}^m))$.)
\end{defn}

\begin{defn}
Let $\varepsilon \in (0,1)$.  Let $\Phi_{\varepsilon}(\varphi^{(0)})$ denote the space of symmetric functions $\varphi : B_1(0) \rightarrow \mathcal{A}_2(\mathbb{R}^m)$ of the form 
\begin{equation*}
	\varphi(x_1,x_2,\ldots,x_n) = \{ \pm {\rm Re} \, \left(c (x_1+ix_2)^{\alpha}\right) \}
\end{equation*}
for some $c \in \mathbb{C}^m$, where $\alpha$ is the degree of homogeneity of $\varphi^{(0)}$, such that 
\begin{equation*}
	\int_{B_1(0)} \mathcal{G}(\varphi(X),\varphi^{(0)}(X))^2 dX \leq \varepsilon^2. 
\end{equation*}
Note that we do not assume that $\varphi$ is Dirichlet minimizing or is in $C^1$.
\end{defn}

\begin{defn}
Let $\varepsilon \in (0,1)$.  Let $\widetilde{\Phi}_{\varepsilon}(\varphi^{(0)})$ denote the space of Dirichlet minimizing symmetric two-valued functions of the form $\varphi(e^A X)$ for some $\varphi \in \Phi_{\varepsilon}(\varphi^{(0)})$ and matrix $A \in \mathcal{S}$ with $|A| \leq \varepsilon$, where $\mathcal{S}$ denotes the space of $n \times n$ skew symmetric matrices $A = (a^{ij})_{i,j = 1,\ldots,n}$ such that $a^{ij} = 0$ if $i,j \leq 2$ and $a^{ij} = 0$ if $i,j \geq 3$. 
\end{defn}

\begin{rmk} \label{graph_rmk}
Given $\varphi \in \Phi_{\varepsilon}(\varphi^{(0)})$, we can regard the graph of $\left.\varphi\right|_{{\mathbb R}^{n} \setminus \{0\} \times {\mathbb R}^{n-2}}$ as an immersed $C^{\infty}$ submanifold in $(\mathbb{R}^n \setminus \{0\} \times \mathbb{R}^{n-2}) \times {\mathbb R}^{m}$.  In particular, if $\varphi^{(0)}$ is Dirichlet minimizing and $\varepsilon$ is sufficiently small, the graph of 
$\left.\varphi\right|_{{\mathbb R}^{n} \setminus \{0\} \times {\mathbb R}^{n-2}}$ is an embedded $C^{\infty}$ submanifold of $(\mathbb{R}^n \setminus \{0\} \times \mathbb{R}^{n-2}) \times {\mathbb R}^{m}$.  Thus given $\Omega \subseteq \mathbb{R}^n \setminus \{0\} \times {\mathbb R}^{n-2}$, we can consider single-valued functions $v \,: \,\op{graph} \varphi |_{\Omega} \rightarrow \mathbb{R}^m$.  In the case that $\alpha$ is a positive integer, the graph of $\varphi$ is the union of two embedded submanifolds, the graphs of $+\op{Re} \, (c (x_1+ix_2)^{\alpha})$ and the graph of $-\op{Re}(c (x_1+ix_2)^{\alpha})$.  Hence any function $v \, : \, \op{graph} \varphi |_{\Omega} \rightarrow \mathbb{R}^m$ is determined by a function on the graph of $\left.+\op{Re} \, (c (x_1+ix_2)^{\alpha})\right|_{\Omega}$ and another function on the graph of $\left.-\op{Re}(c (x_1+ix_2)^{\alpha})\right|_{\Omega}$.  In the case that $\alpha = k/2$ for some odd integer $k \geq 1$, the graph of $\varphi$ consists of a single branched submanifold parameterized by $(re^{i\theta},y,\op{Re}(c^{(0)} r^{\alpha} e^{i\theta}))$ for $r \geq 0$, $\theta \in [0,4\pi]$, and $y \in \mathbb{R}^{n-2}$, and any function 
$v \, : \, {\rm graph} \, \left.\varphi\right|_{\Omega} \to {\mathbb R}^{m}$ can be parametrized over $\Omega$ accordingly.
\end{rmk}

The following is our main  lemma, which is the analogue of Lemma 1 of~\cite{SimonCTC}. 

\begin{lemma} \label{lemma1}
Let $\vartheta \in (0,1/4)$ and let $\varphi^{(0)}$ be as in  (\ref{varphi0}). Suppose either  
(a) $\varphi^{(0)}$ is Dirichlet energy minimizing  or,   
(b) $\varphi^{(0)} \in C^{1, 1/2}(\mathbb{R}^n;\mathcal{A}_2(\mathbb{R}^m))$ and $D\varphi^{(0)}(0) = \{0, 0\}$.
Then there are $\delta_0, \varepsilon_0 \in (0,1/4)$ depending only on $n$, $m$, $\varphi^{(0)}$, $\alpha$, and $\vartheta$ such that if $\varphi \in \widetilde{\Phi}_{\varepsilon_0}(\varphi^{(0)}),$ and if either
\begin{itemize}
\item[(1)] $u \in \mathcal{F}_{\varepsilon_0}^{\text{Dir}}(\varphi^{(0)})$ where $\varphi^{(0)}$ is as in (a), or 
\item[(2)] $u \in \mathcal{F}^{\text{Harm}}_{\varepsilon_0}(\varphi^{(0)})$ where $\varphi^{(0)}$ is as in (b), 
\end{itemize}
then either 
\begin{enumerate}
\item[(i)] $B_{\delta_0}(0,y_0) \cap \{ X \in B_1(0) \cap \Sigma_u : \mathcal{N}_u(X) \geq \alpha \} = \emptyset$ for some $y_0 \in B^{n-2}_{1/2}(0)$ or 

\item[(ii)] there is a $\widetilde{\varphi} \in \widetilde{\Phi}_{\gamma \varepsilon_0}(\varphi^{(0)})$ such that 
\begin{equation*}
	\vartheta^{-n-2\alpha} \int_{B_{\vartheta}(0)} \mathcal{G}(u,\widetilde{\varphi})^2 \leq C\vartheta^{2\mu} \int_{B_1(0)} \mathcal{G}(u,\varphi)^2 
\end{equation*}
where $\gamma \in [1,\infty)$, $\mu \in (0,1)$, and $C \in (0,\infty)$ are constants depending only on $n$, $m$, $\varphi^{(0)}$, and $\alpha$. 
\end{enumerate}
\end{lemma}

Lemma \ref{lemma1}  leads to our main results 
Theorem~\ref{theorem2}, Theorem~\ref{theorem3} and Theorem~\ref{no-gaps} below.
In these theorems, we shall use the following notation and terminology: 

If $u = \{u_{1}, u_{2}\} \, : \, \Omega \to {\mathcal A}_{2}({\mathbb R}^{m}),$ 
then $v = u - u_{a} = \{\pm(u_{1} - u_{2})/2\}$, where $u_{a} = (u_{1}+u_{2})/2.$ 
By a ``cylindrical blow-up of $v$ at $X$'' we mean a blow-up $\varphi$ of $v$ at $X$  for which ${\rm dim} \, S(\varphi) = n-2.$ (See the discussion preceding Lemma~\ref{stratification_lemma}.) 

\begin{theorem} \label{theorem2}
Let $u \in W^{1,2}(B_1(0);\mathcal{A}_2(\mathbb{R}^m))$ be any non-zero Dirichlet minimizing  function.  Then $\Sigma_u$ is countably $(n-2)$-rectifiable.  For any compact set $K \subseteq B_1(0)$, 
$$\Sigma_u \cap K \cap \{ X : \mathcal{N}_v(X) = \alpha \; \mbox{and $v$ has a cylindrical blow-up at $X$}\} \neq \emptyset$$ 
only if $\alpha$ takes one of finitely many values $\in \{1/2, 1, 3/2, 2, \ldots\}$.     For each $\alpha \in \{1/2, 1, 3/2, 2, \ldots\}$,  there is an open set $V_{\alpha} \supset \{ X : \mathcal{N}_v(X)  = \alpha \; \mbox{and $v$ has a cylindrical blow-up at $X$}\}$ such that $V_{\alpha} \cap \{ X : \mathcal{N}_v(X) \geq \alpha \}$ has locally finite $\mathcal{H}^{n-2}$-measure.  
\end{theorem}

\begin{theorem}\label{theorem3}
Let $u \in C^{1,\mu}(B_1(0);\mathcal{A}_2(\mathbb{R}^m))$, where $\mu \in (0,1)$, be a non-zero harmonic  function $B_1(0)$.  Then $\mathcal{K}_u$ is countably $(n-2)$-rectifiable.  For  any compact set $K \subseteq B_1(0)$, 
$$\mathcal{K}_u \cap K \cap \{ X : \mathcal{N}_v(X) = \alpha \; \mbox{and $v$ has a cylindrical blow-up at $X$}\} \neq \emptyset$$ only if $\alpha$ takes one of finitely many values $\in \{3/2, 2, 5/2, \dots\}$.  For every $\alpha  \in \{3/2, 2, 5/2, \ldots\},$ there is an open set 
$V_{\alpha} \supset \{ X : \mathcal{N}_v(X) = \alpha \; \mbox{and $v$ has a cylindrical blow-up at $X$}\}$ such that $V_{\alpha} \{ X : \mathcal{N}_u(X) \geq \alpha \}$ has locally finite $\mathcal{H}^{n-2}$-measure.  
\end{theorem}

\begin{theorem}\label{no-gaps}
Suppose that either (i) $u \in W^{1, 2} \, (B_{1}(0); {\mathcal A}_{2}({\mathbb R}^{m}))$ is non-zero, Dirichlet energy minimizing in $B_{1}(0)$ and ${\mathcal N}_{v}(Z) = 1/2$ for some $Z \in B_{1}(0)$, or (ii) $u \in C^{1, \mu} \, (B_{1}(0); {\mathcal A}_{2}({\mathbb R}^{m}))$ for $\mu \in (0, 1),$ $u$ is non-zero, $u$ is harmonic in $B_{1}(0),$ and ${\mathcal N}_{v}(Z)  = 3/2$ for some $Z \in B_{1}(0)$. Then $Z \in {\mathcal B}_{u}$ and there exists $\r > 0$ such that ${\mathcal B}_{u} \cap B_{\r}(Z)$ is an $(n-2)$-dimensional $C^{1, \alpha}$  submanifold of $B_{\r}(Z)$ for some $\alpha \in (0, 1)$.  
\end{theorem}

Theorems A, B and C of the introduction are fairly direct consequences of Lemma~\ref{lemma1} and  Theorems~\ref{theorem2}, \ref{theorem3}, \ref{no-gaps}. Section~\ref{prooflemma_section} of the paper contains the proofs of all of these results.  

\section{A coarse representation lemma} \label{graphicalrep_section}
\setcounter{equation}{0}
Suppose that $\varphi \in \Phi_{\varepsilon_0}(\varphi^{(0)}),$ and that either $u \in \mathcal{F}^{\text{Dir}}_{\varepsilon_0}(\varphi^{(0)})$ or $u \in \mathcal{F}^{\text{Harm}}_{\varepsilon_0}(\varphi^{(0)})$ for $\varepsilon_{0} \in (0,1]$ small.  We will show, in Lemma~\ref{coarse_graph} below,  that outside a small tubular neighborhood of $\{0\} \times \mathbb{R}^{n-2}$, $\op{graph} \,  u$ is the graph of a single-valued function $v$ over $\op{graph} \varphi$ satisfying certain coarse estimates, giving in particular a way to pair the elements of $u(X)$ and $\varphi(X)$ for $X$ outside of that tubular neighborhood of $\{0\} \times \mathbb{R}^{n-2}$.  More specifically, given $\g \in (0, 1)$ (close to 1) and $\t \in (0, 1)$ (close to 0), we will construct for suitably small $\e_{0}$ depending on $\g$ and $\t$,  a single-valued $v \in C^2(\op{graph} \varphi |_U;\mathbb{R}^m)$ for some open set $U \subseteq B_1(0) \setminus \{0\} \times {\mathbb R}^{n-2}$ with $B_{\g}(0) \setminus \{(x, y) \, : \, |x| > \t\} \subset U$, where we regard $\op{graph} \varphi |_U$ as an immersed submanifold in $\mathbb{R}^{n+m}$, such that 
\begin{equation} \label{v_defn1}
	u(X) = \{ \varphi_1(X) + v(X,\varphi_1(X)), -\varphi_1(X) + v(X,-\varphi_1(X)) \}
\end{equation}
for $X \in U$, where we write $\varphi(X) = \{ \pm \varphi_1(X) \}.$ Note that we can associate such $v$ with a two-valued function 
\begin{equation*}
	\hat v(X) = v(X,\varphi(X)) \equiv \{ v(X,\varphi_1(X)), v(X,-\varphi_1(X)) \}. 
\end{equation*}

First we need the elementary facts given in Lemma~\ref{separation_lemma_dm} and Lemma~\ref{separation_lemma_h}. 

\begin{lemma} \label{separation_lemma_dm}
Let $\gamma \in (0,1)$.  Suppose $\psi \in W^{1,2}(B_1(0);\mathcal{A}_2(\mathbb{R}^m))$ is a Dirichlet minimizing symmetric two-valued function such that $\Sigma_{\psi} = \emptyset$.  There exists $\varepsilon = \varepsilon(n,m,\gamma,\psi) > 0$ such that the following holds true.  Let $u,v \in W^{1,2}(B_1(0);\mathcal{A}_2(\mathbb{R}^m))$ be Dirichlet minimizing symmetric two-valued functions such that 
\begin{equation} \label{separation_dm_eqn1} 
	\int_{B_1(0)} \mathcal{G}(u,\psi)^2 < \varepsilon^2, \quad  \int_{B_1(0)} \mathcal{G}(v,\psi)^2 < \varepsilon^2. 
\end{equation}
Then $u = \{+u_1,-u_1\}$ and $v = \{+v_1,-v_1\}$ on $B_{\gamma}(0)$ for some single-valued harmonic functions $u_1,v_1 \in C^{\infty}(B_{\gamma}(0);\mathbb{R}^m)$ such that 
\begin{equation*} 
	\|u_1 - v_1\|_{C^3(B_{\gamma}(0))} \leq C \left( \int_{B_1(0)} \mathcal{G}(u,v)^2 \right)^{1/2}
\end{equation*}
for some $C = C(n,m,\gamma) \in (0,\infty)$. 
\end{lemma}
\begin{proof}
It follows from (\ref{schauder_dm}) that if $u$ and $v$ are close to $\psi$ in $L^2(B_1(0))$, then $u$ and $v$ are uniformly close to $\psi$ on $B_{(1+\gamma)/2}(0)$; that is, for every $\delta > 0$ there is a $\varepsilon_0 = \varepsilon_0(n,m,\gamma,\psi,\delta) > 0$ such that if $u \in W^{1/2}(B_1(0);\mathcal{A}_2(\mathbb{R}^m))$ is a Dirichlet minimizing symmetric functions with  
\begin{equation*}
	\int_{B_1(0)} \mathcal{G}(u,\psi)^2 < \varepsilon_0^2
\end{equation*}
then $u = \{+u_1,-u_1\}$ on $B_{(1+\gamma)/2}(0)$ for some harmonic single-valued function $u_1 \in C^{\infty}(B_{(1+\gamma)/2}(0);\mathbb{R}^m)$ such that 
\begin{equation*}
	\sup_{B_{(1+\gamma)/2}(0)} |u_1-\psi_1| < \delta, 
\end{equation*}
where $\psi_1 \in C^{\infty}(B_{(1+\gamma)/2}(0);\mathbb{R}^m)$ is a single-valued harmonic function such that $\psi = \{+\psi_1,-\psi_1\}$ on $B_1(0)$.  Now take $\delta = (1/3) \inf_{B_{(1+\gamma)/2}(0)} |\psi_1|$ and $\varepsilon = \varepsilon_0(n,m,\gamma,\psi,\delta)$.  Suppose $u,v \in W^{1,2}(B_1(0);\mathcal{A}_2(\mathbb{R}^m))$ are Dirichlet minimizing symmetric two-valued function such that (\ref{separation_dm_eqn1}) holds.  Observe that $u = \{+u_1,-u_1\}$ and $v = \{+v_1,-v_1\}$ on $B_{(1+\gamma)/2}(0)$ for some harmonic single-valued functions $u_1,v_1 \in C^{\infty}(B_{(1+\gamma)/2}(0);\mathbb{R}^m)$ such that $\sup_{B_{(1+\gamma)/2}(0)} |u_1-\psi_1| < \delta$ and $\sup_{B_{(1+\gamma)/2}(0)} |v_1-\psi_1| < \delta$.  By our choice of $\delta$, $\mathcal{G}(u_1,v) = \sqrt{2} |u_k-v|$ in $B_{(1+\gamma)/2}(0)$.  By standard elliptic estimates for single-valued harmonic functions 
\begin{equation} \label{separation_dm_eqn2}
	\|u_1-v_1\|_{C^3(B_{\gamma}(0))} \leq C \|u_1-v_1\|_{L^2(B_{(1+\gamma)/2}(0))} = C \left( \int_{B_{(1+\gamma)/2}(0)} \mathcal{G}(u,v)^2 \right)^2.  
\end{equation}
for constants $C = C(n,m,\gamma) \in (0,\infty)$. 
\end{proof}

\begin{lemma} \label{separation_lemma_h}
Let $\mu \in (0,1)$ and $\gamma \in (0,1)$.  Suppose $\psi \in C^{1,\mu}(B_1(0);\mathcal{A}_2(\mathbb{R}^m))$ is a symmetric two-valued harmonic function on $B_1(0)$ such that whenever $Y \in {\mathcal Z}_{\psi},$ $D\psi(Y) = \{\pm A\}$ for some rank one $m \times n$ matrix $A$.  
There exists $\varepsilon = \varepsilon(n,m,\gamma,\psi) > 0$ such that if $u,v \in C^{1,\mu}(B_1(0);\mathcal{A}_2(\mathbb{R}^m))$ are symmetric two-valued harmonic functions on $B_{1}(0)$ with
\begin{equation} \label{separation_h_eqn1} 
	\int_{B_1(0)} \mathcal{G}(u,\psi)^2 < \varepsilon^2, \quad  \int_{B_1(0)} \mathcal{G}(v,\psi)^2 < \varepsilon^2, 
\end{equation}
then $u = \{+u_1,-u_1\}$ and $v = \{+v_1,-v_1\}$ on $B_{\gamma}(0)$ for some single-valued harmonic functions $u_1,v_1 \in C^{\infty}(B_{\gamma}(0);\mathbb{R}^m)$ satisfying
\begin{equation} \label{separation_h_eqn2} 
	\|u_1 - v_1\|_{C^3(B_{\gamma}(0))} \leq C \left( \int_{B_1(0)} \mathcal{G}(u,v)^2 \right)^{1/2}
\end{equation}
where $C = C(n,m,\gamma,\psi) \in (0,\infty)$. 
\end{lemma}

\begin{proof} 
Observe that if $\mathcal{Z}_{\psi} = \emptyset$ then Lemma \ref{separation_lemma_h} follows from (\ref{schauder_dm}) by an argument similar to the proof of Lemma \ref{separation_lemma_dm}.  We will consider the special case where $\psi(X) = \{ \pm a x_1 e_1 \}$ on $B_1(0)$ where $a \geq 0$ and $e_1,e_2,\ldots,e_m$ denote the standard basis for $\mathbb{R}^m$.  Then Lemma \ref{separation_lemma_h} will hold true for general $\psi$ by approximation of $\psi$ by affine two-valued functions and a standard covering argument.  

It follows from the estimates (\ref{schauder_h}) that for each $\delta \in (0,a/4),$ there is $\varepsilon_0 = \varepsilon_0(n,m,\gamma,\psi,\delta) > 0$ such that if $u \in C^{1,\mu}(B_1(0);\mathcal{A}_2(\mathbb{R}^m))$ is harmonic on $B_1(0)$ with  
\begin{equation*} 
	\int_{B_1(0)} \mathcal{G}(u,\psi)^2 < \varepsilon_0^2, 
\end{equation*}
then $u = \{+u_1,-u_1\}$ on $B_{(1+\gamma)/2}(0)$ for some harmonic single-valued function $u_1 \in C^{\infty}(B_{(1+\gamma)/2}(0);\mathbb{R}^m)$ such that 
\begin{equation*} 
	\|u_1-\psi_1\|_{C^1(B_{(1+\gamma)/2}(0))} < \delta, 
\end{equation*}
where $\psi_1 \in C^{\infty}(B_{(1+\gamma)/2}(0);\mathbb{R}^m)$ is a single-valued harmonic function such that $\psi = \{+\psi_1,-\psi_1\}$ on $B_1(0)$.  

Now let $\delta > 0$ to be later determined.  Suppose $u,v \in C^{1,\mu}(B_1(0);\mathbb{R}^m)$ symmetric two-valued functions such that $u$ is harmonic on $B_1(0) \setminus \mathcal{B}_u$, $v$ is harmonic on $B_1(0) \setminus \mathcal{B}_v$, and (\ref{separation_h_eqn1}) holds true for $\varepsilon \leq \varepsilon_0$.  Write $u = \{+u_1,-u_1\}$ and $v = \{v_1,-v_1\}$ on $B_{(3+\gamma)/4}(0)$ for $C^3$ single-valued harmonic functions $u_1$ and $v_1$ such that 
\begin{equation} \label{separation_h_eqn3} 
	\|u_1-\psi_1\|_{C^1(B_{(1+\gamma)/2}(0))} < \delta, \quad \|v_1-\psi_1\|_{C^1(B_{(1+\gamma)/2}(0))} < \delta. 
\end{equation}
Let $X = (x_1,x_2,\ldots,x_n) = (x_1,z)$ where $z = (x_2,x_3,\ldots,x_n)$.  Fix $z \in B^{n-1}_{(5+3\gamma)/8}(0)$.  We know that $|u_1(x_1,z) - v_1(x_1,z)| = |u_1(x_1,z) + v_1(x_1,z)|$ precisely when $u_1(x_1,z) \cdot v_1(x_1,z) = 0$.  Since $u_1$ and $v_1$ are close to $a x_1 e_1$ in $C^2$ and $a^2 x_1^2$ is a strictly convex function of $x_1$, $u_1(x_1,z) \cdot v_1(x_1,z)$ is also a strictly convex function of $x_1$.  Hence $u_1(x_1,z) \cdot v_1(x_1,z) = 0$ for at most two $x_1$ with $(x_1,z) \in B_{(5+3\gamma)/8}(0)$.  If $u_1(x_1,z) \cdot v_1(x_1,z) = 0$ for either zero or one $x_1$ with $(x_1,z) \in B_{(5+3\gamma)/8}(0)$, then 
\begin{equation} \label{separation_h_eqn4} 
	\mathcal{G}(u(x_1,z),v(x_1,z)) = \sqrt{2} |u_1(x_1,z) - v_1(x_1,z)| \text{ for all } x_1 \text{ with } (x_1,z) \in B_{(5+3\gamma)/8}(0). 
\end{equation}
Suppose $u_1(x_1,z) \cdot v_1(x_1,z) = 0$ for $x_1 = \zeta,\xi$ with $(\zeta,z), (\xi,z) \in B_{(5+3\gamma)/8}(0)$ and assume $\xi < \zeta$.  Note that $\xi-\zeta \leq 2\sqrt{2} \delta/|a|$.  Observe that 
\begin{align} \label{separation_h_eqn5} 
	&\mathcal{G}(u(x_1,z),v(x_1,z)) = \sqrt{2} |u_1(x_1,z) - v_1(x_1,z)| \text{ if } x_1 \not\in [\xi,\zeta], \nonumber \\ 
	&\mathcal{G}(u(x_1,z),v(x_1,z)) = \sqrt{2} |u_1(x_1,z) + v_1(x_1,z)| \text{ if } x_1 \in [\xi,\zeta]. 
\end{align}
Write 
\begin{align} \label{separation_h_eqn6} 
	u_1(x_1,z) &= u_1(\xi,z) + a e_1 (x_1-\xi) + \int_{\xi}^{x_1} (D_1 u_1(t,z) - ae_1) dt, \nonumber \\
	v_1(x_1,z) &= v_1(\xi,z) + a e_1 (x_1-\xi) + \int_{\xi}^{x_1} (D_1 v_1(t,z) - ae_1) dt, 
\end{align}
By (\ref{separation_h_eqn3}), (\ref{separation_h_eqn6}), and the fact that $u_1(x_1,z) \cdot v_1(x_1,z)$ for $x_1 = \xi,\zeta$, 
\begin{align*} 
	|(u_1(\xi,z) + v_1(\xi,z)) \cdot e_1 + a(\zeta-\xi)| 
	&\leq (\delta/a) (|u_1(\xi,z)| + |v_1(\xi,z)|) + (2\delta + \delta^2/a) (\zeta-\xi), 
	\\&\leq (2\delta/a) |u_1(\xi,z) + v_1(\xi,z)| + (2\delta + \delta^2/a) (\zeta-\xi), 
\end{align*}
which using the triangle inequality implies 
\begin{equation} \label{separation_h_eqn7} 
	|(u_1(\xi,z) + v_1(\xi,z)) \cdot e_1 + a(\zeta-\xi)| 
	\leq (4\delta/a) |u_1(\xi,z)^{\perp} + v_1(\xi,z)^{\perp}| + 10 \delta (\zeta-\xi), 
\end{equation}
provided $\delta \leq a/4$, where we let $\eta^{\perp} = \eta - (\eta \cdot e_1) e_1$ for any $\eta \in \mathbb{R}^m$.  By (\ref{separation_h_eqn3}), (\ref{separation_h_eqn6}), and (\ref{separation_h_eqn7}), 
\begin{align} \label{separation_h_eqn8} 
	&\int_{\xi}^{\zeta} |u_1(x_1,z) + v_1(x_1,z)|^2 dx_1 \nonumber \\
	&\geq \frac{1}{2} \int_{\xi}^{\zeta} |u_1(\xi,z) + v_1(\xi,z) + 2ae_1 (x_1-\xi)|^2 dx_1 - 4 \delta^2 \int_{\xi}^{\zeta} (x_1-\xi)^2 dx_1 \nonumber \\
	&= \frac{1}{2} \int_{\xi}^{\zeta} |(u_1(\xi,z) + v_1(\xi,z)) \cdot e_1 + 2a (x_1-\xi)|^2 dx_1 
		+ \frac{1}{2} |u_1(\xi,z)^{\perp} + v_1(\xi,z)^{\perp}|^2 (\zeta-\xi)
		- \frac{4}{3} \delta^2 (\zeta-\xi)^3 \nonumber \\
	&\geq \frac{1}{4} \int_{\xi}^{\zeta} a^2 (2x_1-\xi-\zeta)^2 dx_1 
		+ \frac{1}{4} |u_1(\xi,z)^{\perp} + v_1(\xi,z)^{\perp}|^2 (\zeta-\xi)
		- 102 \delta^2 (\zeta-\xi)^3 \nonumber \\
	&= \frac{1}{12} a^2 (\zeta-\xi)^3 + \frac{1}{4} |u_1(\xi,z)^{\perp} + v_1(\xi,z)^{\perp}|^2 (\zeta-\xi)
		- 102 \delta^2 (\zeta-\xi)^3 \nonumber \\
	&\geq \frac{1}{24} a^2 (\zeta-\xi)^3 + \frac{1}{4} |u_1(\xi,z)^{\perp} + v_1(\xi,z)^{\perp}|^2 (\zeta-\xi)
\end{align}
provided $\delta \leq a/50$ and 
\begin{align} \label{separation_h_eqn9} 
	&\int_{\zeta}^{\xi} |u_1(x_1,z) - v_1(x_1,z)|^2 dx_1 \nonumber \\
	&\leq 2 |u_1(\xi,z) - v_1(\xi,z)|^2 (\zeta-\xi) + 8 \delta^2 \int_{\xi}^{\zeta} (x_1-\xi)^2 dx_1 \nonumber \\
	&= 2 |u_1(\xi,z) + v_1(\xi,z)|^2 (\zeta-\xi) + \frac{8}{3} \delta^2 (\zeta-\xi)^3 \nonumber \\
	&= 2 |(u_1(\xi,z) + v_1(\xi,z)) \cdot e_1|^2 (\zeta-\xi) + 2 |u_1(\xi,z)^{\perp} + v_1(\xi,z)^{\perp}|^2 (\zeta-\xi) 
		+ \frac{8}{3} \delta^2 (\zeta-\xi)^3 \nonumber \\
	&\leq 4 a^2 (\zeta-\xi)^3 + 4 |u_1(\xi,z)^{\perp} + v_1(\xi,z)^{\perp}|^2 (\zeta-\xi) 
		+ 803 \delta^2 (\zeta-\xi)^3 \nonumber \\
	&\leq 8 a^2 (\zeta-\xi)^3 + 4 |u_1(\xi,z)^{\perp} + v_1(\xi,z)^{\perp}|^2 (\zeta-\xi) 
\end{align}
provided $\delta \leq a/15$.  By combining (\ref{separation_h_eqn8}) and (\ref{separation_h_eqn9}) and taking $\delta$ to be sufficiently small depending on $a$, 
\begin{equation} \label{separation_h_eqn10} 
	\int_{\zeta}^{\xi} |u_1(x_1,z) - v_1(x_1,z)|^2 dx_1 
	\leq 192 \int_{\zeta}^{\xi} |u_1(x_1,z) + v_1(x_1,z)|^2 dx_1. 
\end{equation}
By (\ref{separation_h_eqn5}), 
\begin{equation} \label{separation_h_eqn11} 
	\int_{-\rho(z)}^{\rho(z)} |u_1(x_1,z) - v_1(x_1,z)|^2 dx_1 
	\leq 192 \int_{-\rho(z) - 2\sqrt{2} \delta/a}^{\rho(z) + 2\sqrt{2} \delta/a} 
		\mathcal{G}(u_1(x_1,z), v_1(x_1,z))^2 dx_1 
\end{equation}
for all $z \in B^{n-1}_{(5+3\gamma)/8}(0)$, where $\rho(z) = \sqrt{(5+3\gamma)^2/64 - |z|^2}$.  By integrating over $z \in B^{n-1}_{(5+3\gamma)/8}(0)$ 
\begin{equation*} 
	\int_{\zeta}^{\xi} |u_1(x_1,z) - v_1(x_1,z)|^2 dx_1 
	\leq 192 \int_{\zeta}^{\xi} \mathcal{G}(u_1(x_1,z), v_1(x_1,z))^2 dx_1 
\end{equation*}
provided $\delta \leq a/50$.  Arguing like in (\ref{separation_dm_eqn2}) using standard elliptic estimates for single-valued harmonic functions, we obtain (\ref{separation_h_eqn2}). 
\end{proof}

The following lemma is the analogue of Lemma 2.6 of~\cite{SimonCTC}. 

\begin{lemma} \label{coarse_graph}  Let $\varphi^{(0)}$ be as in  (\ref{varphi0}), with degree of homogeneity $\alpha$. Suppose that either  
(a) $\varphi^{(0)} \in W^{1,2}(\mathbb{R}^n;\mathcal{A}_2(\mathbb{R}^m))$ is 
Dirichlet energy minimizing  or,   
(b) $\varphi^{(0)} \in C^{1,1/2}(\mathbb{R}^n;\mathcal{A}_2(\mathbb{R}^m))$ and $D\varphi^{(0)}(0) = \{0, 0\}$.
Let $\gamma, \beta, \tau \in (0,1)$ be arbitrary with $\tau \leq (1-\gamma)/10$.  There is an $\varepsilon_0 = \varepsilon_0(n,m,\varphi^{(0)},\gamma,\beta,\tau) \in (0,1]$ such that if $\varphi \in \Phi_{\varepsilon_0}(\varphi^{(0)})$ and either
\begin{itemize}
\item[(i)] (a) holds and $u \in \mathcal{F}_{\varepsilon_0}^{\text{Dir}}(\varphi^{(0)})$, or\\
\item[(ii)] (b) holds and $u \in \mathcal{F}^{\text{Harm}}_{\varepsilon_0}(\varphi^{(0)}),$
\end{itemize}
 then there is an open set $U \subset B_1(0) \setminus \{0\} \times \mathbb{R}^{n-2}$ such that 
\begin{align*}
	&(x,y) \in U \Rightarrow (\tilde x,y) \in U \text{ whenever } |x| = |\tilde x|, 
	\\&\{ (x,y) \in B_{\gamma}(0) : |x| > \tau \} \subset U, 
\end{align*}
and there is a single-valued function $v \in C^2(\op{graph} \varphi |_U,\mathbb{R}^m)$ such that (\ref{v_defn1}) holds, $\hat v(X) = \{ v(X,\varphi_1(X)),$ $v(X,-\varphi_1(X)) \}$ is a two-valued harmonic function on $U$, and 
\begin{align*}
	&\sup_{B_{\gamma}(0)} r^{-\alpha} |\hat v| + \sup_{B_{\gamma}(0)} r^{1-\alpha} |D\hat v| \leq \beta, 
	\\&\int_U (|\hat v|^2 + r^2 |D\hat v|^2) 
	+ \int_{B_{\gamma}(0) \setminus U} (|u|^2 + r^2 |Du|^2) \leq C \int_{B_1(0)} \mathcal{G}(u,\varphi)^2 
\end{align*}
where $r(x,y) = |x|$ for $(x,y) \in B_1(0)$ and $C = C(n,m,\varphi^{(0)},\alpha,\gamma,\beta) \in (0,\infty)$ is a constant independent of $\tau$.
\end{lemma}

\begin{proof} 
First consider the case of Dirichlet energy minimizers (i.e.\ hypotheses (a) and $u \in {\mathcal F}_{\varepsilon_{0}}^{Dir}(\varphi^{(0)}))$).  For each $\zeta \in \mathbb{R}^{n-2}$ and $\rho \in (0,1]$, let 
\begin{equation*}
	A_{\rho,\kappa}(\zeta) = \{ (x,y) \in \mathbb{R}^2 \times \mathbb{R}^{n-2} : (|x| - \rho)^2 + |y - \zeta|^2 < \kappa^2 (1-\gamma)^2 \rho^2/4 \} .
\end{equation*}
By Lemma \ref{separation_lemma_dm}, there is a $\delta = \delta(n,m,\varphi^{(0)},\beta) > 0$ such that if $\rho^2 + |\zeta|^2 < \gamma^2$ and $\int_{A_{\rho,3/4}(\zeta)} \mathcal{G}(u,\varphi)^2 < \delta \rho^{n+2\alpha}$ (or equivalently, $\int_{A_{1,3/4}(0)} \mathcal{G}(\tilde u,\varphi)^2 < \delta$ where $\tilde u(x,y) = \rho^{-\alpha} u(\rho x, \zeta + \rho y)$) then there exists a single-valued function $v_{\rho,\zeta} \in C^2(\op{graph} \varphi |_{A_{\rho,1/2}(\zeta)},\mathbb{R}^m)$ such that 
\begin{equation*}
	u(X) = \{ \varphi_1(X) + v_{\rho,\zeta}(X,\varphi_1(X)), -\varphi_1(X) + v_{\rho,\zeta}(X,-\varphi_1(X)) \}
\end{equation*}
for all $X \in A_{\rho,1/2}(\zeta)$, where $\varphi(X) = \{ \pm \varphi_1(X) \}$, and $\hat v_{\r, \, \z}(X) \equiv \{ v_{\rho,\zeta}(X,\varphi_1(X)), v_{\rho,\zeta}(X,-\varphi_1(X)) \}$ satisfies 
\begin{equation} \label{coarse_graph_eqn1}
	\rho^{-\alpha} \sup_{A_{\rho,1/2}(\zeta)} |\hat v_{\rho,\zeta}| + \rho^{1-\alpha} \sup_{A_{\rho,1/2}(\zeta)} |D\hat v_{\rho,\zeta}| \leq 2^{-\alpha} \beta.
\end{equation}
Let $U$ be the union of the $A_{\rho,1/2}(\zeta)$ such that $\rho^2 + |\zeta|^2 < \gamma^2$ and $\int_{A_{\rho,3/4}(\zeta)} \mathcal{G}(u,\varphi)^2 < \delta \rho^{n+2\alpha}$.  We then obtain a well-defined single-valued function $v \in C^2(\op{graph} \varphi |_U,\mathbb{R}^m)$ by requiring that $v = v_{\rho,\zeta}$ on ${\rm graph} \, \varphi|_{A_{\r, \, 1/2}(\z)}$ whenever $\r^{2} + |\z|^{2} < \g^{2}$ and $\int_{A_{\rho,3/4}(\zeta)} \mathcal{G}(u,\varphi)^2 < \delta \rho^{n+2\alpha}$.  For $X \in U$, let $\hat v(X) = \{v(X, \varphi_{1}(X)), v(X, - \varphi_{1}(X))\}.$ Since $\varphi \in \Phi_{\varepsilon_0}(\varphi^{(0)})$ and $u \in \mathcal{F}^{\text{Dir}}_{\varepsilon_0}(\varphi^{(0)})$, it follows from Lemma \ref{separation_lemma_dm} that 
\begin{equation*}
	\{ (x,y) \in B_{\gamma}(0) : |x| > \tau \} \subset U 
\end{equation*}
provided $\varepsilon_0 = \varepsilon(n,m,\varphi^{(0)},\gamma,\beta,\tau) \in (0,1]$ is sufficiently small.  By (\ref{coarse_graph_eqn1}), 
\begin{equation} \label{coarse_graph_eqn2}
	\sup_{B_{\gamma}(0) \cap U} r^{-\alpha} |\hat v(x,y)| + \sup_{B_{\gamma}(0) \cap U} r^{1-\alpha} |D\hat v(x,y)| \leq \beta
\end{equation}
where $r(x,y) = |x|.$ 

For $(\xi,\zeta) \in B_{\gamma} \cap \partial U$ then 
\begin{equation} \label{coarse_graph_eqn3}
	\int_{A_{|\xi|,3/4}(\zeta)} \mathcal{G}(u,\varphi)^2 \geq \delta |\xi|^{n+2\alpha};
\end{equation}
otherwise $A_{|\xi|,1/2}(\zeta)$ is an open neighborhood of $(\xi,\zeta)$ with $A_{|\xi|,1/2}(\zeta) \subseteq U$.  By (\ref{coarse_graph_eqn2}) and (\ref{coarse_graph_eqn3}), 
\begin{equation*}
	\int_{B_{10|\xi|}(0,\zeta) \cap B_{\gamma} \cap U} (|\hat v|^2 + r^2 |D\hat v|^2) 
	\leq C \beta^2 \delta^{-1} \int_{A_{|\xi|,3/4}(\zeta)} \mathcal{G}(u,\varphi)^2
\end{equation*}
for some constant $C = C(n,\alpha) \in (0,\infty)$.  Since 
\begin{equation*}
	\{ (x,y) \in B_{\gamma}(0) \cap U : \op{dist}((x,y), B_{\gamma} \cap \partial U) \leq |x|/2 \} \subset 
	\bigcup_{(\xi,\zeta) \in B_{\gamma} \cap \partial U} B_{2|\xi|}(0,\zeta)
\end{equation*}
and 
\begin{equation*}
	B_{2|\xi|}(0,\zeta) \cap B_{2|\xi'|}(0,\zeta') = \emptyset \hspace{3mm} \Rightarrow \hspace{3mm} 
	A_{|\xi|,3/4}(\zeta) \cap A_{|\xi'|,3/4}(\zeta') = \emptyset, 
\end{equation*}
by the 5-times covering lemma, 
\begin{equation} \label{coarse_graph_eqn4}
	\int_{\{ (x,y) \in B_{\gamma}(0) \cap U : \op{dist}((x,y), B_{\gamma} \cap \partial U) \leq |x|/2 \}} (|\hat v|^2 + r^2 |D\hat v|^2) 
	\leq C \int_{B_1(0)} \mathcal{G}(u,\varphi)^2 
\end{equation}
for $C = C(n,m,\varphi^{(0)},\alpha,\beta) \in (0,\infty)$.  

Define $d : U \rightarrow [0,\infty)$ by $d(x,y) = \op{dist}((x,y), \partial U \cap B_{\gamma}(0))$ and note that $d$ is Lipschitz with Lipschitz constant $1$.  Let $\psi : [0,\infty) \rightarrow [0,1]$ be a smooth function such that $\psi(t) = 0$ for $t \leq 1/4$, $\psi(t) = 1$ for $t \geq 1/2$, and $0 \leq \psi'(t) \leq 6$ for all $t$.  Let $\eta : B_1(0) \rightarrow [0,1]$ be a smooth function such that $\eta = 1$ on $B_{\gamma}(0)$, $\eta = 0$ on $B_1(0) \setminus B_{(1+\gamma)/2}(0)$ and $|D\eta| \leq 3/(1-\gamma)$.  Since $\Delta \, \hat v = 0$ in $U$ (meaning that locally near every point in $U$, $\hat v$ is given by two harmonic functions with zero average), we may, with the help of a suitable partition of unity, multiply this by $\hat v r^2 \psi(d/|x|)^2 \eta^2$ and integrate by parts to obtain
\begin{align*}
	\int_U |D\hat v|^2 r^2 \psi(d/r)^2 \eta^2 
	= {}& 2 \int_U \hat v^{\kappa} D_i \hat v^{\kappa} \left( x_i \psi(d/r)^2 \eta^2 
		+ r^2 \psi(d/r) \psi'(d/r) \eta^2 \left( \frac{D_i d}{r} - \frac{d x_i}{r^3} \right) \right) \\
	& + 2 \int_U \hat v^{\kappa} D_i \hat v^{\kappa} r^2 \psi(d/r)^2 \eta D_i \eta. 
\end{align*}
Using Cauchy's inequality and the definition of $\psi$, we deduce from this that 
\begin{equation*}
	\int_U |D\hat v|^2 r^2 \psi(d/r)^2 \eta^2 
	\leq \int_U |\hat v|^2 (5 \psi(d/r)^2 \eta^2 + 9 \psi'(d/r)^2 \eta^2 + 4r^2 \psi(d/r)^2 |D\eta|^2),
\end{equation*}
which, by the definitions of $\psi$ and $\eta$, yields
\begin{equation} \label{coarse_graph_eqn5}
	\int_{\{ (x,y) \in U \cap B_{\gamma}(0) : d(x,y) \geq |x|/2 \}} r^2 |D\hat v|^2  
	\leq C \int_{B_1(0) \cap U} |\hat v|^2   
\end{equation}
for some constant $C = C(\gamma) \in (0,\infty)$. Combining (\ref{coarse_graph_eqn4}) and (\ref{coarse_graph_eqn5}), 
\begin{equation*}
	\int_{U \cap B_{\gamma}} (|\hat v|^2 + r^2 |D\hat v|^2) \leq C \int_{B_1(0)} \mathcal{G}(u,\varphi)^2 
\end{equation*}
for some constant $C = C(n,m,\varphi^{(0)},\alpha,\gamma,\beta) \in (0,\infty)$. 

Observe that if $\rho^2 + |\zeta|^2 < \gamma^2$ and $\int_{A_{\rho,3/4}(\zeta)} \mathcal{G}(u,\varphi)^2 \geq \delta \rho^{n+2\alpha}$, then since by (\ref{monotonicity_identity1})
\begin{equation*} 
	\int_{A_{\rho,1/2}(\zeta)} (|u|^2 + r^2 |Du|^2) \leq C \int_{A_{\rho,3/4}(\zeta)} |u|^2, 
\end{equation*}
it follows from the triangle inequality and the homogeneity of $\varphi$, 
\begin{align*} 
	&\hspace{-.7in}\int_{A_{\rho,1/2}(\zeta)} (|u|^2 + r^2 |Du|^2) 
	\leq C \int_{A_{\rho,3/4}(\zeta)} \mathcal{G}(u,\varphi)^2 + C \int_{A_{\rho,3/4}(\zeta)} |\varphi|^2 
	\\&\leq C \int_{A_{\rho,3/4}(\zeta)} \mathcal{G}(u,\varphi)^2 + C \rho^{n+2\alpha} \leq C (1+\delta^{-1}) \int_{A_{\rho,3/4}(\zeta)} \mathcal{G}(u,\varphi)^2. 
\end{align*}
for $C = C(n, m,\varphi^{(0)}, \g) \in (0,\infty)$.  Hence 
\begin{equation} \label{coarse_graph_eqn6} 
	\int_{A_{\rho,1/2}(\zeta) \setminus U} (|u|^2 + r^2 |Du|^2) \leq C (1+\delta^{-1}) \int_{A_{\rho,3/4}(\zeta)} \mathcal{G}(u,\varphi)^2. 
\end{equation}
for $C = C(n, m, \varphi^{(0)}, \g) \in (0,\infty)$ for all $\zeta \in B^{n-2}_{\gamma}(0)$ and $\rho \in (0,1/2)$.  By (\ref{coarse_graph_eqn6}) and an easy covering argument, it follows that 
\begin{equation*} 
	\int_{B_{\gamma}(0) \setminus U} (|u|^2 + r^2 |Du|^2) \leq C \int_{B_1(0)} \mathcal{G}(u,\varphi)^2 
\end{equation*}
for some constant $C = C(n,m,\varphi^{(0)},\gamma) \in (0,\infty)$. This completes the proof in the case of Dirichlet energy minimizers. 

In the case of $C^{1, \mu}$ harmonic functions (i.e.\ under hypotheses (b) and $u \in {\mathcal F}_{\varepsilon_{0}}^{Harm}(\varphi^{(0)})$), the same proof carries over if in place of Lemma \ref{separation_lemma_dm} we use Corollary \ref{separation_lemma_h}.
\end{proof}

\section{A priori estimates: non-concentration of excess} \label{nonconcentration_section}
\setcounter{equation}{0}
Let $\varphi^{(0)}$ be a non-zero symmetric two-valued homogeneous function as in (\ref{varphi0}), with degree of homogeneity $\alpha = k/2$ for some positive integer $k$. In Theorem~\ref{thm6_2} and Corollaries~\ref{alpha=1/2}--\ref{cor6_6} below, we establish several key new a priori estimates  for functions $u \in {\mathcal F}^{Dir}_{\varepsilon_{0}}(\varphi^{(0)})$ and  $u \in {\mathcal F}^{Harm}_{\varepsilon_{0}}(\varphi^{(0)})$ for sufficiently small $\varepsilon_{0}$ depending on $\varphi^{(0)}$. These estimates are necessary for the proof of Lemma~\ref{lemma1}, and are inspired by those in Section 3 of~\cite{SimonCTC}. 

The main result in this section is Theorem~\ref{thm6_2}. Its proof is based on the variational identities (\ref{monotonicity_identity1}), (\ref{monotonicity_identity2}), and in particular on the monotonicity formula of Lemma~\ref{new-monotonicity} below, which appears not to have been used before in the context of analysis of branch points. 

Corollary~\ref{alpha=1/2}, which is a consequence of the identity (\ref{monotonicity_identity2}) and Theorem~\ref{thm6_2}(b),  will be needed to handle the case $\alpha = 1/2$ of our main results. 

The rest of the estimates, given in Corollaries~\ref{cor6_3}--\ref{cor6_6}, follow from Theorem~\ref{thm6_2}(a) by the corresponding arguments  in \cite{SimonCTC} with only minor modifications. 

Note in particular Corollary~\ref{cor6_6} (b): it implies, under a certain hypothesis concerning distribution of branch points, that the $L^{2}$ deviation (excess) of $u$ 
from a function $\varphi \in \Phi_{\varepsilon_{0}}(\varphi^{(0)})$ does not concentrate near the axis of $\varphi$; this fact is indispensable in the proof of Lemma~\ref{lemma1}.

 We recall our notation: $D_{u, Y}(\rho) = \r^{2-n}\int_{B_{\r}(Y)}|Du|^{2}$, $H_{u, Y}(\rho) = \r^{1-n}\int_{\partial \, B_{\r}(Y)}|u|^{2}$ and $N_{u, Y}(\r) = D_{u, Y}(\r)/H_{u, Y}(\r).$

\begin{lemma}\label{new-monotonicity}
Let $\alpha \in {\mathbb R}$. If $u \in W^{1, 2}(B_{1}(0); {\mathcal A}_{2}({\mathbb R}^{m}))$ is Dirichlet energy minimizing in $B_{1}(0),$ or if $u \in C^{1, \mu}(B_{1}(0); {\mathcal A}_{2}({\mathbb R}^{m}))$ for some $\mu \in (0, 1)$ and $u$ is harmonic in $B_{1}(0)$, then for each $Y \in B_{1}(0)$, we have that 
$$\frac{d}{d\rho}\left(\rho^{-2\alpha} (D_{u,Y}(\rho) - \alpha H_{u,Y}(\rho))\right)
= 2\rho^{2-n}\int_{\partial\,B_{\rho}(Y)} \left| \frac{\partial (u/R^{\alpha})}{\partial R} \right|^2$$ 
for a.e.\ $\rho \in (0, {\rm dist}\, (Y, \partial \, B_{1}(0))).$ Here $R(X) = |X - Y|.$ 
\end{lemma}

\begin{proof}
Compute directly using the identities (\ref{energy1}) and (\ref{energy2}). 
\end{proof}

It is convenient to state and prove the rest of our results in this section for the case $u \in {\mathcal F}_{\epsilon_{0}}^{Dir}(\varphi^{(0)})$. The results all hold without change in case $u \in {\mathcal F}_{\varepsilon_{0}}^{Harm}(\varphi^{(0)})$, with essentially the same proofs; the only (minor) modification needed is in the proof of Corollary~\ref{cor6_5}, which we shall point out at the end of the section. 
   
\begin{theorem} \label{thm6_2} 
Let $\varphi^{(0)} \in W^{1,2}(\mathbb{R}^n;\mathcal{A}_2(\mathbb{R}^m))$ be a Dirichlet minimizing two-valued function given by (\ref{varphi0}) for some $c^{(0)} \in \mathbb{C}^m \setminus \{0\}$ and $\alpha \in \{1/2,1,3/2,2,\ldots\}$.  Let $\gamma \in (0,1).$ There exists $\varepsilon_0 = \varepsilon_0(\varphi^{(0)},\gamma,\sigma) > 0$ such that if $\varphi \in \Phi_{\varepsilon_0}(\varphi^{(0)})$, $u \in \mathcal{F}^{\text{Dir}}_{\varepsilon_0}(\varphi^{(0)})$, $0 \in \Sigma_u$ and $\mathcal{N}_u(0) \geq \alpha$, then 
\begin{equation*}
\hspace{-1in}(a)\;\;\;\;\; \int_{B_{\gamma}(0)} R^{2-n} \left| \frac{\partial (u/R^{\alpha})}{\partial R} \right|^2 \leq C \int_{B_1(0)} \mathcal{G}(u,\varphi)^2 \;\;\; \mbox{and}
\end{equation*}
\begin{equation*}
\hspace{-2.05in}(b) \;\;\;\;\; \int_{B_{\gamma}(0)} |D_{y}u|^{2} \leq C\int_{B_1(0)} \mathcal{G}(u,\varphi)^2
\end{equation*}
where $R = |X|$ and $C = C(n,m,\varphi^{(0)},\alpha,\gamma) \in (0,\infty)$.
\end{theorem}

\begin{proof} 
Suppose that $\varepsilon_0$ is less than or equal to $\varepsilon_0$ in Lemma \ref{coarse_graph} with $(1+\gamma)/2$ in place of $\gamma$, $\tau = (1-\gamma)/20$ and $\beta = 1/2$.  Let $U$, $v$, and $\hat v$ be as in Lemma \ref{coarse_graph}. 

By Lemma~\ref{new-monotonicity},
\begin{align} \label{eqn4-1}
	\frac{d}{d\rho} \left( \rho^{n-2} (D_{u,0}(\rho) - \alpha H_{u,0}(\rho)) \right) 
	&= \frac{d}{d\rho} \left( \rho^{n-2+2\alpha} \cdot \rho^{-2\alpha} (D_{u,0}(\rho) - \alpha H_{u,0}(\rho)) \right) \nonumber \\ 
	&= (n-2+2\alpha) \rho^{n-3} (D_{u,0}(\rho) - \alpha H_{u,0}(\rho)) 
		+ 2\rho^{2\alpha} \int_{\partial B_{\rho}(0)}  \left| \frac{\partial (u/R^{\alpha})}{\partial R}\right|^2 .
\end{align}
Again by Lemma~\ref{new-monotonicity}  and the fact that $N_{u,0}(\rho) \geq \alpha,$  
\begin{equation*}
	\rho^{-2\alpha} (D_{u,0}(\rho) - \alpha H_{u,0}(\rho)) \geq 2 \int_{B_{\rho}(0)} R^{2-n}  \left| \frac{\partial (u/R^{\alpha})}{\partial R}\right|^2. 
\end{equation*}
Using this in  (\ref{eqn4-1}), we get 
\begin{equation} \label{lemma3_4_eqn3}
	\frac{d}{d\rho} \left( \rho^{n-2} (D_{u,0}(\rho) - \alpha H_{u,0}(\rho)) \right) 
	\geq 2 \frac{d}{d\rho} \left( \rho^{n-2+2\alpha} \int_{B_{\rho}(0)} R^{2-n}  \left| \frac{\partial (u/R^{\alpha})}{\partial R}\right|^2  \right). 
\end{equation}

Let $\psi : [0,\infty) \rightarrow \mathbb{R}$ be a smooth function with $\psi(t) = 1$ for $t \in [0,\gamma]$, $\psi(t) = 0$ for $t \geq (1+\gamma)/2$, and $0 \leq \psi'(t) \leq 3/(1-\gamma)$ for $t \in [0,\infty)$.  Multiplying both sides of (\ref{lemma3_4_eqn3}) by $\psi(\rho)^2$ and integrating yields 
\begin{align} \label{lemma3_4_eqn4}
	&\hspace{-1in}\int_{\mathbb{R}} \frac{d}{d\rho} (\rho^{n-2} D_{u,0}(\rho) - \alpha \rho^{n-2} H_{u,0}(\rho)) \psi(\rho)^2 d\rho \nonumber\\ 
	&\geq 2 \int_{\mathbb{R}} \frac{d}{d\rho} \left( \rho^{n-2 +2\alpha} \int_{B_{\rho}(0)} 
		R^{2-n}\left| \frac{\partial (u/R^{\alpha})}{\partial R} \right|^2 \right) \psi(\rho)^2 d\rho \nonumber 
	\\&= -4 \int_{\mathbb{R}} \rho^{n-2+2\alpha} \int_{B_{\rho}(0)} R^{2-n} 
		\left| \frac{\partial (u/R^{\alpha})}{\partial R} \right|^2 \psi(\rho) \psi'(\rho) d\rho \nonumber 
	\\&= -4 \int_{\gamma}^{(1+\gamma)/2} \rho^{n-2+2\alpha} \int_{B_{\rho}(0)} R^{2-n}
		\left| \frac{\partial (u/R^{\alpha})}{\partial R} \right|^2 \psi(\rho) \psi'(\rho) d\rho \nonumber 
	\\&\geq -4 \gamma^{n-2+2\alpha} \int_{B_{\gamma}(0)} R^{2-n}
		\left| \frac{\partial (u/R^{\alpha})}{\partial R} \right|^2 \int_{\gamma}^{(1+\gamma)/2} \psi(\rho) \psi'(\rho) d\rho \nonumber 
	\\&= 2 \gamma^{n-2+2\alpha} \int_{B_{\gamma}(0)} R^{2-n} \left| \frac{\partial (u/R^{\alpha})}{\partial R} \right|^2.
\end{align}
By the coarea formula
\begin{equation} \label{lemma3_4_eqn5}
	\int_{\mathbb{R}} \frac{d}{d\rho} (\rho^{n-2} D_{u,0}(\rho)) \psi(\rho)^2 d\rho 
	= \int_{\mathbb{R}} \int_{\partial B_{\rho}} |Du|^2 \psi(\rho)^2 d\rho 
	= \int |Du|^2 \psi(R)^2 
\end{equation}
and by integration by parts and the coarea formula again
\begin{align} \label{lemma3_4_eqn6}
	\int_{\mathbb{R}} \frac{d}{d\rho} (\rho^{n-2} H_{u,0}(\rho)) \psi(\rho)^2 d\rho 
	&= -2 \int_{\mathbb{R}} \rho^{n-2} H_{u,0}(\rho) \psi(\rho) \psi'(\rho) d\rho 
	\nonumber \\&= -2 \int_{\mathbb{R}} \int_{\partial B_{\rho}} \rho^{-1} |u|^2 \psi(\rho) \psi'(\rho) d\rho 
	\nonumber \\&= -2 \int R^{-1} |u|^2 \psi(R) \psi'¨.
\end{align}
By (\ref{lemma3_4_eqn4}), (\ref{lemma3_4_eqn5}), (\ref{lemma3_4_eqn6}), 
\begin{equation} \label{lemma3_4_eqn7}
	\int (|Du|^2 \psi(R)^2 + 2\alpha R^{-1} |u|^2 \psi(R) \psi'(R)) 
	\geq C \int_{B_{\gamma}(0)} R^{2-n} \left| \frac{\partial (u/R^{\alpha})}{\partial R} \right|^2 
\end{equation}
where $C = C(n,\alpha,\gamma) \in (0,\infty)$.  Notice that by integration by parts, 
\begin{equation*}
	\int_{B^2_1(0)} (|D\varphi|^2 \psi_0(r)^2 + 2\alpha r^{-1} |\varphi|^2 \psi_0(r) \psi_0'(r)) = 0  
\end{equation*}
for any $\psi_0 \in C^1_c([0, 1))$ where $r=r(x) = |x|$, so if we let $\psi_0(r) = \psi\left(\sqrt{r^2 + |y|^2}\right)$ for each $y \in B^{n-2}_1(0)$ and use the fact that $\psi_0'(r) = (r/R) \psi'(R)$, we obtain 
\begin{equation} \label{lemma3_4_eqn8}
	\int (|D\varphi|^2 \psi(R)^2 + 2\alpha R^{-1} |\varphi|^2 \psi(R) \psi'(R)) = 0. 
\end{equation}

Now let $(\zeta^1,\ldots,\zeta^n) = \psi(R)^2 (x_1,x_2,0,\ldots,0)$ in (\ref{monotonicity_identity2}) to obtain 
\begin{equation} \label{lemma3_4_eqn9}
	\int (|Du|^2 - |D_x u|^2) \psi(R)^2 
	= -2 \int \left( \tfrac{1}{2} r^2 |Du|^2 - r^2 |D_r u|^2 - r D_r u (y \cdot D_y u) \right) R^{-1} \psi(R) \psi'(R) 
\end{equation}
where $r = |x|$.  Let $\zeta = \psi(R)^2$ in (\ref{monotonicity_identity1}) to obtain 
\begin{equation} \label{lemma3_4_eqn10}
	\int |Du|^2 \psi(R)^2 = -2 \int (r u D_r u + u (y \cdot D_y u)) R^{-1} \psi(R) \psi'(R).
\end{equation}
By multiplying (\ref{lemma3_4_eqn10}) by $\a$ and adding it to (\ref{lemma3_4_eqn9}), and adding also $\int 2\alpha R^{-1} |u|^2 \psi(R) \psi'(R)$ to both sides, we obtain 
\begin{align*}
	&\int ((\alpha |Du|^2 + |D_y u|^2) \psi(R)^2 + 2\alpha^2 R^{-1} |u|^2 \psi(R) \psi'(R))
	\\&= -2 \int \left( \tfrac{1}{2} r^2 |Du|^2 - r D_r u (r D_r u - \alpha u) - \alpha^2 |u|^2 
	- (y \cdot D_y u)(r D_r u - \alpha u)  \right) R^{-1} \psi(R) \psi'(R).
\end{align*}
By Cauchy's inequality, 
\begin{align} \label{lemma3_4_eqn11}
	&\int \left( \left( \alpha |Du|^2 + \tfrac{1}{2} |D_y u|^2 \right) \psi(R)^2 + 2\alpha^2 R^{-1} |u|^2 \psi(R) \psi'(R) \right)
	\nonumber \\&\leq -2 \int \left( \tfrac{1}{2} r^2 |Du|^2 - r D_r u (r D_r u - \alpha u) - \alpha^2 |u|^2 \right) R^{-1} \psi(R) \psi'(R) 
	\nonumber \\& \hspace{5mm} + 2 \int |r D_r u - \alpha u|^2 \psi'(R)^2. 
\end{align}
Observe that by direct computation 
\begin{equation} \label{lemma3_4_eqn12}
 \int \left( \tfrac{1}{2} r^2 |D\varphi|^2 - \alpha^2 |\varphi|^2 \right) R^{-1} \psi(R) \psi'(R) = 0. 
\end{equation}
By (\ref{lemma3_4_eqn11}) and (\ref{lemma3_4_eqn12}), 
\begin{align} \label{lemma3_4_eqn13}
	&\int \left( \left( \alpha |Du|^2 + \tfrac{1}{2} |D_y u|^2 \right) \psi(R)^2 + 2\alpha^2 R^{-1} |u|^2 \psi(R) \psi'(R) \right) \nonumber 
	\\&\leq -2 \int \left( \tfrac{1}{2} r^2 (|Du|^2 - |D\varphi|^2) - r D_r u (r D_r u - \alpha u) - \alpha^2 (|u|^2 - |\varphi|^2) \right) 
	R^{-1} \psi(R) \psi'(R) \nonumber \\& \hspace{5mm} + 2 \int |r D_r u - \alpha u|^2 \psi'(R)^2. 
\end{align}

 Now recall that $\varphi(re^{i\theta},y) = \{ \pm \op{Re}(c r^{\alpha} e^{i\alpha \theta}) \}$ for some $c \in \mathbb{C}^m \setminus \{0\}$.  Let 
\begin{equation*} 
	\varphi_*(r,\theta,y) = \op{Re}(c r^{\alpha} e^{i\alpha \theta}), \hspace{5mm} v_*(r,\theta,y) = v(re^{i\theta},y,\varphi_*(r,\theta,y)) 
\end{equation*}
for $r > 0$, $\theta \in \mathbb{R}$, and $y \in \mathbb{R}^{n-2}$ such that $(re^{i\theta},y) \in U$.  Recall from Remark \ref{graph_rmk} that in  case $\alpha = k/2$ for an odd integer $k$, $\varphi_*(r,\theta,y)$ is $4\pi$-periodic as a function of $\theta$.  Since $\varphi(re^{i\theta},y) = \{ \pm \varphi_*(r,\theta,y) \}$, $u(re^{i\theta},y) = \{ \pm (\varphi_*(r,\theta,y) + v_*(r,\theta,y)) \}$, and $r D_r \varphi_* = \alpha \varphi_*$, 
\begin{align} \label{lemma3_4_eqn14}
	&\int_U \left( \tfrac{1}{2} r^2 (|Du|^2 - |D\varphi|^2) - r D_r u (r D_r u - \alpha u) - \alpha^2 (|u|^2 - |\varphi|^2) \right) R^{-1} \psi(R) \psi'(R) 
	\nonumber \\&= \int_0^{4\pi} \int_{(r,y) : (x,y) \in U} \left( D_{\theta} \varphi_* D_{\theta} v_* - \alpha^2 \varphi_* v_* + \tfrac{1}{2} r^2 |Dv_*|^2 
	\right. \nonumber \\& \hspace{5mm} \left. - r D_r v_* (r D_r v_* - \alpha v_*) - \alpha^2 |v_*|^2 \right) R^{-1} \psi(R) \psi'(R) 
	\, r \, dr \, dy \, d\theta. 
\end{align}
In case that $\alpha$ is an integer, $\varphi_*(r,\theta,y)$ is $2\pi$-periodic as a function of $\theta$ such that 
\begin{equation*}
	\varphi(re^{i\theta},y) = \{ \pm \varphi_*(r,\theta,y) \}, \quad v(re^{i\theta},y,\pm \varphi_*(r,\theta,y)) = \pm v_*(r,\theta,y). 
\end{equation*}
The integrand on the left-hand side of (\ref{lemma3_4_eqn14}) is a quadratic expression in terms of the symmetric two-valued functions $u$, $\varphi$, and their derivatives.  Writing $u(X) = \{ \pm u_1(X) \}$ and $\varphi(X) = \{ \pm \varphi_1(X) \}$ where $u_1$ and $\varphi_1$ are the single-valued functions defined by $u_1(re^{i\theta},y) = \varphi_*(r,\theta,y) + v_*(r,\theta,y)$ and $\varphi_1(re^{i\theta},y) = \varphi_*(r,\theta,y)$, 
\begin{align*}
	&\int_U \left( \tfrac{1}{2} r^2 (|Du|^2 - |D\varphi|^2) - r D_r u (r D_r u - \alpha u) - \alpha^2 (|u|^2 - |\varphi|^2) \right) R^{-1} \psi(R) \psi'(R) \\
	&= 2\int_U \left( \tfrac{1}{2} r^2 (|Du_1|^2 - |D\varphi_1|^2) - r D_r u_1 (r D_r u_1 - \alpha u_1) - \alpha^2 (|u_1|^2 - |\varphi_1|^2) \right) 
		R^{-1} \psi(R) \psi'(R), 
\end{align*}
from which (\ref{lemma3_4_eqn14}) follows.  Hence (\ref{lemma3_4_eqn14}) holds regardless of the value of $\alpha$.  

By integration by parts with respect to $\theta$ (keeping in mind rotational symmetry of $U$ with respect to the $x$ variables), 
\begin{align} \label{lemma3_4_eqn16}
	&\hspace{-1in}\int_0^{4\pi} \int_{(r,y) : (x,y) \in U} D_{\theta} \varphi_* D_{\theta} v_* R^{-1} \psi(R) \psi'(R) \, r \, dr \, d\theta \, dy \nonumber 
	\\&= - \int_0^{4\pi} \int_{(r,y) : (x,y) \in U} D_{\theta \theta} \varphi_* v_* R^{-1} \psi(R) \psi'(R) \, r \, dr \, d\theta \, dy \nonumber 
	\\&= \int_0^{4\pi} \int_{(r,y) : (x,y) \in U} \alpha^2 \varphi_* v_* R^{-1} \psi(R) \psi'(R) \, r \, dr \, d\theta \, dy. 
\end{align}
By (\ref{lemma3_4_eqn14}), (\ref{lemma3_4_eqn16}), and the estimates on $\hat v(re^{i\theta},y) = \{ \pm v_*(r,\theta,y) \}$ from Lemma \ref{coarse_graph}, 
\begin{align} \label{lemma3_4_eqn17}
	&\hspace{-.5in}\left| \int_U \left( \tfrac{1}{2} r^2 (|Du|^2 - |D\varphi|^2) - r D_r u (r D_r u - \alpha u) - \alpha^2 (|u|^2 - |\varphi|^2) \right) 
		R^{-1} \psi(R) \psi'(R) \right| \nonumber 
	\\&\hspace{3in}\leq C \int_{B_1(0)} \mathcal{G}(u,\varphi)^2 
\end{align}
for $C = C(n,m,\varphi^{(0)},\alpha,\gamma,\beta) \in (0,\infty)$.  Using (\ref{lemma3_4_eqn17}), the fact that $|r D_r u - \alpha u|^2 = |r D_r \hat v - \alpha \hat v|^2$ on $U$, the fact that 
\begin{eqnarray*}
&\hspace{-2in}\int_{B_{(1+\g)/2} \setminus U}\left(r^{2}|D\varphi|^{2} -2\a^{2}|\varphi|^{2}\right) R^{-1}\psi(R)\psi^{\prime}(R)\nonumber\\  
&\hspace{1in}=\int_{\{(r, y) \, : \, (re^{i\th}, y) \in B_{(1+\g)/2} \setminus U\}} \int_{0}^{4\pi}\left(r^{2}|D\varphi|^{2} -2\a^{2}|\varphi|^{2}\right) R^{-1}\psi(R)\psi^{\prime}(R) rd\th dr dy = 0
\end{eqnarray*}
(which follows  by integration by parts with respect to the $\th$ variable in view of the fact that $\varphi_{\th\th} + \a^{2} \varphi = 0$ on $\{r=1\}$ and $B_{(1+\g)/2} \setminus U$ is rotationally symmetric in the 
$x= re^{i\th}$ variables) 
and the estimates from Lemma \ref{coarse_graph} to bound the right-hand side of (\ref{lemma3_4_eqn13}) we obtain,  
\begin{equation} \label{lemma3_4_eqn18}
	\int \left( \left( \alpha |Du|^2 + \tfrac{1}{2} |D_y u|^2 \right) \psi(R)^2 + 2\alpha^2 R^{-1} |u|^2 \psi(R) \psi'(R) \right) 
	\leq C \int_{B_1(0)} \mathcal{G}(u,\varphi)^2 
\end{equation}
for $C = C(n,m,\varphi^{(0)},\alpha,\gamma,\beta) \in (0,\infty)$.  Combining (\ref{lemma3_4_eqn7}) and (\ref{lemma3_4_eqn18}),  we deduce 
\begin{equation} \label{lemma3_4_eqn19}
	\int_{B_{\gamma}(0)} R^{2-n} \left( \frac{\partial (u/R^{\alpha})}{\partial R} \right)^2 
	+ \int_{B_{\gamma}(0)} |D_y u|^2 \leq C \int_{B_1(0)} \mathcal{G}(u,\varphi)^2 
\end{equation}
for $C = C(n,m,\varphi^{(0)},\alpha,\gamma) \in (0,\infty)$, which is the conclusion of the theorem.
\end{proof}

Next we derive a corollary of Theorem~\ref{thm6_2}(b) which we shall need for the case $\alpha = 1/2$ of our main theorems.

\begin{corollary}\label{alpha=1/2}
Let $\varphi^{(0)} \in W^{1,2}(\mathbb{R}^n;\mathcal{A}_2(\mathbb{R}^m))$ be a Dirichlet minimizing two-valued function of the form given by (\ref{varphi0})  for some $c^{(0)} \in \mathbb{C}^m \setminus \{0\},$ and suppose that $\alpha = 1/2$, i.e.\ that $\varphi^{(0)}$ is homogeneous of degree $1/2$.  Let $\tau \in (0,1/20)$.  There exists $\varepsilon_0 = \varepsilon_0(n,m,\varphi^{(0)},\tau) > 0$ such that if $\varphi \in \Phi_{\varepsilon_0}(\varphi^{(0)})$, $u \in \mathcal{F}_{\varepsilon_{0}}^{\text{Dir}}(\varphi^{(0)})$ with $0 \in \Sigma_u$ and $\mathcal{N}_u(0) \geq \alpha = 1/2$, $v$ is as in 
Lemma~\ref{coarse_graph}  with $\gamma = 1/2$ and $\beta = 1/2$, 
$\hat v(X) = \{v(X, \varphi(X)), v(X, -\varphi(X))\}$, and 
\begin{equation*} \label{eqn1}
	B_{\tau}(0,y) \cap \{ X \in B_1(0) : \mathcal{N}_u(X) \geq \mathcal{N}_{\varphi}(0) = \alpha \} \neq \emptyset \text{ for all } y \in B^{n-2}_{1/2}(0), 
\end{equation*}
then for each function $\z(x, y) = \widetilde{\z}(|x|, y)$ where $\widetilde{\zeta} = \widetilde{\zeta}(r,y) \in C^1_c(B^{n-1}_{1/2}(0))$ with $D_r \widetilde\zeta(r,y) = 0$ whenever $r \leq \tau$, we have that
\begin{eqnarray*}
	&&\left|\int_{B^{n-2}_{1/2}(0)}\int_{\t}^{\infty} \int_{0}^{4\pi} D_j (r \hat v^{\kappa} D_i \varphi^{\kappa}) D_j D_{y_{p}} \zeta d\th dr dy\right|\nonumber\\
	&&\hspace{1in}\leq C \left( \tau^{-2} \int_{B_1(0)} \mathcal{G}(u,\varphi)^2 + \tau^{2\alpha} \left( \int_{B_1(0)} \mathcal{G}(u,\varphi)^2 \right)^{1/2} \right) 
		\sup_{B_{1/2}(0)} |DD_{y_{p}} \zeta|
\end{eqnarray*}
for $i=1, 2$ and for each $p= 1,2,\ldots,n-2,$ where $C = C(n,m,\alpha) \in (0,\infty)$.
\end{corollary}
\begin{proof} 
Fix $i \in \{1, 2\}$ and replace $\zeta^{j}$ with $\d_{ij}D_{y_{p}} \zeta(r,y)$ in (\ref{monotonicity_identity2}) and recall that $u(X) = \{ \pm (\varphi_1(X) + v(X,\varphi_1(X))) \}$ for $X \in B_{1/2}(0) \setminus \{ r \geq \tau \}$ to obtain
\begin{align*}
	&\frac{-1}{2} \int_{B_{1/2}(0) \cap \{r \geq \tau\}} (|D\varphi|^2 + 2D_j \varphi^{\kappa} D_j \hat v^{\kappa} + |D\hat v|^2) D_i D_{y_{p}} \zeta 
	\\&+ \int_{B_{1/2}(0) \cap \{r \geq \tau\}} (D_i \varphi^{\kappa} D_j \varphi^{\kappa} + D_i \varphi^{\kappa} D_j \hat v^{\kappa} 
		+ D_i \hat v^{\kappa} D_j \varphi^{\kappa} + D_i \hat v^{\kappa} D_j \hat v^{\kappa}) D_j D_{y_{p}} \zeta 
	\\&= -\int_{B_1(0) \cap \{r \leq \tau\}} D_i u^{\kappa} D_{y_l} u^{\kappa} \cdot D_{y_l} D_{y_{p}} \zeta. 
\end{align*}
Observe that, since $\varphi$ is independent of $y$, by integration by parts
\begin{equation*}
	\int_{B_{1/2}(0)} \left( \frac{-1}{2} |D\varphi|^2 D_i D_{y_{p}} \zeta + D_i \varphi^{\kappa} D_j \varphi^{\kappa} D_j D_{y_{p}} \zeta \right) = 0.
\end{equation*}
Thus 
\begin{align} \label{half-eqn2}
	&\int_{B_{1/2}(0) \cap \{r \geq \tau\}} (-D_j \varphi^{\kappa} D_j \hat v^{\kappa} D_i D_{y_{p}} \zeta 
		+ D_i \varphi^{\kappa} D_j \hat v^{\kappa} D_j D_{y_{p}} \zeta + D_i \hat v^{\kappa} D_j \varphi^{\kappa} D_j D_{y_{p}} \zeta) \nonumber \\
	&= \int_{B_{1/2}(0) \cap \{r \geq \tau\}} \left( \frac{1}{2} |D\hat v|^2 D_i D_{y_{p}} \zeta - D_i \hat v^{\kappa} D_j v^{\kappa} D_j D_{y_{p}} \zeta \right) 
	- \int_{B_1(0) \cap \{r \leq \tau\}} D_i u^{\kappa} D_{y_l} u^{\kappa} \cdot D_{y_l} \zeta. 
\end{align}
By the estimates on $v$ in Lemma~\ref{coarse_graph},
\begin{equation*}
	\int_{B_{1/2}(0) \cap \{r \geq \tau\}} r^2 |D\hat v|^2 \leq C \int_{B_1(0)} \mathcal{G}(u,\varphi)^2
\end{equation*}
for some constant $C = C(n,m,\varphi^{(0)},\alpha) \in (0,\infty)$, so 
\begin{equation} \label{half-eqn3}
	\int_{B_{1/2}(0) \cap \{r \geq \tau\}} |D\hat v|^2 |DD_{y_{p}} \zeta| \leq C \tau^{-2} \int_{B_1(0)} \mathcal{G}(u,\varphi) \sup_{B_{1/2}(0)} |DD_{y_{p}} \zeta| 
\end{equation}
for some $C = C(n,m,\varphi^{(0)},\alpha) \in (0,\infty)$.  By hypothesis, for every $y \in B^{n-2}_{1/2}(0)$ there exists a $Z \in B_{\tau}(0,y)$ such that $\mathcal{N}_u(Z) \geq \alpha$ and thus by (\ref{schauder_dm}) and (\ref{doubling_estimate2}), 
\begin{equation*}
	\int_{B_{2\tau}(0,y)} |Du|^2 \leq \int_{B_{3\tau}(Z)} |Du|^2 \leq C \tau^{-2} \int_{B_{6\tau}(Z)} |u|^2 
	\leq C \tau^{n-2+2\alpha} \int_{B_{1/2}(Z)} |u|^2 \leq C \tau^{n-2+2\alpha} 
\end{equation*}
for constants $C = C(n,m,\varphi^{(0)}) \in (0,\infty)$ and thus by a standard covering argument 
\begin{equation} \label{half-eqn4}
	\int_{B_{1/2}(0) \cap \{r \leq \tau\}} |Du|^2 \leq C \tau^{2\alpha} 
\end{equation}
for some constant $C = C(n,m,\varphi^{(0)}) \in (0,\infty)$.  By (\ref{half-eqn4}) and Theorem 6.2(b),  
\begin{align} \label{half-eqn5}
	\left|\int_{B_1(0) \cap \{r \leq \tau\}} D_i u^{\kappa} D_{y_l} u^{\kappa} \cdot D_{y_l} D_{y_{p}}\zeta\right|
	&\leq \left( \int_{B_{1/2}(0) \cap \{r \leq \tau\}} |Du|^2 \right)^{1/2} \left( \int_{B_{1/2}(0)} |D_y u|^2 \right)^{1/2} \sup_{B_{1/2}(0)} |DD_{y_{p}} \zeta|
		 \nonumber \\
	&\leq C \tau^{2\alpha} \left( \int_{B_1(0)} \mathcal{G}(u,\varphi) \right)^{1/2} \sup_{B_{1/2}(0)} |DD_{y_{p}} \zeta|
\end{align}
for some $C = C(n,m,\varphi^{(0)}) \in (0,\infty)$.  Using (\ref{half-eqn3}) and (\ref{half-eqn5}) to bound the right-hand side of (\ref{half-eqn2}), we get
\begin{align} \label{half-eqn6}
	&\hspace{-.5in}\left|\int_{B_{1/2}(0) \cap \{r \geq \tau\}} (-D_j \varphi^{\kappa} D_j \hat v^{\kappa} D_i D_{y_{p}} \zeta 
		+ D_i \varphi^{\kappa} D_j \hat v^{\kappa} D_j D_{y_{p}} \zeta + D_i \hat v^{\kappa} D_j \varphi^{\kappa} D_j D_{y_{p}} \zeta) \right|\nonumber \\
	&\hspace{.5in}\leq C \left( \tau^{-2} \int_{B_1(0)} \mathcal{G}(u,\varphi)^2 + \tau^{2\alpha} \left( \int_{B_1(0)} \mathcal{G}(u,\varphi)^2 \right)^{1/2} \right) 
		\sup_{B_{1/2}(0)} |DD_{y_{p}} \zeta|
\end{align}
for some $C = C(n,m,\varphi^{(0)}) \in (0,\infty)$.  Now observe that, using the fact that $\varphi$ is homogeneous degree $1/2$ and independent of $y$ and $\z$ depends only on $r$ and $y$,
\begin{eqnarray*}
&&\hspace{-.5in}	\int_{B^{n-2}_{1/2}(0)} \int_{\t}^{\infty} \int_{0}^{4\pi}D_j (r \hat v^{\kappa} D_i \varphi^{\kappa}) D_j D_{y_{p}} \zeta d\th dr dy\nonumber\\ 
&&\hspace{.5in}	= \int_{B^{n-2}_{1/2}(0)} \int_{\t}^{\infty}\int_{0}^{4\pi} \left( \frac{1}{2} \hat v^{\kappa} D_i \varphi^{\kappa} D_r D_{y_{p}} \zeta 
		+ r D_j \hat v^{\kappa} D_i \varphi^{\kappa} D_j D_{y_{p}} \zeta \right) d\th dr  dy 
\end{eqnarray*}
Using (\ref{half-eqn6}) to substitute for the integral of $r D_j \hat v^{\kappa} D_i \varphi^{\kappa} D_j D_{y_{p}}\z$, 
\begin{align*}
	&\hspace{-.5in}\int_{B^{n-2}_{1/2}(0)}\int_{\t}^{\infty}\int_{0}^{4\pi} D_j (r \hat v^{\kappa} D_i \varphi^{\kappa}) D_j D_{y_{p}} \zeta d\th dr  dy 
	= \frac{1}{2} \int_{B^{n-2}_{1/2}(0)}\int_{\t}^{\infty}\int_{0}^{4\pi} \hat v^{\kappa} D_i \varphi^{\kappa} D_r D_{y_{p}} \zeta d\th dr  dy
		\\&\hspace{.5in}+ \int_{B^{n-2}_{1/2}(0)}\int_{\t}^{\infty}\int_{0}^{4\pi} (r D_j \hat v^{\kappa} D_j \varphi^{\kappa} D_i D_{y_{p}} \zeta 
		- r D_i \hat v^{\kappa} D_j \varphi^{\kappa} D_j D_{y_{p}} \zeta) d\th dr  dy + \mathcal{R} 
\end{align*}
where $\mathcal{R}$ is such that 
\begin{equation*}
	|\mathcal{R}| \leq C \left( \tau^{-2} \int_{B_1(0)} \mathcal{G}(u,\varphi)^2 + \tau^{2\alpha} 
		\left( \int_{B_1(0)} \mathcal{G}(u,\varphi)^2 \right)^{1/2} \right) \sup_{B_{1/2}(0)} |DD_{y_{p}} \zeta|
\end{equation*}
for some constant $C = C(n,m,\varphi^{(0)}) \in (0,\infty)$.  Again since $\varphi$ is homogeneous degree $1/2$ and independent of $y$ and $\zeta$ depends only on $r$ and $y$,
\begin{align*}
	&\int_{B^{n-2}_{1/2}(0)} \int_{\t}^{\infty}\int_{0}^{4\pi}D_j (r \hat v^{\kappa} D_i \varphi^{\kappa}) D_j D_{y_{p}} \zeta d\th dr  dy 
	= \frac{1}{2} \int_{B^{n-2}_{1/2}(0)} \int_{\t}^{\infty}\int_{0}^{4\pi}\hat v^{\kappa} D_i \varphi^{\kappa} D_r D_{y_{p}} \zeta d\th dr dy
		\\&\hspace{1.75in}+ \int_{B^{n-2}_{1/2}(0)}\int_{\t}^{\infty}\int_{0}^{4\pi} \left( x_i D_j \hat v^{\kappa} D_j \varphi^{\kappa} 
		- \frac{1}{2} D_i \hat v^{\kappa} \varphi^{\kappa} \right) D_r D_{y_{p}} \zeta d\th dr  dy + \mathcal{R}\nonumber\\
		&\hspace{1.5in}=\int_{B^{n-2}_{1/2}(0)}\int_{\t}^{\infty}\int_{0}^{4\pi} \left( \frac{-1}{2} D_i (\hat v^{\kappa} \varphi^{\kappa}) + D_j (x_i \hat v^{\kappa} D_j \varphi^{\kappa}) 
		\right) D_r D_{y_{p}} \zeta d\th dr  dy + \mathcal{R}\nonumber\\
		&\hspace{1.5in} = \int_{B_{1/2}(0) \cap \{r\geq \t\}}   \left( \frac{-1}{2} D_i (\hat v^{\kappa} \varphi^{\kappa}) + D_j (x_i \hat v^{\kappa} D_j \varphi^{\kappa}) 
		\right) r^{-1}D_r D_{y_{p}} \zeta \, dX
\end{align*}
where we have used the fact that $\varphi$ is locally given by harmonic functions away from $\{r=0\}.$ By integrating by parts with respect to $r$ we deduce from the preceding line that 
\begin{align*}
	&\int_{B^{n-2}_{1/2}(0)}\int_{\t}^{\infty}\int_{0}^{4\pi} D_j (r \hat v^{\kappa} D_i \varphi^{\kappa}) D_j D_{y_{p}} \zeta d\th dr  dy 
	\\&\hspace{.5in}= -\int_{B_{1/2}(0) \cap \{r = \tau\}} \left( \frac{-1}{2} \frac{x_i}{r} \hat v^{\kappa} \varphi^{\kappa} + x_i \hat v^{\kappa} D_r \varphi^{\kappa} 
		\right) r^{-1}D_r D_{y_{p}} \zeta  \, d{\mathcal H}^{n-1}
		\\&\hspace{1in}- \int_{B_{1/2}(0) \cap \{r \geq \tau\}} \left( \frac{-1}{2} \frac{x_i}{r} \hat v^{\kappa} \varphi^{\kappa} 
		+ x_i \hat v^{\kappa} D_r \varphi^{\kappa} \right) D_{r}\left(r^{-1} D_{r}D_{y_{p}} \zeta\right) \, dX + \mathcal{R}.
\end{align*}
Since $\varphi$ is homogeneous degree $1/2$, $D_r \varphi = (1/2) r^{-1} \varphi$ and thus 
\begin{equation*}
	\int_{B^{n-2}_{1/2}(0)}\int_{\t}^{\infty}\int_{0}^{4\pi} D_j (r \hat v^{\kappa} D_i \varphi^{\kappa}) D_j D_{y_{p}} \zeta d\th dr dy = \mathcal{R}.
\end{equation*}
This completes the proof. 
\end{proof}

\begin{corollary}\label{cor6_3}
Let the hypotheses be as in Theorem~\ref{thm6_2}, and let $\sigma \in (0, 1)$. Then 
\begin{align*}
	\int_{B_{\gamma}(0)} R^{-n+\sigma-2\alpha} \mathcal{G}(u,\varphi)^2 \leq C \int_{B_1(0)} \mathcal{G}(u,\varphi)^2 
\end{align*}
where $C = C(n,m,\varphi^{(0)},\alpha,\gamma,\sigma) \in (0,\infty)$. 
\end{corollary}
\begin{proof}
Recall that for any single-valued vector fields $\zeta = (\zeta^1,\ldots,\zeta^n) \in C_c^{0,1}(\mathbb{R}^n)$ that 
\begin{equation*}
	\int D_i \zeta^i = 0.  
\end{equation*}
Let $\zeta^i = \psi(R)^2 \eta_{\delta}(R) R^{-n+\sigma-2\alpha} \mathcal{G}(u,\varphi)^2 X_i$, where for each $\delta > 0$, $\eta_{\delta} \in C^{\infty}([0,\infty))$ is a cutoff function such that $\eta_{\delta}(t) = 0$ for $t \in [0,\delta/2]$, $\eta_{\delta}(t) = 1$ for $t \in [\delta,\infty)$, and $0 \leq \eta'_{\delta}(t) \leq 3/\delta$ for all $t \in [0,\infty)$, to obtain 
\begin{align} \label{lemma3_4_eqn20}
	&\hspace{-.5in}\sigma \int \psi(R)^2 \eta_{\delta}(R) R^{-n+\sigma-2\alpha} \mathcal{G}(u,\varphi)^2 
	= - \int \psi(R)^2 \eta_{\delta}(R) R^{1-n+\sigma} D_R (R^{-2\alpha} \mathcal{G}(u,\varphi)^2) \nonumber 
	\\&\hspace{.5in} + 2 \int \psi(R) \psi'(R) \eta_{\delta}(R) R^{1-n+\sigma-2\alpha} \mathcal{G}(u,\varphi)^2 
	+ \int \psi(R)^2 \eta'_{\delta}(R) R^{1-n+\sigma-2\alpha} \mathcal{G}(u,\varphi)^2 .
\end{align}
Observe that 
\begin{equation*}
	|D_R (R^{-2\alpha} \mathcal{G}(u,\varphi)^2)| 
	= |D_R (\mathcal{G}(u/R^{\alpha}, \varphi/R^{\alpha})^2)| 
	\leq 4 \mathcal{G}(u/R^{\alpha}, \varphi/R^{\alpha}) |D_R (u/R^{\alpha})| 
\end{equation*}
a.e. in $B_1(0)$.  Thus using Cauchy-Schwartz in (\ref{lemma3_4_eqn20}) and using (\ref{lemma3_4_eqn19}), we obtain 
\begin{align} \label{lemma3_4_eqn21}
	\int \psi(R)^2 \eta_{\delta}(R) R^{-n+\sigma-2\alpha} \mathcal{G}(u,\varphi)^2 
	&\leq \frac{36}{\sigma^2} \int \left( \psi(R)^2 R^{2-n+\sigma} |D_R (u/R^{\alpha})|^2 + \psi'(R)^2 R^{2-n+\sigma-2\alpha} \mathcal{G}(u,\varphi)^2 \right) 
		\nonumber \\&\hspace{1.7in} + \frac{6}{\sigma} \int \psi(R)^2 \eta'_{\delta}(R) R^{1-n+\sigma-2\alpha} 
		\mathcal{G}(u,\varphi)^2 \nonumber 
	\\&\hspace{-1.8in}\leq \frac{36}{\sigma^2} \int \psi(R)^2 R^{2-n+\sigma} |D_R (u/R^{\alpha})|^2 + C \int_{B_1(0)} \mathcal{G}(u,\varphi)^{2} + \frac{6}{\sigma} \int \psi(R)^2 \eta'_{\delta}(R) R^{1-n+\sigma-2\alpha} 
		\mathcal{G}(u,\varphi)^2. 
\end{align}
for $C = C(n,m,\alpha,\gamma,\sigma) \in (0,\infty)$.  By the definition of $\eta_{\delta}$, the Cauchy inequality, and (\ref{doubling_estimate2}), 
\begin{align*}
	&\int \psi(R)^2 \eta'_{\delta}(R) R^{1-n+\sigma-2\alpha} \mathcal{G}(u,\varphi)^2 
	\leq 6 \delta^{-n+\sigma-2\alpha} \int_{B_{\delta}(0)} (|u|^2 + |\varphi|^2) 
	\leq 6 \delta^{\sigma} \int_{B_1(0)} (|u|^2 + |\varphi|^2) \leq C \delta^{\sigma} 
\end{align*}
for some constant $C = C(\varphi^{(0)}) \in (0,\infty)$, so letting $\delta \downarrow 0$ in (\ref{lemma3_4_eqn21}) using the monotone convergence theorem, 
\begin{align*}
	\int \psi(R)^2 R^{-n+\sigma-2\alpha} \mathcal{G}(u,\varphi)^2 
	\leq \frac{36}{\sigma^2} \int \psi(R)^2 R^{2-n+\sigma} |D_R (u/R^{\alpha})|^2 + C \int_{B_1(0)} \mathcal{G}(u,\varphi)^2
\end{align*}
for $C = C(n,m,\alpha,\gamma,\sigma) \in (0,\infty)$.  
Hence by Theorem~\ref{thm6_2}(a) 
\begin{align*}
	\int_{B_{\gamma}(0)} R^{-n+\sigma-2\alpha} \mathcal{G}(u,\varphi)^2 \leq C \int_{B_1(0)} \mathcal{G}(u,\varphi)^2 
\end{align*}
for $C = C(n,m,\varphi^{(0)},\alpha,\gamma,\beta,\sigma) \in (0,\infty)$.
\end{proof}

For the next two corollaries, we will consider a point $Z = (\xi,\zeta) \in \Sigma_u \cap B_{1/2}(0)$ such that $\mathcal{N}_u(Z) \geq \alpha$.  We will need the following:  Let $u \in \mathcal{F}_{\varepsilon_0}^{\text{Dir}}(\varphi^{(0)})$ and $\varphi \in \Phi_{\varepsilon_0}(\varphi^{(0)})$.  Observe that for $\vartheta \in (0,1)$ and $X = (x,y) \in B_1(0) \setminus \{0\} \times {\mathbb R}^{n-2}$, we have that $\varphi = \{ \pm \varphi_1 \}$ on $B^2_{\vartheta |x|}(x) \times \mathbb{R}^{n-2}$ for some harmonic single-valued function $\varphi_1$.  By Taylor's theorem applied to $\varphi_1$ and the homogeneity of $\varphi$, there is a $\vartheta = \vartheta(\varphi^{(0)}) \in (0,1)$ such that provided $\varepsilon_0$ is sufficiently small, if $|\xi| \leq \vartheta |x|$ then 
\begin{equation} \label{lemma3_9_eqn0}
	\mathcal{G}(u(X),\varphi(X-Z)) = \mathcal{G}(u(X),\varphi(X) - D_x \varphi(X) \cdot \xi) + \mathcal{R}  
\end{equation}
for $\mathcal{R}$ such that $|\mathcal{R}| \leq C |x|^{\alpha-2} |\xi|^2$ for some $C = C(\varphi^{(0)}) \in (0,\infty)$ and 
\begin{equation} \label{lemma3_9_eqn1}
	\mathcal{G}(u(X),\varphi(X-Z)) \geq |D_x \varphi(X) \cdot \xi| - \mathcal{G}(u(X),\varphi(X)) - C |x|^{\alpha-2} |\xi|^2 
\end{equation}
for $C = C(\varphi^{(0)}) \in (0,\infty)$.  By the triangle inequality and the fundamental theorem of calculus, for a.e. $X = (x,y) \in B_1(0)$, 
\begin{align} \label{lemma3_9_eqn2}
	|\mathcal{G}(u(X),\varphi(X-Z)) - \mathcal{G}(u(X),\varphi(X))| 
	&\leq \mathcal{G}(\varphi(X-Z),\varphi(X)) \nonumber 
	\\&\leq \left( \int_0^1 |D\varphi(x-t\xi,y)|^2 dt \right)^{1/2} |\xi|.  
\end{align}
Note that by Lemma \ref{coarse_graph} for every $\varepsilon \in (0,1)$ there is a $\delta = \delta(\varepsilon) \in (0,1)$ such that if $u \in \mathcal{F}^{Dir}_{\varepsilon}(\varphi^{(0)})$ then $|\xi| < \delta$ and $\delta(\varepsilon) \rightarrow 0$ as $\varepsilon \downarrow 0$.  Thus by (\ref{lemma3_9_eqn2}) with $\varphi^{(0)}$ in place of $\varphi$, whenever $u \in \mathcal{F}^{Dir}_{\varepsilon}(\varphi^{(0)})$, 
\begin{equation*}
	4^{-n-2\alpha} \int_{B_{1/4}(Z)} \mathcal{G}(u(X),\varphi^{(0)}(X-Z))^2 \leq C \int_{B_1(0)} \mathcal{G}(u(X),\varphi^{(0)}(X))^2 + C |\xi|^2 
	\leq C(\varepsilon + \delta(\varepsilon)) < \varepsilon_0, 
\end{equation*}
where $C = C(n,\varphi^{(0)}) \in (0,\infty)$ and $\varepsilon_0$ is as in the statement of Theorem~\ref{thm6_2}.  Hence Theorem~\ref{thm6_2} holds with $4^{-\alpha} u(Z+X/4)$ in place of $u(X)$.

\begin{corollary} \label{cor6_4}
Let $\varphi^{(0)} \in W^{1,2}(\mathbb{R}^n;\mathcal{A}_2(\mathbb{R}^m))$ be a Dirichlet minimizing two-valued function given by (\ref{varphi0}) for some $c^{(0)} \in \mathbb{C}^m \setminus \{0\}$ and $\alpha \in \{1/2,1,3/2,2,\ldots\}$.  There is a $\varepsilon_0 = \varepsilon_0(n,m,\varphi^{(0)},\alpha) > 0$ such that if $\varphi \in \Phi_{\varepsilon_0}(\varphi^{(0)})$, $u \in \mathcal{F}^{\text{Dir}}_{\varepsilon_0}(\varphi^{(0)})$, $Z \in \Sigma_u \cap B_{1/2}(0)$ and $\mathcal{N}_u(Z) \geq \alpha$, then 
\begin{align*}
	\op{dist}^{2} \, (Z, \{0\} \times \mathbb{R}^{n-2}) + \int_{B_1(0)} \mathcal{G}(u(X),\varphi(X-Z))^2 dX \leq C \int_{B_1(0)} \mathcal{G}(u(X),\varphi(X))^2 dX 
\end{align*}
for $C = C(n,m,\varphi^{(0)},\alpha) \in (0,\infty)$. 
\end{corollary}
\begin{proof} 
We first claim that there is a constant $\delta_1 = \delta_1(\varphi^{(0)}) > 0,$ and given $\r \in (0, 1/4)$, a constant $\varepsilon_{0}(\r, \varphi^{(0)})$ such that for $u \in {\mathcal F}^{Dir}_{\varepsilon_{0}}(\varphi^{(0)})$, $\varphi \in \Phi_{\varepsilon_{0}}(\varphi^{(0)})$, $a \in \mathbb{R}^{2}$ and $Z = (\xi,\eta) \in \mathcal{B}_u \cap B_{1/2}(0)$, 
\begin{equation} \label{lemma3_9_eqn3}
	\mathcal{L}^n \{ X \in B_{\rho}(Z) : \delta_1 |a| |x|^{\alpha-1} \leq |D_x \varphi(X) \cdot a| \} \geq \delta_1 \rho^n. 
\end{equation}
If not, then for every $\delta_1 > 0$ there is a $\rho > 0$ such that with $\varepsilon_j = 1/j$, there exists $u_j \in \mathcal{F}^{\text{Dir}}_{\varepsilon_j}(\varphi^{(0)})$, $\varphi_j \in \Phi_{\varepsilon_j}(\varphi^{(0)})$, $Z_j \in \mathcal{B}_{u_j}$, and $a_j \in S^1$ such that 
\begin{equation*}
	\mathcal{L}^n \{ X \in B_{\rho}(Z_j) : \delta_1 |x|^{\alpha-1} \leq |D_x \varphi_j(X) \cdot a_j| \} < \delta_1 \rho^n. 
\end{equation*}
After passing to a subsequence, $\varphi_j \rightarrow \varphi^{(0)}$ in $C^{0}(B_1(0);\mathcal{A}_2(\mathbb{R}^m))$, $Z_j \rightarrow Z$ for some $Z \in \{0\} \times \mathbb{R}^{n-2}$, and $a_j \rightarrow a$ with $a \in S^1$ such that 
\begin{equation} \label{lemma3_9_eqn4}
	\mathcal{L}^n \{ X \in B_{\rho}(Z) : \delta_1 |x|^{\alpha-1} \leq |D_x \varphi^{(0)}(X) \cdot a| \} < \delta_1 \rho^n.
\end{equation}
Thus we have shown that, assuming (\ref{lemma3_9_eqn3}) is false, for every $\delta_1 > 0$ there is a $\rho > 0$, $Z \in \{0\} \times \mathbb{R}^{n-2}$, and $a \in S^1$ such that (\ref{lemma3_9_eqn4}) holds.  By translating and rescaling, we may suppose $\rho = 1$ and $Z = 0$.  Thus for $\delta_1 = 1/j$, there is an $a_j \in S^1$ such that 
\begin{equation*}
	\mathcal{L}^n \{ X \in B_1(0) : (1/j) |x|^{\alpha-1} \leq |D_x \varphi^{(0)}(X) \cdot a_j| \} < (1/j).
\end{equation*}
After passing to a subsequence, $a_j \rightarrow a$ with $a \in S^1$ such that 
\begin{equation*}
	D_x \varphi^{(0)}(X) \cdot a = 0 \text{ a.e. in } B_1(0),
\end{equation*}
which is obviously false since $\varphi^{(0)}(x,y) = \op{Re}(c^{(0)} (x_1+ix_2)^{\alpha})$. \\

Let now $a = \xi$ in (\ref{lemma3_9_eqn3}) and use (\ref{lemma3_9_eqn1}) to obtain that for some set $S \subset B_{\rho}(Z)$ with $\mathcal{L}^n(S) \geq \delta_1 \rho^n$, 
\begin{align*}
	\int_S |x|^{2\alpha-2} |\xi|^2 \leq {}& \delta_1^{-2} \int_{B_{\rho}(Z)} |D_x \varphi(X) \cdot \xi|^2 
	\\ \leq {}& 3\delta_1^{-2} \int_{B_{\rho}(Z)} \mathcal{G}(u(X),\varphi(X-Z))^2 dX + 3\delta_1^{-2} \int_{B_1(0)} \mathcal{G}(u(X),\varphi(X))^2 dX 
		\\& + 3C \delta_1^{-2} \int_{B_{\rho}(Z) \cap \{|x| \geq |\xi|/\vartheta\}} |x|^{2\alpha-4} |\xi|^4 
		+ \delta_1^{-2} \int_{B_{\rho}(Z) \cap \{|x| \leq |\xi|/\vartheta\}} |D_x \varphi(X)|^2 |\xi|^2 
\end{align*}
for all $\rho \in (0,1/4)$ and for $C = C(\varphi^{(0)}) \in (0,\infty)$.  For some $\kappa = \kappa(n)$, $\mathcal{L}^{n}(B^2_{\kappa \delta_1^{1/2} \rho}(0) \times B^{n-2}_{\rho}(0)) < \delta_1 \rho^n/2$ and thus $\mathcal{L}^n(\{ (x,y) \in S : |x| \geq \kappa \delta_1^{1/2} \rho \} \geq \delta_1 \rho^n/2$.  Hence 
\begin{align} \label{lemma3_9_eqn5}
	\rho^{n+2\alpha-2} |\xi|^2 \leq {}& C \int_{B_{\rho}(Z)} \mathcal{G}(u(X),\varphi(X-Z))^2 dX + C \int_{B_1(0)} \mathcal{G}(u(X),\varphi(X))^2 dX \nonumber
		\\& + C \int_{B_{\rho}(Z) \cap \{|x| \geq |\xi|/\vartheta\}} |x|^{2\alpha-4} |\xi|^4 
		+ C \int_{B_{\rho}(Z) \cap \{|x| \leq |\xi|/\vartheta\}} |D_x \varphi(X)|^2 |\xi|^2 
\end{align}
for some constant $C = C(n,\varphi^{(0)},\alpha) \in (0,\infty)$.  

We need to bound the terms on the right-hand side of (\ref{lemma3_9_eqn5}).  For the first term on the right-hand side of (\ref{lemma3_9_eqn5}), by Corollary~\ref{cor6_3} with $\sigma = 1/2$ and (\ref{lemma3_9_eqn2}), 
\begin{align*}
	&\hspace{-.5in}\rho^{-n-2\alpha+1/2} \int_{B_{\rho}(Z)} \mathcal{G}(u(X),\varphi(X-Z))^2 dX \nonumber 
	\leq C \int_{B_1(0)} \mathcal{G}(u(X),\varphi(X-Z))^2 dX \nonumber 
	\\&\hspace{.5in}\leq C \int_{B_1(0)} \mathcal{G}(u(X),\varphi(X))^2 dX + C |\xi|^2 \int_{B_1(0)} \int_0^1 |D\varphi(x-t\xi,y)|^2 dt dx dy 
\end{align*}
for $C = C(n,m,\varphi^{(0)},\alpha) \in (0,\infty)$.  Using the change of variable $x' = x-t\xi$, 
\begin{align} \label{lemma3_9_eqn6}
	&\hspace{-.5in}\int_{B_1(0)} \int_0^1 |D\varphi(x-t\xi,y)|^2 dt dx dy 
	\leq C \sup_{\partial B^2_1(0) \times \mathbb{R}^{n-2}} |D\varphi|^2 \int_{B^2_1(0)} \int_0^1 |x-t\xi|^{2\alpha-2} dt dx \nonumber 
	\\&\hspace{1in}\leq C \sup_{\partial B^2_1(0) \times \mathbb{R}^{n-2}} |D\varphi|^2 \int_0^1 \int_{B^2_{1+t|\xi|}(0)} |x'|^{2\alpha-2} dx' dt  
	\leq C \sup_{\partial B^2_1(0) \times \mathbb{R}^{n-2}} |D\varphi|^2 
\end{align}
for $C = C(n,\alpha) \in (0,\infty)$.  Hence,  
\begin{equation} \label{lemma3_9_eqn7}
	\rho^{-n-2\alpha+1/2} \int_{B_{\rho}(Z)} \mathcal{G}(u(X),\varphi(X-Z))^2 dX \leq C \int_{B_1(0)} \mathcal{G}(u(X),\varphi(X))^2 dX + C |\xi|^2 
\end{equation}
for some constant $C = C(n,m,\varphi^{(0)},\alpha) \in (0,\infty)$.  For the third term on the right-hand side of (\ref{lemma3_9_eqn5}), by direct computation considering the cases where $\alpha = 1/2$, $\alpha = 1$, and $\alpha > 1$ separately 
\begin{equation} \label{lemma3_9_eqn8}
	\int_{B_{\rho}(Z) \cap \{|x| \geq |\xi|/\vartheta\}} |x|^{2\alpha-4} |\xi|^4 \leq C (\rho^{2\alpha-5/2} + |\xi|^{2\alpha-5/2}) \rho^{n-2} |\xi|^4
\end{equation}
for some constant $C = C(n,\alpha, \varphi^{(0)}) \in (0,\infty).$  In fact, in the cases where $\alpha = 1/2$ or $\alpha > 1$ we can bound the left-hand side of (\ref{lemma3_9_eqn8}) by $C (\rho^{2\alpha-2} + |\xi|^{2\alpha-2}) \rho^{n-2} |\xi|^4$ with $C = C(n, \a)$, and in the case where $\alpha = 1$ we can bound the left-hand side of (\ref{lemma3_9_eqn8}) by $C \rho^{n-2} |\xi|^4 |\log (|\xi| + \r) \leq C \rho^{n-2} |\xi|^{7/2}$ with $C = C(n, \a, \varphi^{(0)})$ (since $\r > (\vartheta^{-1}-1)|\xi|$, or else $B_{\r}(Z) \cap 
\{|x| >|\xi|/\vartheta\}  = \emptyset$).  For the last term on the right-hand side of (\ref{lemma3_9_eqn5}),  
\begin{equation} \label{lemma3_9_eqn9}
	\int_{B_{\rho}(Z) \cap \{|x| \leq |\xi|/\vartheta\}} |D_x \varphi(X)|^2 |\xi|^2 
	\leq C \rho^{n-2} \int_{B^2_{|\xi|/\vartheta}(0)} |x|^{2\alpha-2} |\xi|^2
	\leq C \rho^{n-2} |\xi|^{2\alpha+2}
\end{equation}
for some constant $C = C(n,\alpha,\varphi^{(0)}) \in (0,\infty)$.  Therefore, by (\ref{lemma3_9_eqn5}), (\ref{lemma3_9_eqn7}), (\ref{lemma3_9_eqn8}), and (\ref{lemma3_9_eqn9}), 
\begin{equation} \label{lemma3_9_eqn10}
	\rho^{n+2\alpha-2} |\xi|^2 
	\leq C \int_{B_1(0)} \mathcal{G}(u(X),\varphi(X))^2 dX 
	+ C (\rho^{3/2} + \rho^{-5/2} |\xi|^2 + \rho^{-2\alpha} |\xi|^{2\alpha-1/2}) \rho^{n+2\alpha-2} |\xi|^2 
\end{equation}
for some  $C = C(n,m,\varphi^{(0)},\alpha) \in (0,\infty)$.  By Lemma \ref{coarse_graph}, given $\tau > 0$ for $\varepsilon_0 = \varepsilon_0(n,m,\varphi^{(0)},\tau)$ sufficiently small, $|\xi| \leq \tau$.  Choosing $\rho$ and $\tau$ small enough that $C (\rho^{3/2} + \rho^{-5/2} \tau^2 + \rho^{-2\alpha} \tau^{2\alpha-1/2}) < 1/2$, (\ref{lemma3_9_eqn10}) yields 
\begin{equation} \label{lemma3_9_eqn11}
	|\xi|^2 \leq C \int_{B_1(0)} \mathcal{G}(u(X),\varphi(X))^2 dX 
\end{equation}
for some constant $C = C(n,m,\varphi^{(0)},\alpha) \in (0,\infty)$.  By (\ref{lemma3_9_eqn7}) and (\ref{lemma3_9_eqn11}), 
\begin{equation*}
	\int_{B_1(0)} \mathcal{G}(u(X),\varphi(X-Z))^2 dX \leq C \int_{B_1(0)} \mathcal{G}(u(X),\varphi(X))^2 dX . \qedhere
\end{equation*}
\end{proof}

\begin{corollary} \label{cor6_5} 
Let $\varphi^{(0)} \in W^{1,2}(\mathbb{R}^n;\mathcal{A}_2(\mathbb{R}^m))$ be a Dirichlet minimizing two-valued function given by (\ref{varphi0}) for some $c^{(0)} \in \mathbb{C}^m \setminus \{0\}$ and $\alpha \in \{1/2,1,3/2,2,\ldots\}$.  Let $\gamma, \tau, \sigma \in (0,1).$ There are $\varepsilon_0 = \varepsilon_0(\varphi^{(0)},\tau,\gamma) \in (0, 1)$ and $\beta_0 = \beta_0(\varphi^{(0)}) \in  (0, 1)$ such that if $\varphi \in \Phi_{\varepsilon_0}(\varphi^{(0)})$, $u \in \mathcal{F}^{\text{Dir}}_{\varepsilon_0}(\varphi^{(0)})$, $v$ is as in Lemma \ref{coarse_graph} with $\beta = \beta_0$, and $Z = (\xi,\eta) \in \Sigma_u \cap B_{1/2}(0)$ with $\mathcal{N}_u(Z) \geq \alpha$, then 
\begin{equation*} \label{thm3_1_eqn1}
	\int_{B_{\gamma}(0)} R_Z^{2-n} \left| \frac{\partial (u/R_Z^{\alpha})}{\partial R_Z} \right|^2 
	\leq C \int_{B_1(0)} \mathcal{G}(u,\varphi)^2 
\end{equation*}
where $R_Z = |X-Z|$ and $C = C(n,m,\varphi^{(0)},\alpha,\gamma) \in (0,\infty);$  furthermore,  
\begin{equation*} 
	\int_{B_{\gamma}(0)} \frac{\mathcal{G}(u,\varphi)^2}{|X-Z|^{n-1-\sigma}} 
	+ \int_{B_{\gamma}(0) \cap \{ |x| > \tau \}} \frac{|v(X,\varphi(X)) - D_x \varphi(X) \cdot \xi|^2}{|X-Z|^{n+2\alpha-\sigma}} 
	\leq C_{1} \int_{B_1(0)} \mathcal{G}(u,\varphi)^2 
\end{equation*}
where $C_{1} = C_{1}(n,m,\varphi^{(0)},\alpha,\gamma,\sigma) \in (0,\infty).$ In particular, the constants $C, C_{1}$ are both independent of $\tau$. 
\end{corollary}
\begin{proof}
By Theorem~\ref{thm6_2} and Corollary~\ref{cor6_4}, the first conclusion obviously holds.  By (\ref{lemma3_9_eqn2}), 
\begin{align} \label{thm3_1_eqn3}
	&\hspace{-.5in}\int_{B_{\gamma}(0)} \frac{\mathcal{G}(u(X),\varphi(X))^2}{|X-Z|^{n-1-\sigma}} \nonumber
	\\&\leq 2 \int_{B_{1/4}(Z)} \frac{\mathcal{G}(u(X),\varphi(X-Z))^2}{|X-Z|^{n-1-\sigma}} 
		+ 2 \int_{B_{1/4}(Z)} \int_0^1 \frac{|D\varphi(x-t\xi,y)|^2 |\xi|^2}{|X-Z|^{n-1-\sigma}} dt dx dy \nonumber 
		\\&\hspace{5mm} + 4^{n-1-\sigma} \int_{B_1(0)} \mathcal{G}(u(X),\varphi(X))^2. 
\end{align}
We need to bound the terms on the right-hand side of (\ref{thm3_1_eqn3}).  For the first term on the right-hand side of (\ref{thm3_1_eqn3}), observe that by Theorem~\ref{thm6_2}, (\ref{lemma3_9_eqn2}), (\ref{lemma3_9_eqn6}) and Corollary~\ref{cor6_4}, 
\begin{align} \label{thm3_1_eqn4}
	&\hspace{-.5in}\int_{B_{1/4}(Z)} \frac{\mathcal{G}(u(X),\varphi(X-Z))^2}{|X-Z|^{n+2\alpha-\sigma}} \leq C \int_{B_{1/2}(Z)} \mathcal{G}(u(X),\varphi(X-Z))^2 \nonumber
	\\&\leq C \int_{B_1(0)} \mathcal{G}(u(X),\varphi(X))^2 + C \int_{B_{1/2}(Z)} \int_0^1 |D\varphi(x-t\xi,y)|^2 |\xi|^2 dt dx dy \nonumber
	\\&\leq C \int_{B_1(0)} \mathcal{G}(u(X),\varphi(X))^2 + C \int_{B_{1/2}(Z)} \int_0^1 |x-t\xi|^{2\alpha-2} |\xi|^2 dt dx dy \nonumber 
	\\&\leq C \int_{B_1(0)} \mathcal{G}(u(X),\varphi(X))^2 + C |\xi|^2  \leq C \int_{B_1(0)} \mathcal{G}(u(X),\varphi(X))^2 
\end{align}  
for $C = C(n,m,\varphi^{(0)},\alpha) \in (0,\infty)$.  For the second term on the right-hand side of (\ref{thm3_1_eqn3}), first observe that if $\alpha \geq 1$ then by Corollary~\ref{cor6_4}, 
\begin{align} \label{thm3_1_eqn5}
	\int_{B_{1/4}(Z)} \int_0^1 \frac{|x-t\xi|^{2\alpha-2} |\xi|^2}{|X-Z|^{n-1-\sigma}} dt dx dy 
	\leq \int_{B_{1/4}(Z)} \int_0^1 \frac{|\xi|^2}{|X-Z|^{n-1-\sigma}} dt dx dy  
	= C |\xi|^2 \leq C \int_{B_1(0)} \mathcal{G}(u,\varphi)^2. 
\end{align}
for $C = C(n,m,\varphi^{(0)},\alpha) \in (0,\infty)$.  If instead $\alpha = 1/2$, then again by 
Corollary~\ref{cor6_4}, 
\begin{align} \label{thm3_1_eqn6}
	\int_{B_{1/4}(Z)} \int_0^1 \frac{|x-t\xi|^{2\alpha-2} |\xi|^2}{|X-Z|^{n-1-\sigma}} dt dx dy 
	\leq {}& \int_0^1 \int_{B_{1/4}(Z)} \frac{|\xi|^2}{|x-t\xi| |x-\xi|^{1-\sigma/2} |y-\zeta|^{n-2-\sigma/2}} dx dy dt \nonumber\\
	&\hspace{-1.5in} \leq \int_0^1 \int_{B^{n-2}_{1/4}(\zeta)} \int_{\{|x-t\xi| \leq |x-\xi|\}} \frac{|\xi|^2}{|x-t\xi|^{2-\sigma/2} |y-\zeta|^{n-2-\sigma/2}} dx dy dt 
		\nonumber \\&\hspace{-1in}+ \int_0^1 \int_{B^{n-2}_{1/4}(\zeta)} \int_{\{|x-t\xi| \geq |x-\xi|\}} \frac{|\xi|^2}{|x-\xi|^{2-\sigma/2} |y-\zeta|^{n-2-\sigma/2}} 
		dx dy dt \nonumber
	\\&\hspace{1in} \leq C |\xi|^2 \leq C \int_{B_1(0)} \mathcal{G}(u,\varphi)^2. 
\end{align}
for some constant $C = C(n,m,\alpha,\sigma) \in (0,\infty)$.  Putting (\ref{thm3_1_eqn3}), (\ref{thm3_1_eqn4}), (\ref{thm3_1_eqn5}), and (\ref{thm3_1_eqn6}) together yield 
\begin{equation*}
	\int_{B_{\gamma}(0)} \frac{\mathcal{G}(u(X),\varphi(X))^2}{|X-Z|^{n-1-\sigma}} \leq C \int_{B_1(0)} \mathcal{G}(u(X),\varphi(X))^2  
\end{equation*}
for $C = C(n,m,\varphi^{(0)},\alpha,\gamma,\sigma) \in (0,\infty)$.  By (\ref{lemma3_9_eqn0}), assuming $\varepsilon_0$ is small enough that $|\xi| \leq \vartheta \tau$ and using the fact that the graphs of $\varphi^{(0)}$, $\varphi$, and $u$ are embedded in $B_{\gamma}(0) \cap \{|x| \geq \tau \}$, 
\begin{align} \label{thm3_1_eqn7}
	&\hspace{-.5in}\int_{B_{\gamma}(0) \cap \{ |x| > \tau \}} \frac{|v(X,\varphi(X)) - D_x \varphi(X) \cdot \xi|^2}{|X-Z|^{n+2\alpha-\sigma}} \nonumber 
	\\&\leq C \int_{B_{1/4}(Z)} \frac{\mathcal{G}(u(X),\varphi(X-Z))^2}{|X-Z|^{n+2\alpha-\sigma}}  
		+ C \tau^{2\alpha-4} |\xi|^4 \int_{B_{1/4}(Z) \cap \{|x| > \tau\}} \frac{1}{|X-Z|^{n+2\alpha-\sigma}}  \nonumber 
		\\&\hspace{5mm} + C 4^{n+2\alpha-\sigma} \int_{(B_{\gamma}(0) \setminus B_{1/4}(Z)) \cap \{ |x| > \tau \}} (|v(X,\varphi(X))|^2 + |x|^{2\alpha-2} |\xi|^2) 
\end{align}
for $C = C(\varphi^{(0)}) \in (0,\infty)$, so applying (\ref{thm3_1_eqn4}), Theorem~\ref{thm6_2}, and Corollary~\ref{cor6_4}, 
\begin{align} \label{thm3_1_eqn8}
	&\hspace{-.5in}\int_{B_{\gamma}(0) \cap \{ |x| > \tau \}} \frac{|v(X,\varphi(X)) - D_x \varphi(X) \cdot \xi|^2}{|X-Z|^{n+2\alpha-\sigma}} \nonumber 
	\\&\leq C \tau^{2\alpha-4} |\xi|^4 \int_{B_{1/4}(Z) \cap \{|x| > \tau\}} \frac{1}{|X-Z|^{n+2\alpha-\sigma}}  
		+ C \int_{B_1(0)} \mathcal{G}(u(X),\varphi(X))^2 
\end{align}
for $C = C(n,m,\varphi^{(0)},\alpha,\gamma,\sigma) \in (0,\infty)$.  By Corollary~ \ref{cor6_4}, $|\xi| < C\varepsilon_0$ for some $C = C(n,m,\varphi^{(0)},\alpha) \in (0,\infty)$.  Take $\varepsilon_0 < \tau^2$ so that $|X-Z| \geq \tau/2$ if $X = (x,y)$ with $|x| > \tau$ and $|\xi|^2/\tau^4 < 1$ and thus 
\begin{align} \label{thm3_1_eqn9}
	\tau^{2\alpha-4} |\xi|^4 \int_{B_{1/4}(Z) \cap \{|x| > \tau\}} \frac{1}{|X-Z|^{n+2\alpha-\sigma}} \nonumber 
	&\leq \frac{|\xi|^4}{\tau^4} \int_{B_{1/4}(Z) \cap \{|x| > \tau\}} \frac{1}{|X-Z|^{n-\sigma}} \nonumber 
	\\&\leq C|\xi|^2 
	\leq C \int_{B_1(0)} \mathcal{G}(u(X),\varphi(X))^2 
\end{align}
for $C = C(n,m,\alpha,\varphi^{(0)}, \s) \in (0,\infty)$, where the last inequality follows from Corollary~ \ref{cor6_4}.  By (\ref{thm3_1_eqn8}) and (\ref{thm3_1_eqn9}), 
\begin{align*}
	\int_{B_{\gamma}(0) \cap \{ |x| > \tau \}} \frac{|v(X,\varphi(X)) - D_x \varphi(X) \cdot \xi|^2}{|X-Z|^{n+2\alpha-\sigma}} 
	\leq C \int_{B_1(0)} \mathcal{G}(u(X),\varphi(X))^2 
\end{align*}
for $C = C(n,m,\varphi^{(0)},\alpha,\gamma,\sigma) \in (0,\infty)$.
\end{proof}

\begin{corollary} \label{cor6_6} 
Let $\varphi^{(0)} \in W^{1,2}(\mathbb{R}^n;\mathcal{A}_2(\mathbb{R}^m))$ be a Dirichlet minimizing two-valued function given by (\ref{varphi0}) for some $c^{(0)} \in \mathbb{C}^m \setminus \{0\}$ and $\alpha \in \{1/2,1,3/2,2,\ldots\}$.  Let  $\sigma, \tau, \delta \in (0,1)$.  There are $\varepsilon_0 = \varepsilon_0(\varphi^{(0)},\tau) > 0$ and $\beta_0 = \beta_0(\varphi^{(0)}) > 0$ such that if $\varphi \in \Phi_{\varepsilon}(\varphi^{(0)})$ and $u \in \mathcal{F}^{\text{Dir}}_{\varepsilon}(\varphi^{(0)})$ with $\varepsilon \leq \min\{\varepsilon_0,\delta\}$, and if 
\begin{equation*} 
	B_{\delta}(0,y) \cap \{ Z \in B_1(0) : \mathcal{N}_u(Z) \geq \mathcal{N}_{\varphi}(0) = \alpha \} \neq \emptyset \text{ for each } y \in B^{n-2}_{1/2}(0), \;\;then
\end{equation*}

\begin{equation*}
\int_{B_{1/2}(0)} \frac{\mathcal{G}(u,\varphi)^2}{r_{\delta}^{1-\sigma}} 
	\leq C \int_{B_1(0)} \mathcal{G}(u,\varphi)^2,
\end{equation*}	
where $C = C(n,m,\varphi^{(0)},\alpha,\sigma) \in (0,\infty)$ and $r_{\delta}(X) = \max\{|x|,\delta\}$ where $X = (x, y) \in {\mathbb R}^{2} \times {\mathbb R}^{n-2}$.
\end{corollary}
\begin{proof} 
By hypothesis, for every $z \in B^{n-2}_{1/2}(0)$, there is a $Z = (\xi,\eta) \in B_{\delta}(0,z)$ such that $\mathcal{N}_{u}(Z) \geq \alpha$.  By Corollary~\ref{cor6_5}, for all $\rho \in (\delta,1/4]$, 
\begin{equation} \label{cor3_2_eqn1}
	\rho^{-n+1+\sigma} \int_{B_{\rho}(0,z)} \mathcal{G}(u,\varphi)^2 \leq C \int_{B_1(0)} \mathcal{G}(u,\varphi)^2. 
\end{equation}
We can cover $B_{1/2}(0) \cap B^2_{\rho/2}(0) \times \mathbb{R}^{n-2}$ by $N \leq C(n,k) \rho^{2-n}$ balls $B_{\rho}(0,z_j)$ with $z_j \in B^{n-2}_{1/2}(0)$.  Use the fact that (\ref{cor3_2_eqn1}) 
holds for $z = z_j$ and sum over $j$ to obtain 
\begin{equation} \label{cor3_2_eqn3}
	\rho^{-1+\sigma} \int_{B_{1/2}(0) \cap B_{\rho/2}^2(0) \times \mathbb{R}^{n-2}} \mathcal{G}(u,\varphi)^2 \leq C \int_{B_1(0)} \mathcal{G}(u,\varphi)^2
\end{equation}
for $\rho \in (\delta,1/4]$.  Replace $\sigma$ with $\sigma/2$ in (\ref{cor3_2_eqn3}), then multiply by $\rho^{-1+\sigma/2}$ and integrate with respect to $\rho \in (\delta,1/4]$ to get the desired conclusion. 
\end{proof}

Now we want to verify that Theorem~\ref{thm6_2} and Corollaries~\ref{cor6_3}--\ref{cor6_6} all hold if we suppose $u \in \mathcal{F}^{\text{Harm}}_{\varepsilon_0}(\varphi^{(0)})$ and $\alpha \geq 3/2.$ 
The results and their proofs in this case are identical to what is given above, except for the following two points concerning Corollary~\ref{cor6_5}:\\ 
\noindent
(i) Since $\alpha \geq 1,$  we have the stronger estimate 
\begin{equation*}
	\int_{B_{\gamma}(0)} \frac{\mathcal{G}(u,\varphi)^2}{|X-Z|^{n-\sigma}} \leq C \int_{B_1(0)} \mathcal{G}(u,\varphi)^2 
\end{equation*}
for $C = C(n,m,\varphi^{(0)},\alpha,\gamma,\sigma) \in (0,\infty).$ (The weaker estimate given in the statement of Corollary~\ref{cor6_5} however still suffices for the proofs of the main theorems.) \\
\noindent 
(ii) We need to modify slightly the proof of the estimate  
\begin{equation} \label{thm3_1h_eqn1}
	\int_{B_{\gamma}(0) \cap \{ |x| > \tau \}} \frac{|v(X,\varphi(X)) - D_x \varphi(X) \cdot \xi|^2}{|X-Z|^{n+2\alpha-\sigma}} \leq C \int_{B_1(0)} \mathcal{G}(u,\varphi)^2 
\end{equation}
 to account for the fact that the graphs of $\varphi^{(0)}$, $\varphi$, and $u$ might have transverse self-intersections in $B_{\gamma}(0) \setminus \{ |x| \geq \tau \}$.  Recall that for $\vartheta \in (0,1)$ and $X = (x,y) \in B_1(0)$, $\varphi = \{ \pm \varphi_1 \}$ on $B^2_{\vartheta |x|}(x) \times \mathbb{R}^{n-2}$ for some harmonic single-valued function $\varphi_1$.  Instead of (\ref{lemma3_9_eqn0}), Taylor's theorem applied to $\varphi_1$ and the homogeneity of $\varphi$ implies that there is a $\vartheta = \vartheta(\varphi^{(0)}) \in (0,1)$ such that if $|\xi| \leq \vartheta |x|$ then 
\begin{equation} \label{thm3_1h_eqn2}
	|\varphi(X) + v(X,\varphi(X)) - \varphi(X-Z)| = |v(X,\varphi(X)) - D_x \varphi(X) \cdot \xi)| + \mathcal{R}, 
\end{equation}
for $\mathcal{R}$ such that $|\mathcal{R}| \leq C |x|^{\alpha-2} |\xi|^2$ for some $C = C(\varphi^{(0)}) \in (0,\infty)$.  Recall that $u(X) = \varphi(X)+v(X,\varphi(X))$ in the sense that (\ref{v_defn1}) holds.  Thus assuming $\varepsilon_0$ is small enough that $|\xi| \leq \vartheta \tau$, 
\begin{align} \label{thm3_1h_eqn3}
	&\hspace{-.7in}\int_{B_{\gamma}(0) \cap \{ |x| > \tau \}} \frac{|v(X,\varphi(X)) - D_x \varphi(X) \cdot \xi|^2}{|X-Z|^{n+2\alpha-\sigma}} \nonumber\\ 
	&\leq 2 \int_{B_{1/4}(Z) \cap \{ |x| > \tau \}} \frac{|\varphi(X) + v(X,\varphi(X)) - \varphi(X-Z)|^2}{|X-Z|^{n+2\alpha-\sigma}} \nonumber
		\\&\hspace{.5in}+ C \tau^{2\alpha-4} |\xi|^4 \int_{B_{1/4}(Z) \cap \{|x| > \tau\}} \frac{1}{|X-Z|^{n+2\alpha-\sigma}}  \nonumber 
		\\&\hspace{1in}+ 2 \cdot 4^{n+2\alpha-\sigma} \int_{(B_1(0) \setminus B_{1/4}(Z)) \cap \{ |x| > \tau \}} (|v(X,\varphi(X))|^2 + C |x|^{2\alpha-2} |\xi|^2).  
\end{align}
It suffices to show that whenever $r_0^2 + |y_0-\zeta|^2 < 1/16$, 
\begin{equation} \label{thm3_1h_eqn4}
	\int_{\partial B^2_{r_0}(\xi) \times \{y_0\}} |\varphi(X) + v(X,\varphi(X)) - \varphi(X-Z)|^2 
	\leq C \int_{\partial B^2_{r_0}(\xi) \times \{y_0\}} \mathcal{G}(u(X),\varphi(X-Z))^2 
\end{equation}
for a constant $C = C(n, m, \varphi^{(0)}, \a) \in (0, \infty);$ in particular, $C$ is independent of $r_{0}$ and $y_{0}$. But this follows by computations similar to those in the proof of Lemma~\ref{separation_lemma_h}.

\section{Asymptotic decay for blow-ups} \label{lineartheory_section}
\setcounter{equation}{0}
Fix $\varphi^{(0)}$ as in (\ref{varphi0}). Our main result in this section (Lemma~\ref{lemma4_14}) is an $L^{2}$ decay estimate for a class of functions $w \in C^2(\op{graph} \varphi^{(0)} |_{B_1(0) \setminus \{0\} \times \mathbb{R}^{n-2}};\mathbb{R}^m) \cap L^2(\op{graph} \varphi^{(0)} |_{B_1(0)};\mathbb{R}^m, p^{\star}{\mathcal L}^{n})$  satisfying certain integral hypotheses and the property that the associated two-valued function $w(X, \varphi^{(0)}(X))$ is harmonic in $B_{1}(0) \setminus \{0\} \times {\mathbb R}^{n-2}$. Here $p$ denotes the restriction to ${\rm graph} \, \varphi^{(0)}$ of the orthogonal projection $p \, : \, {\mathbb R}^{n+m} \to {\mathbb R}^{n}$ and ${\mathcal L}^{n}$ denotes the Lebesgue measure on ${\mathbb R}^{n}$. This result is needed for the proof of Lemma \ref{lemma1}, which we shall give in the next section, in which such $w$ arise as ``blow-ups'' of certain sequences of two-valued Dirichlet energy minimizing or $C^{1, \mu}$ harmonic functions converging in $L^{2}$ to $\varphi^{(0)}$.   

We begin with the following classification result for such $w$ which are also homogeneous of degree $\a$ (= the degree of homogeneity of $\varphi^{(0)}$).

\begin{lemma} \label{lemma4_2} 
Let $\sigma \in (0,1)$.  Let $\varphi^{(0)} : \mathbb{R}^n \rightarrow \mathcal{A}_2(\mathbb{R}^m)$ be a two-valued function given by (\ref{varphi0}) for some $c^{(0)} \in \mathbb{C}^m \setminus \{0\}$ and $\alpha \in \{1/2,1,3/2,2,\ldots\}$.  Let $w \in C^2(\op{graph} \varphi^{(0)} |_{B_1(0) \setminus \{0\} \times \mathbb{R}^{n-2}};\mathbb{R}^m) \cap L^2(\op{graph} \varphi^{(0)} |_{B_1(0)};\mathbb{R}^m, p^{\star} \, {\mathcal L}^{n})$ be such that $w(X,\varphi^{(0)}(X))$ ($\equiv \{w(X, +\varphi_{1}^{(0)}(X)), w(X, -\varphi_{1}^{(0)}(X))\}$, where $\varphi^{(0)}(X) = \{\pm\varphi_{1}^{(0)}(X)\}$) is a homogeneous degree $\alpha$, harmonic, symmetric 2-valued function on $B_1(0) \setminus \{0\} \times \mathbb{R}^{n-2}$ satisfying 
\begin{equation} \label{lemma4_2_decay}
	\int_{B_1(0)} \frac{|w(re^{i\theta},y,\varphi^{(0)}(re^{i\theta},y)) - \kappa(re^{i\theta},y,\varphi^{(0)}(re^{i\theta},y))|^2}{r^{2+2\alpha-\sigma}} < \infty 
\end{equation}
for some 
\begin{equation*}
	\kappa(re^{i\theta},y,\varphi^{(0)}(re^{i\theta},y)) = \kappa_1(r,y) D_1 \varphi^{(0)}(re^{i\theta},y) + \kappa_2(r,y) D_2 \varphi^{(0)}(re^{i\theta},y)
\end{equation*}
with $\kappa_1, \kappa_2 \in L^{\infty}(B_1(0);\mathbb{R}^m)$.  If $\alpha = 1/2$, further assume that 
\begin{equation} \label{lemma4_2_alphaishalf}
	\lim_{r \downarrow 0} \frac{\partial^2}{\partial r \partial y_{p}} \int_{S^1} r w(re^{i\theta},y,\varphi^{(0)}(re^{i\theta},y)) D_i\varphi^{(0)}(re^{i\theta},y) d\theta = 0  
\end{equation} 
for $i = 1,2$.  Then
\begin{align} \label{lemma4_2_eqn1}
	&\hspace{-.2in}w(re^{i\theta},y,\varphi^{(0)}(re^{i\theta},y)) 
	\nonumber \\&\hspace{.5in} = a_0 r^{\alpha} \cos(\alpha \theta) + b_0 r^{\alpha} \sin(\alpha \theta) 
	+ \sum_{j=1}^{n-2} \left( a_j D_1 \varphi^{(0)}(re^{i\theta},y) y_j + b_j D_2 \varphi^{(0)}(re^{i\theta},y) y_j \right) 
\end{align}
for some $a_0,b_0 \in \mathbb{R}^m$ and $a_j,b_j \in \mathbb{R}$ for $j = 1,\ldots,n-2$. 
\end{lemma}
\begin{proof} 
We follow the argument of  \cite{SimonCTC}, Lemma~4.2 with modifications necessary to account for the fact that the degree of homogeneity $\a$ takes various values; in \cite{SimonCTC}, $\a$ is always one. Note in particular the case $\a = 1/2$, which requires the additional hypothesis (\ref{lemma4_2_alphaishalf}). Note also that if $n=2$ or $n=3$, certain steps of the argument below need to be altered in obvious ways  which we shall not comment on any further. 

Recall from Remark \ref{graph_rmk} that the graph of $\varphi^{(0)}$ is a branched submanifold parameterized by $(re^{i\theta},y,\op{Re}(c^{(0)} r^{\alpha} e^{i\a\theta}))$ for $r \geq 0$, $\theta \in [0,4\pi]$, and $y \in \mathbb{R}^{n-2}$.  Let $\phi_1,\phi_2,\ldots \, : \, {\mathbb R} \to {\mathbb R}^{m}$ be the $4\pi$ periodic functions whose restriction to $[0, 4\pi]$ is an ($L^2$) orthonormal basis for $L^2([0,4\pi]; {\mathbb R}^{m})$ consisting of eigenfunctions such that $\phi_l''(\theta) + \lambda_l \phi_l(\th) = 0$ in $(0,4\pi)$, where the labeling is such that $\lambda_1 \leq \lambda_2 \leq \ldots$.  We can regard the $\phi_k$ as functions on $\op{graph} \varphi^{(0)}$ that are independent of $r$ and $y$. 
Let $l_0$ be the smallest positive integer such that $\lambda_{l_0} = (\alpha-1)^2$.  We may assume the span of $\phi_{l_0}$ and $\phi_{l_0+1}$ equals the span of $D_1 \varphi^{(0)}(e^{i\theta})$ and $D_2 \varphi^{(0)}(e^{i\theta})$.  Consider the Fourier coefficients $w_l$ of $w$ given by 
\begin{equation*}
	w_l(r,y) = \frac{1}{4\pi} \int_0^{4\pi} w(re^{i\theta},y,\varphi^{(0)}(re^{i\theta},y)) \phi_l(\theta) d\theta, 
\end{equation*}
Each $w_l \in C^{\infty}((0,\infty) \times \mathbb{R}^{n-2}; {\mathbb R}^{m})$ is homogeneous of degree $\alpha$ single-valued functions satisfying 
\begin{equation} \label{lemma4_2_eqn2}
	\frac{1}{r} \frac{\partial}{\partial r} \left( r \frac{\partial w_l}{\partial r} \right) - \frac{\lambda_l}{r^2} w_l + \Delta_y w_l = 0. 
\end{equation} 
Fix $l$.  By (\ref{lemma4_2_decay}), if $\phi_l$ is in the span of $D_1 \varphi^{(0)}(e^{i\theta})$ and $D_2 \varphi^{(0)}(e^{i\theta})$ then $\lambda_l = (\alpha-1)^2$ and 
\begin{equation} \label{lemma4_2_decay1a}
	\int_{B_1(0)} \frac{|w_l - \kappa_l(r,y)|^2}{r^{2+2\alpha-\sigma}} < \infty, 
\end{equation}
where $\kappa_l = \kappa_l(r,y)$ in $L^{\infty}((0,\infty) \times \mathbb{R}^{n-2})$ is given by 
\begin{equation*}
	\kappa_l(r,y) = \frac{1}{4\pi} \int_0^{4\pi} \kappa(re^{i\theta},y,\varphi^{(0)}(re^{i\theta},y)) \phi_l(\theta) d\theta. 
\end{equation*}
If $\phi_l$ is not in the span of $D_1 \varphi^{(0)}(e^{i\theta})$ and $D_2 \varphi^{(0)}(e^{i\theta})$ then 
\begin{equation} \label{lemma4_2_decay1b}
	\int_{B_1(0)} \frac{|w_l|^2}{r^{2+2\alpha-\sigma}} < \infty. 
\end{equation}
Let $w_l(r,y) = r^{\alpha} \psi(y/r)$ where $\psi(z) = w_l(1,z)$ is in $C^{\infty}(\mathbb{R}^{n-2}; {\mathbb R}^{m})$.  Then (\ref{lemma4_2_eqn2}) becomes 
\begin{equation} \label{lemma4_2_eqn3}
	\Delta_z \psi + z_i z_j D_{z_i z_j} \psi - (2\alpha-1) z_i D_{z_i} \psi + (\alpha^2 - \lambda_l) \psi = 0.  
\end{equation}
By (\ref{lemma4_2_decay1a}) and (\ref{lemma4_2_decay1b}), if $\phi_l$ is in the span of $D_1 \varphi^{(0)}(e^{i\theta})$ and $D_2 \varphi^{(0)}(e^{i\theta})$ then  
\begin{equation} \label{lemma4_2_decay2a}
	\int_{S^{n-3}} \int_1^{\infty} r^{1-\sigma} \left| \frac{\psi(r\omega)}{r} - \kappa(r\omega) \right|^2 dr d\omega < \infty 
\end{equation}
for some $\kappa(z)$ in $L^{\infty}((0,\infty) \times \mathbb{R}^{n-2})$ and otherwise 
\begin{equation} \label{lemma4_2_decay2b}
	\int_{S^{n-3}} \int_1^{\infty} r^{-1-\sigma} |\psi(r\omega)|^2 dr d\omega < \infty. 
\end{equation}
By (\ref{lemma4_2_eqn3}), $v = D_p \psi$, where $p \in \{1,2,\ldots,n\}$, satisfies 
\begin{equation} \label{lemma4_2_eqn4}
	\Delta_z v + z_i z_j D_{z_i z_j} v - (2\alpha-3) z_i D_{z_i} v + ((\alpha-1)^2 - \lambda_l) v = 0.  
\end{equation}
Write $\psi$ in terms of its Fourier expansion 
\begin{equation*}
	\psi(r\omega) = \sum_{k=0}^{\infty} \gamma_k(r) \chi_k(\omega) 
\end{equation*}
for $r \geq 0$ and $\omega \in S^{n-3}$, where $\gamma_k : (0,\infty) \rightarrow \mathbb{R}$ and $\chi_k$ denote eigenfunctions for the Laplacian on $S^{n-3}$ such that $\Delta_{S^{n-3}} \chi_k + \mu_k(\mu_k+n-4) \chi_k = 0$ on $S^{n-3}$ for some integer $\mu_k \geq 0$.  Replace $\psi(r\omega)$ with $\gamma_k(r) \chi_k(\omega)$ in $C^{\infty}(\mathbb{R}^{n-2} \setminus \{0\})$ and note that (\ref{lemma4_2_eqn3})  and (\ref{lemma4_2_eqn4}) still hold for the new $\psi$ and $v = D_p \psi$ and one of (\ref{lemma4_2_decay2a}) and (\ref{lemma4_2_decay2b}) hold for the new $\psi$ depending on the values of $\lambda_l$ and $\alpha$.  (\ref{lemma4_2_eqn3}) yields 
\begin{equation} \label{lemma4_2_eqn5}
	(1+r^2) \frac{\partial^2 \gamma_k}{\partial r^2} + ((n-3) r^{-1} - (2\alpha-1) r) \frac{\partial \gamma_k}{\partial r} - \mu_k(\mu_k+n-4) r^{-2} \gamma_k 
	+ (\alpha^2 - \lambda_l) \gamma_k = 0.  
\end{equation}
Since $\psi$ is bounded near $0$, by elementary ODE estimates we obtain 
\begin{equation*}
	\sup_{\rho/2 < |z| < \rho} |D\psi| \leq C \rho^{-1} \sup_{\rho/4 < |z| < 2\rho} |\psi|. 
\end{equation*}
for some constant $C = C(n,l,k) \in (0,\infty)$. 

In order to prove (\ref{lemma4_2_eqn1}), we want to show the following: If $\phi_l$ is in the span of $D_1 \varphi^{(0)}(e^{i\theta})$ and $D_2 \varphi^{(0)}(e^{i\theta})$, then $\psi(z) = a + b \cdot z$ for some $a \in \mathbb{R}$ and $b \in \mathbb{R}^{n-2}$ if $\alpha = 1/2$ and $\psi(z) = b \cdot z$ for some $b \in \mathbb{R}^{n-2}$ if $\alpha \geq 1$.  If $\alpha \geq 1$ and $\lambda_l = \alpha^2$, then $\psi(z) = ar$ for some $a \in \mathbb{R}$.  In all other cases, $\psi(z) = 0$.  We will do so by considering following three cases:
\begin{enumerate}
\item[(1)] $\alpha = 1/2$ and $\phi_l$ is in the span of $\{D_1 \varphi^{(0)}(e^{i\theta}), D_2 \varphi^{(0)}(e^{i\theta})\}$;
\item[(2)] either $\alpha = 1/2$ and $\phi_l$ is not in the span of $\{D_1 \varphi^{(0)}(e^{i\theta}), D_2 \varphi^{(0)}(e^{i\theta})\}$ or $\alpha \geq 1$ and $\lambda_l \geq (\alpha-1)^2$;  
\item[(3)] $\alpha \geq 1$ and $\lambda_l < (\alpha-1)^2$. 
\end{enumerate}

Suppose case (1) occurs.  It follows from (\ref{lemma4_2_eqn2}) that $r^{1/2} w_{l}(r,y)$ is harmonic on $(0,\infty) \times \mathbb{R}^{n-2}$.  For $p = 1,2,\ldots,n-2$, define $\hat w_p$ on $\mathbb{R}^{n-1}$ to be the even extension in the $r$ variable of $r^{1/2} D_{y_p} w_{l}(r,y)$ given by $\hat w_p(r,y) = |r|^{1/2} D_{y_p} w_{l}(|r|,y)$ for all $r \in \mathbb{R}$ and $y \in \mathbb{R}^{n-2}$.  By (\ref{lemma4_2_alphaishalf}), $\hat w_p$ is harmonic on $\mathbb{R}^{n-1}$, and since $\hat w_{p}$ is homogeneous degree zero on $\mathbb{R}^{n-1}$, it follows that $\hat w_{p}$ is constant.  Thus $r^{1/2} D_{y_p} w_{l}(r,y)$ is constant on $(0,\infty) \times \mathbb{R}^{n-2}$ and consequently $r^{1/2} w_{l}(r,y) = a r + b \cdot y$ for some $a \in \mathbb{R}$ and $b \in \mathbb{R}^{n-1}$.  Equivalently, $\psi(z) = a + b \cdot z$.

Suppose case (2) occurs.  We can write (\ref{lemma4_2_eqn4}) as 
\begin{equation} \label{lemma4_2_eqn6}
	\frac{\partial}{\partial r} \left( g \frac{\partial v}{\partial r} \right) + \frac{g}{r^2 (1+r^2)} \Delta_{S^{n-3}} v 
		+ \frac{((\alpha-1)^2 - \lambda_l) g}{1+r^2} v = 0 
\end{equation}
where $g = r^{n-3} (1+r^2)^{-(2\alpha+n-6)/2}$.  Since $r^{-1} \nabla_{S^{n-3}} v$ is bounded as $r \downarrow 0$, we can write (\ref{lemma4_2_eqn6}) in the weak form 
\begin{equation} \label{lemma4_2_eqn5w}
	\int_0^{\infty} \int_{S^{n-3}} \left( g \frac{\partial v}{\partial r} \frac{\partial \zeta}{\partial r} 
	+ \frac{g}{r^2 (1+r^2)} \nabla^{S^{n-3}}  v \cdot \nabla^{S^{n-3}} \zeta - \frac{((\alpha-1)^2 - \lambda_l) g}{1+r^2} v \zeta \right) d\omega dr = 0 
\end{equation}
for all $\zeta \in C_c^{\infty}(\mathbb{R}^{n-2})$.  Replace $\zeta$ in (\ref{lemma4_2_eqn5w}) with $v \zeta^2$ where $\zeta = \zeta(r)$ is in $C^{\infty}(0,\infty)$ and is constant near $r = 0$ and vanishes for large $r$ and apply Cauchy's inequality to get 
\begin{align} \label{equation5a}
	&\int_0^{\infty} \int_{S^{n-3}} \left( g \left| \frac{\partial v}{\partial r} \right|^2 
	+ \frac{g}{r^2 (1+r^2)} |\nabla^{S^{n-3}} v|^2 + \frac{(\lambda_l - (\alpha-1)^2) g}{1+r^2} v^2 \right) \zeta^2 d\omega dr \nonumber 
	\\&\leq 4 \int_0^{\infty} g |\zeta'(r)|^2 \int_{S^{n-3}} v^2 d\omega dr 
\end{align}
for some constant $C \in (0,\infty)$.  Given $\rho \in (1,\infty)$, let $\zeta$ be the logarithmic cutoff function 
\begin{equation*}
	\zeta(r) = \left\{ \begin{array}{ll}
		1 &\text{if } r \leq \rho \\
		-\log(r/\rho^2)/\log(\rho) &\text{if } \rho^2 < r < \rho \\
		0 &\text{if } r \geq \rho^2 
	\end{array} \right.
\end{equation*}
to get that if $\alpha \geq 1$, $\lambda_l = (\alpha-1)^2$, and $\phi_l$ is in the span of $D_1 \varphi^{(0)}(e^{i\theta})$ and $D_2 \varphi^{(0)}(e^{i\theta})$ then by (\ref{lemma4_2_decay2a})
\begin{align*}
	&\int_0^{\rho} \int_{S^{n-3}} \left( g \left| \frac{\partial v}{\partial r} \right|^2 + \frac{g}{r^2 (1+r^2)} |\nabla^{S^{n-3}} v|^2 \right) d\omega dr 
	\\& \leq \frac{4}{\log^2(\rho)} \int_{\rho}^{\rho^2} r^{1-2\alpha} \int_{S^{n-3}} v^2 d\omega dr 
	\\& \leq \frac{C}{\log^2(\rho)} \int_{\rho/2}^{2\rho^2} r^{-1-2\alpha} \int_{S^{n-3}} \psi^2 d\omega dr 
	\\& \leq \frac{C}{\log^2(\rho)} \int_{\rho/2}^{2\rho^2} r^{1-2\alpha} \int_{S^{n-3}} \left| \frac{\psi(r\omega)}{r} - \kappa(r\omega) \right|^2 d\omega dr 
		+ \frac{C}{\log(\rho)} \rho^{2-2\alpha} \sup |\kappa|^2
	\\& \leq \frac{C}{\log^2(\rho)} \rho^{-2\alpha+\sigma} + \frac{C}{\log(\rho)} \rho^{2-2\alpha} \sup |\kappa|^2 
	\rightarrow 0 \text{ as } \rho \downarrow \infty, 
\end{align*}
where $C = C(n) \in (0,\infty)$, so $v$ is constant and thus $\psi(z) = a + b z$ for some $a \in \mathbb{R}$ and $b \in \mathbb{R}^m$.  By (\ref{lemma4_2_eqn3}), $b = 0$.  Similarly if either $\phi_l$ is not in the span of $D_1 \varphi^{(0)}(e^{i\theta})$ and $D_2 \varphi^{(0)}(e^{i\theta})$ or $\lambda_l > (\alpha-1)^2$ then by (\ref{lemma4_2_decay2b}) 
\begin{align*}
	&\int_0^{\rho} \int_{S^{n-3}} \left( g \left| \frac{\partial v}{\partial r} \right|^2 + \frac{g}{r^2 (1+r^2)} |\nabla^{S^{n-3}} v|^2 
		+ \frac{(\lambda_l - (\alpha-1)^2) g}{1+r^2} v^2 \right) d\omega dr 
	\\& \leq \frac{4}{\log^2(\rho)} \int_{\rho}^{\rho^2} r^{1-2\alpha} \int_{S^{n-3}} v^2 d\omega dr 
	\\& \leq \frac{C}{\log^2(\rho)} \int_{\rho/2}^{2\rho^2} r^{-1-2\alpha} \int_{S^{n-3}} \psi^2 d\omega dr 
	\\& \leq \frac{C}{\log^2(\rho)} \rho^{-2\alpha+\sigma} \rightarrow 0 \text{ as } \rho \downarrow \infty 
\end{align*}
so $v \equiv 0$ and thus $\psi(z) = a$ for some $a \in \mathbb{R}$.  By (\ref{lemma4_2_eqn3}), either $a = 0$ or $\lambda_l = \alpha^2$.

Suppose case (3) occurs.  Recall from (\ref{lemma4_2_eqn5}) that $\psi(r\omega) = \gamma_k(r) \varphi_k(\omega)$ where 
\begin{equation*}
	(1+r^2) \frac{\partial^2 \gamma_k}{\partial r^2} + ((n-3) r^{-1} - (2\alpha-1) r) \frac{\partial \gamma_k}{\partial r} 
	- \mu_k(\mu_k+n-4) r^{-2} \gamma_k + (\alpha^2 - \lambda_l) \gamma_k = 0.  
\end{equation*}
By using the series expansions near $\pm \infty$, any solution $\gamma_k$ to (\ref{lemma4_2_eqn5}) that is not identically zero must satisfy 
\begin{equation*}
	\liminf_{\rho \rightarrow \infty} \rho^{-1-2\alpha+2\sqrt{\lambda_l}} \int_{\rho/2 < r < \rho} |\gamma_k(r)|^2 dr > 0, 
\end{equation*}
contradicting (\ref{lemma4_2_decay2b}).  Therefore $\gamma_k$ is identically zero, hence $\psi(z) = 0$. 
\end{proof}

In Lemma~\ref{lemma4_12} and Lemma~\ref{lemma4_14} below, we shall denote by ${\mathcal L}$ the subspace of the Hilbert space $L^2(\op{graph} \varphi^{(0)} |_{B_{1}(0)}; {\mathbb R}^{m}, p^{\star} \, {\mathcal L}^{n})$ spanned by the set of functions 
$$\left\{a_0 r^{\alpha} \cos(\alpha \theta), \; b_0 r^{\alpha} \sin(\alpha \theta), \; D_i \varphi^{(0)}(re^{i\theta},y) y_j \, : \,  a_0,b_0 \in \mathbb{R}^m, \; i \in \{1,2\},  \; j \in\{1,\ldots,n-2\}\right\}$$ and  
let $\mathcal{L}^{\perp}$ denote the orthogonal compliment of $\mathcal{L}$ in 
$L^2(\op{graph} \varphi^{(0)}|_{B_{1}(0)}; {\mathbb R}^{m}, p^{\star} \, {\mathcal L}^{n})$.

\begin{lemma} \label{lemma4_12} 
Let $\varphi^{(0)} : \mathbb{R}^n \rightarrow \mathcal{A}_2(\mathbb{R}^m)$ be a two-valued function given by (\ref{varphi0}) for some $c^{(0)} \in \mathbb{C}^m \setminus \{0\}$ and $\alpha \in \{1/2,1,3/2,2,\ldots\}$.  Given $\sigma \in (0,1)$ and $\beta_1, \beta_2 \in (0,\infty)$, there exists a constant $\beta = \beta(n,\alpha,\varphi^{(0)},\sigma,\beta_1,\beta_2) \in (0,\infty)$ such that if $w \in \mathcal{L}^{\perp}$ such that $w(X,\varphi^{(0)}(X))$ is a harmonic symmetric two-valued function on $B_1(0) \setminus \{0\} \times \mathbb{R}^{n-2}$, 
\begin{equation} \label{lemma4_12_decay-0}
	\int_{B_{1/4}((0, z))} \frac{|w(X,\varphi^{(0)}(X)) - (\lambda_{1}(z)D_{1}\varphi^{(0)}(X) + \lambda_{2}(z) D_{2}\varphi^{(0)}(X))|^2}{|X - (0, z)|^{n+2\alpha-\sigma}} 
	\leq \beta_1 \int_{B_1(0)} |w(X,\varphi^{(0)}(X))|^2
\end{equation}
for each $z \in B_{1/2}^{n-2}(0)$ and some bounded functions $\lambda_{1}, \lambda_{2} \, : \, B^{n-2}_{1/2}(0) \to {\mathbb R}$ with 
$${\rm sup}_{B^{n-2}_{1/2}(0)} \, \left(|\lambda_{1}(z)|^{2} + |\lambda_{2}(z)|^{2} \right)\leq \b_{2}\int_{B_1(0)} |w(X,\varphi^{(0)}(X))|^2$$
and (\ref{lemma4_2_alphaishalf}) holds if $\alpha = 1/2$, then 
\begin{equation*} \label{lemma4_12_conclusion}
	\int_{B_1(0)} |w(X,\varphi^{(0)}(X))|^2 
	\leq \beta \int_{B_1(0) \setminus B_{1/4}(0)} \left| \frac{\partial}{\partial R} \left( \frac{w(X,\varphi^{(0)}(X))}{R^{\alpha}} \right) \right|^2. 
\end{equation*}
\end{lemma}
\begin{proof} Let the hypotheses be as in the lemma. We claim first that $w$  satisfies 
\begin{equation} \label{lemma4_12_decay}
	\int_{B_{1/2}(0)} \frac{|w(X,\varphi^{(0)}(X)) - \kappa(X,\varphi^{(0)}(X))|^2}{r^{2+2\alpha-\sigma}} 
	\leq C \int_{B_1(0)} |w(X,\varphi^{(0)}(X))|^2
\end{equation}
for some constant $C = C(\b_{1}) \in (0, \infty)$ and for some function $\k$ of the form 
\begin{equation}\label{kappa-cont-1-1}
\kappa(re^{i\th}, y, \varphi^{(0)}(re^{i\th}, y)) = \kappa_1(r,y) D_1 \varphi^{(0)}(re^{i\theta},y) + \kappa_2(r,y) D_2 \varphi^{(0)}(re^{i\theta},y)
\end{equation}
with  
\begin{equation} \label{lemma4_12_kappaest}
	|\kappa_1(r,y)|^2 + |\kappa_2(r,y)|^2 \leq \beta_2 \int_{B_1(0)} |w(X,\varphi^{(0)}(X))|^2,
\end{equation}
and that the functions $\lambda_{j}$, $j=1, 2$ satisfy the uniform continuity estimate 
\begin{equation}\label{kappa-cont-1}
\sup_{z_{1}, z_{2} \in B^{n-2}_{1/2}(0), \, z_{1} \neq z_{2}} \, \frac{|\lambda_{j}(z_{1}) - \lambda_{j}(z_{2})|^{2}}{|z_{1} - z_{2}|^{2 - \s}} \leq C\int_{B_{1}(0)} |w(X, \varphi^{(0)}(X))|^{2},
\end{equation}
where $C = C(n, \b_{1}, \b_{2}, \s) \in (0, \infty).$

To see (\ref{lemma4_12_decay}), note that by hypothesis, we have for each $z \in B^{n-2}_{1/2}(0)$ and $\r \in (0, 1/4)$, 
\begin{eqnarray}\label{kappa-cont}
	&&\hspace{-.5in}\rho^{-n-2\alpha+\sigma} \int_{B_{\rho}(0,z)} |w(X,\varphi^{(0)}(X)) - (\lambda_{1}(z)D_{1} \varphi^{(0)}(X) + \lambda_{2}(z)D_{2}\varphi^{(0)}(X))|^2\nonumber\\ 
	&&\hspace{3in}\leq \b_{1} \int_{B_1(0)} |w(X,\varphi^{(0)}(X))|^2.
\end{eqnarray}
For each $(r,y)$ with $r^{2} + |y|^{2} < 1/2$,   choose $\kappa_1(r,y), \kappa_2(r,y) \in \mathbb{R}$ to satisfy 
\begin{align*}
	&\hspace{-.5in}\int_0^{4\pi} |w(re^{i\theta},y,\varphi^{(0)}(re^{i\theta},y)) - \kappa_1(r,y) D_1 \varphi^{(0)}(re^{i\th}) - \kappa_2(r,y) D_2 \varphi^{(0)}(re^{i\th})|^2 d\theta. 
	\\&= \inf \int_0^{4\pi} |w(re^{i\theta},y,\varphi^{(0)}(re^{i\theta},y)) - \m_1 D_1 \varphi(re^{i\th}) - \m_2 D_2 \varphi(re^{i\th})|^2 d\theta, 
\end{align*}
where the infimum is taken over all $\m_1, \m_2 \in \mathbb{R}$ such that $|\m_1|^2 + |\m_2|^2 \leq \b_{2} \int_{B_1(0)} |w(X, \varphi^{(0)}(X))|^{2}$. Thus $\kappa = \kappa(re^{i\theta},y,\varphi^{(0)}(re^{i\theta},y)) = \kappa_1(r,y) D_1 \varphi^{(0)}(re^{i\theta},y) + \kappa_2(r,y) D_2 \varphi^{(0)}(re^{i\theta},y)
$ satisfies 
\begin{equation} \label{cor3_2_eqn2}
	\rho^{-n-2\alpha+\sigma} \int_{B_{\rho}(0,z)} |w(X,\varphi^{(0)}(X)) - \kappa(X,\varphi^{(0)}(X))|^2
	\leq \b_{1} \int_{B_1(0)} |w(X, \varphi^{(0)})|^2
\end{equation}
for each $z \in B_{1/2}^{n-2}(0)$ and $\r \in (0, 1/4).$ A straightforward covering argument then yields 
\begin{equation*} 
	\rho^{-2-2\alpha+\sigma} \int_{B_{1/2}(0) \cap \{r< \r/2\}} |w(X,\varphi^{(0)}(X)) - \kappa(X,\varphi^{(0)}(X))|^2
	\leq C\b_{1} \int_{B_1(0)} |w(X, \varphi^{(0)}(X))|^2
\end{equation*}
for $\r \in (0, 1/4)$ where $C = C(n) \in (0, \infty)$; replacing $\s$ with $\s/2$ in this, multiplying by $\rho^{-1+\sigma/2}$ and integrate with respect to $\rho \in (0,1/4)$ gives (\ref{lemma4_12_decay}).

To see (\ref{kappa-cont-1}), take two points $z_{1}, z_{2} \in B_{1/2}^{n-2}(0)$ with $0< |z_{1} - z_{2}| < 1/8$, and apply 
(\ref{kappa-cont}) with $z = z_{j}$, $j=1, 2$ and $\r = 2|z_{1} - z_{2}|$ to deduce, with the help of the triangle inequality, that 

\begin{eqnarray*}
&&|z_{1} - z_{2}|^{-n-2\a +\s}\int_{B_{|z_{1} - z_{2}|}(0, z_{1})} |(\lambda_{1}(z_{1}) - \lambda_{1}(z_{2})) D_{1}\varphi^{(0)}(X) + 
(\lambda_{2}(z_{1}) - \lambda_{2}(z_{2}))D_{2}\varphi^{(0)}(X)|^{2}\nonumber\\ 
&&\hspace{3in}\leq C\int_{B_{1}(0)} |w(X, \varphi^{(0)}(X)|^{2}
\end{eqnarray*}
where $C = C(n, \b_{1}, \s) \in (0, \infty).$ In view of $L^{2}(S^{1})$ orthogonality of $\left.D_{1}\varphi^{(0)}\right|_{r=1}$ and 
$\left.D_{2}\varphi^{(0)}\right|_{r=1}$, and the hypothesis that ${\rm sup}_{B^{n-2}_{1/2}(0)} \, \left(|\lambda_{1}(z)|^{2} + |\lambda_{2}(z)|^{2} \right)\leq \b_{2}\int_{B_1(0)} |w(X,\varphi^{(0)}(X))|^2$, 
this readily implies (\ref{kappa-cont-1}). 

Now suppose the lemma is false. Then for every integer $k \geq 1$ there is $w_k \in \mathcal{L}^{\perp}$ and functions $\lambda_{1, k}, \lambda_{2, k} \, : \, B^{n-2}_{1/2}(0) \to {\mathbb R}$ 
satisfying the hypotheses of the lemma with $w_{k}, \lambda_{k, 1}, \lambda_{k, 2}$ in place of $w, \lambda_{1}, \lambda_{2}$ and yet 
\begin{equation} \label{kappa-cont-2-2}
	\int_{B_1(0)} |w_{k}(X,\varphi^{(0)}(X))|^2 
	> k \int_{B_1(0) \setminus B_{1/4}(0)} \left| \frac{\partial}{\partial R} \left( \frac{w_{k}(X,\varphi^{(0)}(X))}{R^{\alpha}} \right) \right|^2. 
\end{equation}
By scaling we may suppose that 
\begin{equation*}
	\int_{B_1(0)} |w_k(X,\varphi^{(0)}(X))|^2 = 1
\end{equation*}
for all $k$.  By (\ref{lemma4_12_decay}) with $w_k$ in place of $w$, 
\begin{equation} \label{lemma4_12_eqn1}
	\int_{B_{1/2}(0) \cap \{ |x| \leq \delta \}} |w_k(X,\varphi^{(0)}(X))|^2 \leq C (\beta_1 + \b_{2}) \delta^2  
\end{equation}
for each $\d \in (0, 1/4)$ and some constant $C = C(n) \in (0,\infty)$.  
By standard elliptic estimates, $\sup_k \|w_k\|_{C^3(K)} < \infty$ for every compact subset $K$ of $B_1(0) \setminus \{0\} \times B_{1}^{n-2}(0)$ and hence there exists a two valued function $w = w(X, \varphi^{(0)}(X))$ in $C^{2} \, (B_{1}(0) \setminus \{0\} \times {\mathbb R}^{n-2})$ such that after passing to a subsequence, $w_k$ converges to $w$ in $C^2(K)$ for every compact subset $K$ of  $B_1(0) \setminus \{0\} \times B_{1}^{n-2}(0)$ and $\Delta w(X,\varphi^{(0)}(X)) = 0$ in $B_1(0) \setminus \{0\} \times B^{n-2}_{1}(0)$.  By (\ref{lemma4_12_eqn1}), $w_k(X, \varphi^{(0)}(X))$ in fact converges to $w(X, \varphi^{(0)}(X))$ in $L^{2}(B_{1/2}(0)).$ Since for $j=1, 2$ and all $k= 1, 2, 3, \ldots,$
$$\sup_{B_{1/2}^{n-2}(0)} |\lambda_{k, j}(z)|^{2} + 
\sup_{z_{1}, z_{2} \in B_{1/2}^{n-2}(0), \, z_{1} \neq z_{2}} \frac{|\lambda_{k, j}(z_{1}) - \lambda_{k, j}(z_{2})|^{2}}{|z_{1} - z_{2}|^{2 - \s}} \leq C,$$ 
where $C = C(n, \b_{1}, \b_{2}, \s) \in (0, \infty)$, it follows that  $w$ satisfies 

\begin{equation} \label{kappa-cont-3}
	\int_{B_{1/4}((0, z))} \frac{|w(X,\varphi^{(0)}(X)) - (\lambda_{1}(z)D_{1}\varphi^{(0)}(X) + \lambda_{2}(z) D_{2}\varphi^{(0)}(X))|^2}{|X - (0, z)|^{n+2\alpha-\sigma}} 
	\leq \beta_1 
\end{equation}
for each $z \in B_{1/2}^{n-2}(0)$ and some functions  $\lambda_{1}, \lambda_{2} \, : \, B^{n-2}_{1/2}(0) \to {\mathbb R}$ with 
$${\rm sup}_{B^{n-2}_{1/2}(0)} \, \left(|\lambda_{1}(z)|^{2} + |\lambda_{2}(z)|^{2} \right)\leq \b_{2}.$$
This implies, by the argument leading to (\ref{lemma4_12_decay}), that 
\begin{equation*}
	\int_{B_{1/2}(0)} \frac{|w(X,\varphi^{(0)}(X)) - \kappa(X,\varphi^{(0)}(X))|^2}{r^{2+2\alpha-\sigma}} < C
\end{equation*}
for some constant $C = C(\b_{1}) \in (0, \infty)$ and some function $\kappa$ of the form (\ref{kappa-cont-1-1}) with $\kappa_1, \kappa_2 \in L^{\infty}(B_1(0);\mathbb{R}^m)$.  By (\ref{kappa-cont-2-2}), $w(X,\varphi^{(0)}(X))$ is homogeneous degree $\alpha$ in $\left(B_1(0) \setminus B_{1/4}(0)\right) \setminus \{0\} \times {\mathbb R}^{n-2}$.  Since $\Delta w(X,\varphi^{(0)}(X)) = 0$ in $B_1(0) \setminus \{0\} \times B^{n-2}_1(0)$, by unique continuation $w(X,\varphi^{(0)}(X))$ is homogeneous degree $\alpha$ in $B_1(0) \setminus \{0\} \times {\mathbb R}^{n-2}$.  Thus by Lemma \ref{lemma4_2}, $w \in \mathcal{L}$.  Since $w_{k} \in {\mathcal L}^{\perp}$, we then have that 
$\int_{B_{1}(0)} w_{k}(X, \varphi^{(0)}(X)) \cdot w(X, \varphi^{(0)}(X)) = 0$ and thus 
for each $\s \in (1/2, 1)$ and $\d \in (0, 1/2)$, 
\begin{eqnarray*}
&&\int_{B_{\s}(0) \setminus B_{\d}^{2}(0) \times {\mathbb R}^{n-2}}w_{k}(X, \varphi^{(0)}(X)) \cdot w (X, \varphi^{(0)}(X))\nonumber\\
&&\hspace{1in} = -\int_{B_{\s}(0) \cap B_{\d}^{2}(0) \times {\mathbb R}^{n-2}}w_{k}(X, \varphi^{(0)}(X)) \cdot w (X, \varphi^{(0)}(X))\nonumber\\
&&\hspace{2in}-  \int_{B_{1}(0) \setminus B_{\s}(0)}w_{k}(X, \varphi^{(0)}(X)) \cdot w (X, \varphi^{(0)}(X))
\end{eqnarray*}
whence
\begin{eqnarray*}
&&\left|\int_{B_{\s}(0) \setminus B_{\d}^{2}(0) \times {\mathbb R}^{n-2}}w_{k}(X, \varphi^{(0)}(X)) \cdot w (X, \varphi^{(0)}(X))\right|\nonumber\\ 
&&\hspace{1in}\leq \left(\int_{B_{\s}(0) \cap B_{\d}^{2}(0) \times {\mathbb R}^{n-2}}|w|^{2}\right)^{1/2} + \left(\int_{B_{1}(0) \setminus B_{\s}(0)}|w|^{2}\right)^{1/2};
\end{eqnarray*}
letting $k \to \infty$ first and then $\s \to 1$, $\d \to 0^{+}$ in this, we conclude that $w \equiv 0$.  Thus $w_k \rightarrow 0$ in $L^2(\op{graph} \varphi^{(0)} \cap B_{1/2}(0) \times \mathbb{R}^m)$.

Now observe that for $\rho_1,\rho_2 \in (1/4,1)$ and $\omega \in S^{n-1}$, 
\begin{equation*}
	\left| \frac{w_k(\rho_1 \omega,\varphi^{(0)}(\rho_1 \omega))}{\rho^{\alpha}} - \frac{w_k(\rho_2 \omega,\varphi^{(0)}(\rho_2 \omega))}{\rho_2^{\alpha}} \right| 
	\leq \int_{1/4}^1 \left| \frac{\partial}{\partial R} \left( \frac{w_k(R \omega,\varphi^{(0)}(R \omega))}{R^{\alpha}} \right) \right| .
\end{equation*}
Thus by the triangle inequality and Cauchy-Schwartz, 
\begin{equation*}
	|w_k(\rho_1 \omega,\varphi^{(0)}(\rho_1 \omega))|^2 \leq C \left( |w_k(\rho_2 \omega,\varphi^{(0)}(\rho_2 \omega))|^2 
		+ \int_{1/4}^1 \left| \frac{\partial}{\partial R} \left( \frac{w_k(R \omega,\varphi^{(0)}(R \omega))}{R^{\alpha}} \right) \right|^2 \right) 
\end{equation*}
for some constant $C = C(\alpha) \in (0,\infty)$.  Multiply both sides by $\rho_1^{n-1} \rho_2^{n-1}$ and integrate over $\omega \in S^{n-1}$, $\rho_1 \in (1/4,1)$, and $\rho_2^{n-1}$ and integrate over $\rho_2 \in (1/4,1/2)$ to get 
\begin{align*}
	&\int_{B_1(0) \setminus B_{1/4}(0)} |w_k(X,\varphi^{(0)}(X))|^2 
	\\&\leq C \left( \int_{B_{1/2}(0) \setminus B_{1/4}(0)} |w_k(X,\varphi^{(0)}(X))|^2 
		+ \int_{B_1(0) \setminus B_{1/2}(0)} \left| \frac{\partial}{\partial R} \left( \frac{w_k(X,\varphi^{(0)}(X))}{R^{\alpha}} \right) \right|^2 \right) .
\end{align*}
for some constant $C = C(n,\alpha) \in [1,\infty)$.  Thus 
\begin{align*}
	&1 = \int_{B_1(0)} |w_k(X,\varphi^{(0)}(X))|^2 
	\\&\leq C \left( \int_{B_{1/2}(0)} |w_k(X,\varphi^{(0)}(X))|^2 
		+ \int_{B_1(0) \setminus B_{1/2}(0)} \left| \frac{\partial}{\partial R} \left( \frac{w_k(X,\varphi^{(0)}(X))}{R^{\alpha}} \right) \right|^2 \right) ,
\end{align*}
which yields a contradiction since the right-hand side converges to zero as $k \rightarrow \infty$. 
\end{proof}

In the next lemma, we shall use the following notation: Given $w \in L^2(\op{graph} \varphi^{(0)}|_{B_{1}(0)}; {\mathbb R}^{m}, p^{\star} \, {\mathcal L}^{n})$ and $\rho \in (0,1)$, we let $\psi_{\rho} \in \mathcal{L}$ be such that 
\begin{equation*}
	\int_{B_{\rho}(0)} |w(X,\varphi^{(0)}(X)) - \psi_{\rho}(X,\varphi^{(0)}(X))|^2 
	= \inf_{\psi \in \mathcal{L}} \int_{B_{\rho}(0)} |w(X,\varphi^{(0)}(X)) - \psi(X,\varphi^{(0)}(X))|^2 
\end{equation*}
and let 
\begin{equation*}
	w_{\rho} = w - \psi_{\rho}. 
\end{equation*}
By standard Hilbert space theory, $\psi_{\rho}$ exists and $w_{\rho}$ is orthogonal to ${\mathcal L}$ in the Hilbert space $L^{2}(\op{graph} \varphi^{(0)}|_{B_{\r}(0)}; {\mathbb R}^{m}, p^{\star} \, {\mathcal L}^{n})$. 

\begin{lemma} \label{lemma4_14} 
Let $\varphi^{(0)} : \mathbb{R}^n \rightarrow \mathcal{A}_2(\mathbb{R}^m)$ be a two-valued function given by (\ref{varphi0}) for some $c^{(0)} \in \mathbb{C}^m \setminus \{0\}$ and $\alpha \in \{1/2,1,3/2,2,\ldots\}$.  Let $\vartheta \in (0,1/8)$, $\sigma \in (0,1)$, and $\beta_1, \beta_2 \in (0,\infty)$.  If $w \in C^2(\op{graph} \varphi^{(0)} |_{B_1(0) \setminus \{0\} \times \mathbb{R}^{n-2}};\mathbb{R}^m) \cap L^2(\op{graph} \varphi^{(0)} |_{B_1(0)};\mathbb{R}^m)$ such that $w(X,\varphi^{(0)}(X))$ is a harmonic symmetric 2-valued function on $B_1(0) \setminus \{0\} \times \mathbb{R}^{n-2}$ and for each $\rho \in [\vartheta,1/4]$ and $z \in B^{n-2}_{\r/2}(0)$,
\begin{eqnarray*} 
	&&\hspace{-.7in}\r^{n+2\a - \s}\int_{B_{\r/4}((0, z))} \frac{|w_{\r}(X,\varphi^{(0)}(X)) - (\r\lambda_{1, \, \r}(z)D_{1}\varphi^{(0)}(X) + \r\lambda_{2, \, \r}(z) D_{2}\varphi^{(0)}(X))|^2}{|X - (0, z)|^{n+2\alpha-\sigma}}\nonumber\\ 
	&&\hspace{3in}\leq \beta_1 \int_{B_{\r}(0)} |w_{\r}(X,\varphi^{(0)}(X))|^2
\end{eqnarray*}
for some  $\lambda_{1, \rho}, \lambda_{2, \rho} \in L^{\infty}(B^{n-2}_{\r/2}(0);\mathbb{R}^m)$ with 
\begin{equation*} 
	|\lambda_{1, \rho}(z)|^2 + |\lambda_{2, \rho}(z)|^2 \leq \beta_2 \rho^{-n- 2\a} \int_{B_{\rho}(0)} |w_{\rho}(X,\varphi^{(0)}(X))|^2, 
\end{equation*}
\begin{equation} \label{lemma4_14_eqn1}
	\int_{B_{\rho/4}(0)} R^{2-n} \left| \frac{\partial}{\partial R} \left( \frac{w(X,\varphi^{(0)}(X))}{R^{\alpha}} \right) \right|^2 
	\leq \beta_2 \r^{-n-2\a}\int_{B_{\rho}(0)} |w_{\rho}(X,\varphi^{(0)}(X))|^2, 
\end{equation}
and (\ref{lemma4_2_alphaishalf}) holds if $\alpha = 1/2$, then 
\begin{equation} \label{lemma4_14_eqn2}
	\vartheta^{-n-2\alpha} \int_{B_{\vartheta}(0)} |w_{\vartheta}(X,\varphi^{(0)}(X))|^2 \leq C \vartheta^{2\mu} \int_{B_1(0)} |w_1(X,\varphi^{(0)}(X))|^2
\end{equation}
for some constants $\mu \in (0,1)$ and $C \in (0, \infty)$ depending only on $n$, $\alpha$, $\varphi^{(0)}$, $\sigma$, $\beta_1$, $\beta_2$.  Note that in particular $C$ and $\mu$ are independent of $\vartheta$.  
\end{lemma}

\begin{proof} 
By Lemma \ref{lemma4_12}, 
\begin{equation} \label{lemma4_14_eqn3}
	\rho^{-n-2\alpha} \int_{B_{\rho}(0)} |w_{\rho}(X,\varphi^{(0)}(X))|^2 
	\leq \beta \int_{B_{\rho}(0) \setminus B_{\rho/4}(0)} R^{2-n} 
		\left| \frac{\partial}{\partial R} \left( \frac{w(X,\varphi^{(0)}(X))}{R^{\alpha}} \right) \right|^2 
\end{equation}
for some $\beta = \beta(n,\alpha,\varphi^{(0)},\sigma,\beta_1,\beta_2) \in (0,\infty)$.  By (\ref{lemma4_14_eqn1}) and (\ref{lemma4_14_eqn3}), 
\begin{equation} \label{lemma4_14_eqn4}
	\int_{B_{\rho/4}(0)} R^{2-n} \left| \frac{\partial}{\partial R} \left( \frac{w(X,\varphi^{(0)}(X))}{R^{\alpha}} \right) \right|^2 
	\leq \beta \beta_2 \int_{B_{\rho}(0) \setminus B_{\rho/4}(0)} R^{2-n} 
		\left| \frac{\partial}{\partial R} \left( \frac{w(X,\varphi^{(0)}(X))}{R^{\alpha}} \right) \right|^2 . 
\end{equation}
Adding $\beta \beta_2$ times the left-hand side of (\ref{lemma4_14_eqn4}) to both sides of (\ref{lemma4_14_eqn4}) we get 
\begin{equation} \label{lemma4_14_eqn5}
	\int_{B_{\rho/4}(0)} R^{2-n} \left| \frac{\partial}{\partial R} \left( \frac{w(X,\varphi^{(0)}(X))}{R^{\alpha}} \right) \right|^2 
	\leq \gamma \int_{B_{\rho}(0)} R^{2-n} \left| \frac{\partial}{\partial R} \left( \frac{w(X,\varphi^{(0)}(X))}{R^{\alpha}} \right) \right|^2 
\end{equation}
for $\gamma = \beta \beta_2/(1+\beta \beta_2) \in (0,1)$.  Iterating (\ref{lemma4_14_eqn5}) with $\rho = 4^{-j}$ for $j = 1,\ldots,k$, where $4^{-k-1} < \vartheta \leq 4^{-k}$, we get 
\begin{equation*}
	\int_{B_{4^{-k}}(0)} R^{2-n} \left| \frac{\partial}{\partial R} \left( \frac{w(X,\varphi^{(0)}(X))}{R^{\alpha}} \right) \right|^2 
	\leq \gamma^k \int_{B_{1/4}(0)} R^{2-n} \left| \frac{\partial}{\partial R} \left( \frac{w(X,\varphi^{(0)}(X))}{R^{\alpha}} \right) \right|^2, 
\end{equation*}
and thus 
\begin{equation} \label{lemma4_14_eqn7}
	\int_{B_{\vartheta}(0)} R^{2-n} \left| \frac{\partial}{\partial R} \left( \frac{w(X,\varphi^{(0)}(X))}{R^{\alpha}} \right) \right|^2 
	\leq\vartheta^{2\mu} \int_{B_{1/4}(0)} R^{2-n} \left| \frac{\partial}{\partial R} \left( \frac{w(X,\varphi^{(0)}(X))}{R^{\alpha}} \right) \right|^2, 
\end{equation}
where $2\mu = -\log \g/\log 4.$  By combining (\ref{lemma4_14_eqn1}), (\ref{lemma4_14_eqn3}), and (\ref{lemma4_14_eqn7}), we conclude (\ref{lemma4_14_eqn2}). 
\end{proof}

\section{Proofs of the main results} \label{prooflemma_section}
\setcounter{equation}{0}
We can now complete the proofs of Lemma~\ref{lemma1} and Theorems~\ref{theorem2}, ~\ref{theorem3}, \ref{no-gaps} arguing more or less as in \cite{SimonCTC}. We begin with Lemma~\ref{lemma1}.

\begin{proof}[Proof of Lemma~\ref{lemma1}] Let $\vartheta \in (0,1/4)$ and $\varphi^{(0)}$ be fixed as in the lemma.  Consider case (a), i.e. the case that $\varphi^{(0)}$ is Dirichlet energy minimizing. The proof of the lemma in case (b), i.e. when $\varphi^{(0)} \in C^{1, 1/2}$ with $D\varphi^{(0)}(0) = \{0, 0\}$ requires only the obvious changes. Let $0 < \varepsilon_j \leq \delta_j \downarrow 0$, $u_j \in \mathcal{F}^{\text{Dir}}_{\varepsilon_j}(\varphi^{(0)})$ be such that alternative (i) in Lemma \ref{lemma1} is false with $\delta_j$, $u_j$ in place of $\delta_{0}$, $u$ respectively, and let $\varphi_j \in \widetilde{\Phi}_{\varepsilon_j}(\varphi^{(0)})$ .  We want to show that for some constants $\gamma \in [1,\infty)$, $\mu \in (0,1)$ and $C \in (0,\infty)$ depending only on $n$, $m$, and $\vartheta$, and for infinitely many $j$, there exists $\widetilde{\varphi}_j \in \widetilde{\Phi}_{\gamma \varepsilon_j}(\varphi^{(0)})$ such that 
\begin{equation} \label{lemma1_eqn1} 
	\vartheta^{-n-2\a}\int_{B_{\vartheta}(0)} \mathcal{G}(u_j,\widetilde{\varphi}_j)^2 \leq C\vartheta^{2\mu} \int_{B_1(0)} \mathcal{G}(u_j,\varphi_j)^2.
\end{equation}
 By the arbitrariness of the sequences this will complete the proof of 
Lemma~\ref{lemma1}.  Observe that by definition $\varphi_{j}^{\prime}  \equiv \varphi_j(e^{-A_j}(\cdot) ) \in \Phi_{\varepsilon_j}(\varphi^{(0)})$ for some matrix $A_j \in \mathcal{S}$ with $|A_j| \leq \varepsilon_j$. 

Let $\beta_0 = \beta_0(\varphi^{(0)}) > 0$ be as in Theorem \ref{thm6_2} and Corollary \ref{cor6_6} hold.  Let $\tau_j \downarrow 0$ slowly enough that Lemma~\ref{coarse_graph} and Theorem~\ref{thm6_2} hold with $\gamma = 3/4$, $\tau = \tau_j$, $\beta = \beta_0$, $\sigma = 1/2$, $u=u_{j}$, and $\varphi = \varphi_j^{\prime}$.   By Lemma~\ref{coarse_graph} we have an open set 
$U_{j} \subset B_{1}(0) \setminus \{0\} \times {\mathbb R}^{n-2}$ and $v_j \in C^2(\op{graph} \varphi_{j}^{\prime} |_{U_j},\mathbb{R}^m)$ such that $U_j \supset \{ (x,y) \in B_{3/4}(0) \cap \{ |x| \geq \tau_j \}$, 
\begin{equation*}
	u_j(e^{-A_j} X) = \{ \varphi_{j,1}^{\prime}(X) + v_j(X,\varphi_{j,1}^{\prime}(X)), -\varphi_{j,1}^{\prime}(X) + v_j(X,-\varphi_{j,1}^{\prime}(X)) \} 
\end{equation*} 
for $X \in U_j$ and $v_j(X,\varphi_{j}^{\prime}(X))$ is a harmonic two-valued function, where we have written 
$\varphi_{j}^{\prime} = \{\pm\varphi_{j, 1}^{\prime}\}$ locally away from $\{0\} \times {\mathbb R}^{n-2}$ for a single-valued harmonic function $\varphi_{j, 1}^{\prime}$. Assume $\tau_j \downarrow 0$ slowly enough that $\varepsilon_j/\tau_j^{3/2} \rightarrow 0$ and let $\rho \in [\vartheta,1/2]$. By Theorem~\ref{thm6_2} and Corollary~\ref{cor6_6} (taken with $\rho^{-2\alpha} u_j(\rho e^{-A_j} X)$ in place of $u$ and $\varphi_j^{\prime}$ in place of $\varphi$), we have that 
 \begin{equation} \label{lemma1_eqn3} 
	 \int_{B_{\rho/2}(0)} R^{2-n} \left| \frac{\partial (u_j/R^{\alpha})}{\partial R} \right|^2 
	\leq C \rho^{-n-2\alpha} \int_{B_{\r}(0)} \mathcal{G}(u_j,\varphi_j)^2 
\end{equation}
for $C = C(n,m,\varphi^{(0)},\alpha) \in (0,\infty)$, and 
\begin{equation} \label{lemma1_eqn4} 
		 \rho^{1/2} \int_{B_{\rho/2}(0)} \frac{\mathcal{G}(u_j,\varphi_j)^2}{r_{\delta\r}^{1/2}} \leq 
		C\int_{B_{\rho}(0)} \mathcal{G}(u_j,\varphi_j)^2. 
\end{equation}
for $\delta \geq 2\delta_j/\vartheta,$ where $C = C(n,m,\varphi^{(0)},\alpha) \in (0,\infty)$, $r = |x|$ and$r_{\delta} = \max\{r,\delta\}$. Moreover, for each $z \in B_{\r/2}^{n-2}(0)$, there exists 
$Z_{j}  = (\xi_{j}, \z_{j}) \in B_{\d_{j}}^{n}((0, z))$ with ${\mathcal N}_{u_{j}}(Z_{j}) \ge \alpha$, so that by Corollary~\ref{cor6_4},  
\begin{equation}\label{lemma1_eqn4-0}
\r^{-2}|\xi_{j}|^{2} \leq C\r^{-n-2\a}\int_{B_{\r}(0)} {\mathcal G} \, (u_{j}, \varphi_{j})^{2}
\end{equation}
and by Corollary~\ref{cor6_5}, 
\begin{equation}\label{lemma1_eqn4-1}
\r^{n+2\a-\s}\int_{B_{\r/2}((0,  z)) \cap \{ |x| > \vartheta^{-1}\tau_{j}\r \}} \frac{|v_{j}(X,\varphi_{j}(X)) -  D_x \varphi_{j}(X) \cdot \xi_{j}|^2}{|X-\r Z_{j}|^{n+2\alpha-\sigma}} 
	\leq C_{1} \int_{B_{\r}(0)} \mathcal{G}(u_{j},\varphi_{j})^2. 
\end{equation}
Note that in particular the constants $C$, $C_{1}$ in (\ref{lemma1_eqn3})--(\ref{lemma1_eqn4-1}) are independent of $\vartheta$. 

Next let 
\begin{equation*}
	w_j = v_j/E_j \hspace{3mm} \text{where} \hspace{3mm} E_j = \left( \int_{B_1(0)} \mathcal{G}(u_j,\varphi_j)^2 \right)^{1/2}.
\end{equation*}
By standard estimates for harmonic functions, for every compact set $K \subseteq B_1(0) \setminus \{0\} \times \mathbb{R}^{n-2}$, 
\begin{equation} \label{lemma1_eqn6} 
	\|w_j(X,\varphi_{j}^{\prime}(X))\|_{C^3(K)} \leq C
\end{equation}
for $j$ sufficiently large depending on $K$, where $C = C(n,m,K) \in (0,\infty)$.  Recall that $\varphi^{(0)}(re^{i\theta},y) = \op{Re}(c^{(0)} r^{\alpha} e^{i\alpha \theta})$ and $\varphi_j^{\prime}(re^{i\theta},y) = \op{Re}(c_j r^{\alpha} e^{i\alpha \theta})$ for some constants $c^{(0)}, c_j \in \mathbb{C}^m \setminus \{0\}$ with 
\begin{equation} \label{lemma1_eqn7} 
	|c_j - c^{(0)}| \leq C \varepsilon_j
\end{equation}
for some constant $C = C(n,\alpha) \in (0,\infty)$ and thus 
\begin{equation} \label{lemma1_eqn8} 
	\varphi_j^{\prime}(X) = \{ \varphi_1^{(0)}(X) + \psi_j(X,\varphi_1^{(0)}(X)), -\varphi_1^{(0)}(X) + \psi_j(X,-\varphi_1^{(0)}(X)), 
\end{equation}
where $\varphi^{(0)}(X) = \{ \pm \varphi_1^{(0)}(X) \}$ and  $\psi_j \in C^2(\op{graph} \varphi^{(0)} |_{B_1(0) \setminus B_{\tau_j/2}^{2}(0) \times \mathbb{R}^{n-2}};\mathbb{R}^m)$ is defined by 
\begin{equation} \label{lemma1_eqn9} 
	\psi_j(re^{i\theta},y,\op{Re}(c^{(0)} r^{\alpha} e^{i\alpha \theta})) = \op{Re}((c_j - c^{(0)}) r^{\alpha} e^{i\alpha \theta}). 
\end{equation}
By (\ref{lemma1_eqn6}), (\ref{lemma1_eqn7}), (\ref{lemma1_eqn8}) and (\ref{lemma1_eqn9}), there exists $w \in C^{2} \, ({\rm graph} \, \varphi^{(0)} \setminus \{0\} \times {\mathbb R}^{n-2} \times {\mathbb R}; {\mathbb R}^{m})$ with  the associated two-valued function $w(X,\varphi^{(0)}(X))$ harmonic on $B_1(0) \setminus \{0\} \times \mathbb{R}^{n-2}$ such that $\psi_j+ w_j$ converges to $w$ in $C^2(\op{graph} \varphi^{(0)} |_K;\mathbb{R}^m)$ for each compact set $K \subseteq B_1(0) \setminus \{0\} \times \mathbb{R}^{n-2}.$ By (\ref{lemma1_eqn4}) with $\r = 1/2$
\begin{equation*} 
	\int_{B_{1/4}(0) \cap B^2_{\delta}(0) \times \mathbb{R}^{n-2}} \mathcal{G}(u_j,\varphi_j)^2 
	\leq C \delta^{1/2} \int_{B_{1/2}(0)} \mathcal{G}(u_j,\varphi_j)^2 \leq C \delta^{1/2} E_j^2
\end{equation*}
for $C\delta_j \leq \delta <1/2$, so $\psi_j+ w_j$ converges to $w$ in $L^2(\op{graph} \varphi^{(0)} |_{B_{1/2}(0)};\mathbb{R}^m)$ and 
\begin{equation} \label{lemma1_eqn10} 
	\lim_{j \rightarrow \infty} \int_{B_{\rho}(0)} E_j^{-2} \mathcal{G}(u_j,\varphi_j)^2 = \int_{B_{\rho}(0)} |w(X,\varphi^{(0)}(X))|^2
\end{equation}
for each $\r \in (0, 1/2]$. 

Now (\ref{lemma1_eqn3}) implies that for each $\rho \in [\vartheta,1/4]$, 
\begin{equation} \label{lemma1_eqn11} 
	\int_{B_{\rho/2}(0)} R^{2-n} \left| \frac{\partial (w(X,\varphi^{(0)}(X))/R^{\alpha})}{\partial R} \right|^2 
	\leq C \rho^{-n-2\alpha} \int_{B_1(0)} |w(X,\varphi^{(0)}(X))|^2 
\end{equation}
where $C = C(n,m,\varphi^{(0)},\alpha) \in (0,\infty)$. We claim that for  each $\r \in (\vartheta, 1/4]$ and $z \in B_{\r/2}^{n-2}(0)$, there exist $\lambda_{\r}(z) = (\lambda_{1, \r}(z), \lambda_{2, \r}(z)) \in {\mathbb R}^{2}$ such that  
\begin{equation}\label{lemma1_eqn11-0}
 |\lambda_{\r}(z)|^{2} \leq C\r^{-n-2\a}\int_{B_{\r}(0)}|w(X, \varphi^{(0)}(X)|^{2} \;\;\; {\rm and}
 \end{equation} 
\begin{equation}\label{lemma1_eqn11-1}
\r^{n+2\a-\s}\int_{B_{\r/2}((0, z))} \frac{|w(X,\varphi^{(0)}(X)) -  \r D_x \varphi^{(0)}(X) \cdot \lambda_{\r}(z)|^2}{|X-(0,  z)|^{n+2\alpha-\sigma}} 
	\leq C_{1} \int_{B_{\r}(0)} |w(X, \varphi^{(0)}(X))|^2 
\end{equation}
where $C = C(n, m, \varphi^{(0)}, \a) \in (0, \infty)$ and $C_{1} = C_{1}(n, m. \varphi^{(0)}, \a, \s) \in (0, \infty).$
Indeed, given $z \in B_{\r/2}^{n-2}(0)$, if we choose $Z_{j} = (\xi_{j}, \z_{j}) \in B_{\d_{j}}((0, z))$ with 
${\mathcal N}_{u_{j}}(Z_{j}) \geq \a$, then by (\ref{lemma1_eqn4-0}), after passing to a subsequence, $E_{j}^{-1}\xi_{j} \to \lambda(z)$ for some $\lambda(z) \in {\mathbb R}^{2}$ satisfying 
$$\r^{-2}|\lambda(z)|^{2} \leq C \r^{-n-2\a}\int_{B_{\r}(0)}|w(X, \varphi^{(0)}(X))|^{2},$$ and by 
(\ref{lemma1_eqn4-1}), 
$$\r^{n+2\a-\s}\int_{B_{\r/2}((0, z))} \frac{|w(X,\varphi^{(0)}(X)) -  D_x \varphi^{(0)}(X) \cdot \lambda(z)|^2}{|X-(0,  z)|^{n+2\alpha-\sigma}} 
	\leq C_{1} \int_{B_{\r}(0)} |w(X, \varphi^{(0)}(X))|^2,$$ 
with both inequalities valid for $\r \in (\vartheta, 1/4]$. Note that $\lambda(z)$ satisfying the preceding inequality is unique (and hence is independent of the sequence $\{Z_{j}\}$) since if the inequality also holds with $\lambda^{\prime}(z) \in {\mathbb R}^{2}$ in place of $\lambda(z)$, then by the triangle inequality we have that $\int_{B_{1/8}((0, z))} \frac{|D_{x}\varphi^{(0)}(X) \cdot (\lambda(z) - \lambda^{\prime}(z))|^2}{|X-(0,  z)|^{n+2\alpha-\sigma}} < \infty,$ which in terms of the co-ordinates $(r, \th, y^{\prime})$ where $X - (0, z) = (x, y - z) = (re^{i\th}, y^{\prime}) \in {\mathbb R}^{2} \times {\mathbb R}^{n-2}$, implies that 
 $$\int_{B_{1/4}^{n-2}(0)}\int_{0}^{1/4}\int_{0}^{4\pi}\frac{r^{2\a - 1}|g(\th) \cdot(\lambda(z) - \lambda^{\prime}(z))|^{2}}{(r^{2} + |y^{\prime}|^{2})^{(n+2\a - \s)/2}} \, d\th drdy^{\prime} < \infty$$
 where $g= \left.D_{x}\varphi^{(0)}\right|_{r=1} = (\left. D_{1}\varphi^{(0)}\right|_{r=1}, \left.D_{2}\varphi^{(0)}\right|_{r=1});$  in view of $L^{2}([0, 4\pi])$ orthogonality of $\left. D_{1}\varphi^{(0)}\right|_{r=1}$ and $\left.D_{2}\varphi^{(0)}\right|_{r=1},$ this is impossible unless 
 $\lambda = \lambda^{\prime}.$ Setting $\lambda_{\r}(z) = \r^{-1}\lambda(z)$, we obtain (\ref{lemma1_eqn11-0}) and (\ref{lemma1_eqn11-1}) as claimed. Note that the constants $C$, $C_{1}$ in (\ref{lemma1_eqn11}), (\ref{lemma1_eqn11-0}) and (\ref{lemma1_eqn11-1}) are independent of $\vartheta$. 

Now let $\gamma \geq 1$ arbitrary.  Observe that the estimates above all hold if we replace $\varphi_j \in \widetilde{\Phi}_{\varepsilon_j}(\varphi^{(0)})$ with $\widetilde{\varphi}_j \in \widetilde{\Phi}_{\gamma \varepsilon_j}(\varphi^{(0)})$.  In this case we get $\tilde v_j$, $\tilde w$ in place of $v_j$, $w$ respectively and for $j$ sufficiently large depending on $\gamma$, (\ref{lemma1_eqn10}) and (\ref{lemma1_eqn11}) hold with $\widetilde{\varphi}_{j}$, $\widetilde{w}$ in place of $\varphi_{j},$ $w$, and (\ref{lemma1_eqn11-0}), (\ref{lemma1_eqn11-1}) hold with some function $\widetilde{\lambda}$ in place of $\lambda$ and $\widetilde{w}$ in place of $w$.  Note in particular the constants in these estimates do not depend on $\gamma$.  Recall that $\mathcal{L}$ is the set of functions on the graph of $\left.\varphi^{(0)}\right|_{B_{1}(0)}$ with values in $\mathbb{R}^m$ spanned by $a_0 r^{\alpha} \cos(\alpha \theta)$ and $b_0 r^{\alpha} \sin(\alpha \theta)$ for $a_0,b_0 \in \mathbb{R}^m$ and $D_i \varphi^{(0)}(re^{i\theta},y) y_j$ for $i = 1,2$ and $j = 1,\ldots,n-2$.  Let $\psi \in \mathcal{L}$ with $\sup_{B_1(0)} |\psi| \leq \beta$ for some $\beta \in (0,\infty)$ to be determined, and write $\psi$ as 
\begin{equation*}
	\psi((re^{i\theta},y),\op{Re}(c^{(0)} r^{\alpha} e^{i\alpha \theta})) 
	= \op{Re} (a r^{\alpha} e^{i\alpha \theta} + \alpha b_1 \cdot y c^{(0)} r^{\alpha-1} e^{i(\alpha-1) \theta} 
		+ i\alpha b_2 \cdot y c^{(0)} r^{\alpha-1} e^{i(\alpha-1) \theta})
\end{equation*} 
for some $a \in \mathbb{C}^m$ and $b_1,b_2 \in \mathbb{R}^{n-2}$.  Let $B = (B_{ij}) \in \mathcal{S}$ be the skew symmetric matrix whose first row is $(0,0,b_1)$ and second row is $(0,0,b_2)$.  Then we can find $\widetilde{\varphi}_j \in \widetilde{\Phi}_{\gamma \varepsilon_j}(\varphi^{(0)})$, where $\gamma = \gamma(\varphi^{(0)},\beta) \geq 1$, such that 
\begin{equation*}
	\widetilde{\varphi}_j(X) = \{ \pm (\varphi^{(0)}(X) + \psi_j(X,\varphi^{(0)}(X)) + E_j \psi(X,\varphi^{(0)}(X)) + \mathcal{R}_j) \}
\end{equation*}
on $B_1(0) \setminus B^2_{\tau_j}(0) \times \mathbb{R}^{n-2}$, where $|\mathcal{R}_j|/E_j \rightarrow 0$ uniformly on $B_1(0) \setminus B^2_{\tau}(0) \times \mathbb{R}^{n-2}$ for each $\t \in (0, 1/2)$ as $j \rightarrow \infty$, and in fact $\widetilde{\varphi}_j$ satisfies 
\begin{equation*}
	\widetilde{\varphi}_j(X) = \{ \pm (\op{Re}((c_j + E_j a)(x_1+ix_2)^{\alpha})) \circ e^{E_j B} \},  
\end{equation*}
i.e. equals the composition of the two-valued function $X \mapsto \{ \pm \op{Re}((c_j + E_j a)(x_1+ix_2)^{\alpha}) \}$ and the rotation $e^{E_j B}$.  Note that this follows using the assumption that $\tau_j \rightarrow 0$ slowly enough that $\varepsilon_j/\tau_j^{3/2} \rightarrow 0$.  Thus 
\begin{equation*}
	\tilde v_j(X,\widetilde{\varphi}_j(X)) = v_j(X,\varphi^{(0)}(X) + \psi_j(X,\varphi^{(0)}(X)) - E_j \psi(X,\varphi^{(0)}(X)) + \mathcal{R}'_j 
\end{equation*}
on $B_1(0) \setminus B^2_{\tau_j}(0) \times \mathbb{R}^{n-2}$, where $|\mathcal{R}'_j|/E_j \rightarrow 0$ uniformly on $B_1(0) \setminus B^2_{\tau}(0) \times \mathbb{R}^{n-2}$ for each $\t \in (0, 1/2)$ as $j \rightarrow \infty$.  Hence as $j \rightarrow \infty$, $\tilde w_j = \tilde v_j/E_j$ converges to $\tilde w = w - \psi$.  In place of (\ref{lemma1_eqn10})-(\ref{lemma1_eqn11-1}), we have for each $\rho \in [\vartheta,1/4]$ and $z \in B_{\r/2}^{n-2}(0)$,
\begin{equation} \label{lemma1_eqn15} 
	\lim_{j \rightarrow \infty} \int_{B_{\rho}(0)} E_j^{-2} \mathcal{G}(u_j,\widetilde{\varphi}_j)^2 
	= \int_{B_{\rho}(0)} |w(X,\varphi^{(0)}(X)) - \psi(X,\varphi^{(0)}(X))|^2; 
\end{equation}
\begin{equation} \label{lemma1_eqn16} 
	\int_{B_{\rho/2}(0)} R^{2-n} \left| \frac{\partial (w(X,\varphi^{(0)}(X))/R^{\alpha})}{\partial R} \right|^2 
	\leq C \rho^{-n-2\alpha} \int_{B_1(0)} |w(X,\varphi^{(0)}(X)) - \psi(X,\varphi^{(0)}(X))|^2; 
\end{equation}
\begin{eqnarray}\label{lemma1_eqn16-1}
&&\hspace{-.7in}\r^{n+2\a-\s}\int_{B_{\r/2}((0, z))} \frac{|w(X,\varphi^{(0)}(X))  - \psi(X, \varphi^{(0)}(X)) -  \r D_x \varphi^{(0)}(X) \cdot \lambda_{\psi, \r}(z)|^2}{|X-(0,  z)|^{n+2\alpha-\sigma}}\nonumber\\ 
&&\hspace{2in}	\leq C_{1} \int_{B_{\r}(0)} |w(X, \varphi^{(0)}(X)) - \psi(X, \varphi^{(0)}(X))|^2 
\end{eqnarray}
for $\lambda_{\psi, \r}(z) \in {\mathbb R}^{2}$ with 
\begin{equation}\label{lemma1_eqn16-0}
 |\lambda_{\psi, \r}(z)|^{2} \leq C\r^{-n-2\a}\int_{B_{\r}(0)}|w(X, \varphi^{(0)}(X)|^{2}
 \end{equation} 
where $C = C(n,m,\varphi^{(0)},\alpha) \in (0,\infty)$ and $C_{1} = C_{1}(n, m, \varphi^{(0)}, \a,\s) \in (0, \infty)$.  Note that (\ref{lemma1_eqn15})-(\ref{lemma1_eqn16-0}) hold true for arbitrary $\rho \in [\vartheta,1]$ and arbitrary $\beta \geq 1$ with constants $C,$ $C_{1}$ independent of $\vartheta$ and $\beta$. 

Now for given $\rho \in [\vartheta,1/4]$, choose $\psi_{\rho} \in \mathcal{L}$ to be the homogeneous degree $\alpha$ function such that 
\begin{equation*}
	\int_{B_{\rho}(0)} |w(X,\varphi^{(0)}(X)) - \psi_{\rho}(X,\varphi^{(0)}(X))|^2 
	= \inf_{\psi \in \mathcal{L}} \int_{B_{\rho}(0)} |w(X,\varphi^{(0)}(X)) - \psi(X,\varphi^{(0)}(X))|^2
\end{equation*}
so that $w-\psi_{\rho}$ is $L^{2}$-orthogonal to ${\mathcal L}$ on $B_{\r}(0)$.  Since $\int_{B_1(0)} |w(X,\varphi^{(0)}(X))|^2 \leq 1$, 
\begin{equation*}
	\rho^{-n-2\alpha} \int_{B_{\rho}(0)} |\psi_{\rho}(X,\varphi^{(0)}(X))|^2 \leq 4\vartheta^{-n-2\alpha}. 
\end{equation*}
Thus by standard estimates for single-valued harmonic functions and homogeneity, 
\begin{equation*}
	\sup_{B_1(0)} |\psi_{\rho}(X,\varphi^{(0)}(X))| \leq C \vartheta^{-n/2-\alpha} 
\end{equation*}
for some $C = C(n,m) \in (0,\infty)$ and all $\rho \in [\vartheta,1/4]$.  Thus 
(\ref{lemma1_eqn15})-(\ref{lemma1_eqn16-0}) hold with $\beta = C\vartheta^{-n/2-\alpha}$, $\gamma = \gamma(n,m,\varphi^{(0)},\alpha,\vartheta) \geq 1$, and $\psi = \psi_{\rho}.$ Consequently,  if $\alpha \geq 1$,  we may apply Lemma \ref{lemma4_14} (with $\sigma = 1/2$) to conclude that  
\begin{equation*}
	\vartheta^{-n-2\alpha} \int_{B_{\vartheta}(0)} |w(X,\varphi^{(0)}(X)) - \psi_{\vartheta}(X,\varphi^{(0)}(X))|^2 
	\leq C \vartheta^{2\mu} \int_{B_1(0)} |w(X,\varphi^{(0)}(X))|^2 \leq C \vartheta^{2\mu}
\end{equation*}
for some constants $\mu \in (0,1)$ and $C \in (0,\infty)$ depending only on $n$, $m$, $\alpha$, $\varphi^{(0)}$, where we used the fact that $\int_{B_1(0)} |w(X,\varphi^{(0)}(X))|^2 = 1$.  By (\ref{lemma1_eqn15}) with $\rho = \vartheta$, we get that 
\begin{equation*}
	\vartheta^{-n-2\alpha} \int_{B_{\vartheta}(0)} \mathcal{G}(u_j,\widetilde{\varphi}_j)^2 \leq 2C \vartheta^{2\mu} E_j^2, 
\end{equation*}
for $j$ sufficiently large, which is (\ref{lemma1_eqn1}) as required, in case $\a \geq 1$. 

To complete the proof in case $\a = 1/2$, we only need to show the additional fact that $w$ satisfies hypothesis  (\ref{lemma4_2_alphaishalf}) of Lemma \ref{lemma4_2}, for then, we may apply Lemma~\ref{lemma4_14} exactly as above to reach the preceding conclusion for sufficiently large $j$. So suppose $\alpha = 1/2$.  We may apply Corollary~\ref{alpha=1/2} with $\t>0$ arbitrary and 
with $u_{j}, \varphi_{j}$ in place of $u$, $\varphi$ for sufficiently large $j,$ divide both sides of the resulting inequality by $E_{j}$ and first let $j \to \infty$ and then let $\t \to 0$ to conclude, 
keeping in mind that $D_{y_{q}} \, w \in L^{2}({\rm graph} \, \varphi^{(0)}; {\mathbb R}^{m})$ for each $q=1, 2, \ldots, n-2$, that 
$$\int_{B_{1/2}(0)} D_{j}\left(rw(re^{i\th},y, \varphi^{(0)}(re^{i\th})) \cdot D_{i} \varphi^{(0)}(re^{i\th})\right) D_{j}D_{y_{p}}\z \, dr d\theta dy = 0$$
for any $\z = \widetilde\z(|x|, y)$ where $\widetilde{\z}  = \widetilde{\z}(r, y) \in C^{\infty}_{c}(B^{n-1}_{1/2}(0))$ with $D_{r}\widetilde\z = 0$ in a neighborhood of $r = 0.$  It follows from this  that the even extension in the $r$ variable of the function 
 $$(r, y) \mapsto r\int_{0}^{4\pi} D_{y_{p}}w(re^{i\th}, y, \varphi^{(0)}(re^{i\th})) \cdot D_{j} \varphi^{(0)}(re^{i\th})d\th$$ is harmonic on $B_{1/2}^{n-1}(0)$, so in particular (\ref{lemma4_2_alphaishalf}) holds. The proof of the lemma is thus complete.
\end{proof}

The following result can now be deduced by iteratively applying  Lemma~\ref{lemma1} the same way as in the proof of Theorem 1 of \cite{SimonCTC}. We sketch its proof for the reader's  convenience.  

\begin{proposition} \label{theorem1} let $\varphi^{(0)}$ be as in  (\ref{varphi0}). Suppose either  
(a) $\varphi^{(0)} \in W^{1,2}(\mathbb{R}^n;\mathcal{A}_2(\mathbb{R}^m))$ is 
Dirichlet energy minimizing  or,   (b) $\varphi^{(0)} \in C^{1,1/2}(\mathbb{R}^n;\mathcal{A}_2(\mathbb{R}^m))$ and $D\varphi^{(0)}(0) = \{0, 0\}$.
 There are $\varepsilon, \delta_0, \mu \in (0,1)$ depending only on $n$, $m$, and $\varphi^{(0)}$ such that  if either
\begin{itemize}
\item[(1)] $u \in \mathcal{F}_{\varepsilon}^{\text{Dir}}(\varphi^{(0)})$ where $\varphi^{(0)}$ is as in (a), or 
\item[(2)] $u \in \mathcal{F}^{\text{Harm}}_{\varepsilon}(\varphi^{(0)})$ where $\varphi^{(0)}$ is as in (b), 
\end{itemize}
then 
\begin{equation*} 
	\{ X \in \Sigma_u \cap B_{1/2}(0) : \mathcal{N}_u(X) \geq \alpha \} = S \cup T \;\;\; \mbox{in case of (1) and}
\end{equation*}
\begin{equation*} 
	\{ X \in {\mathcal K}_u \cap B_{1/2}(0) : \mathcal{N}_u(X) \geq \alpha \} = S \cup T \;\;\; \mbox{in case of (2)},
\end{equation*}
where $S$ is contained in a properly embedded $(n-2)$-dimensional $C^{1,\mu}$ submanifold $\Gamma$ of $B_{1/2}(0)$ with $\mathcal{H}^{n-2}(\Gamma \cap B_{1/2}(0)) \leq \omega_{n-2}$ and $T  \subseteq \bigcup_{j=1}^{\infty} B_{\rho_j}(X_j)$ for some family of balls $B_{\rho_j}(X_j)$ with $\sum_{j} \rho_j^{n-2} \leq 1-\delta_0$. 
\end{proposition}

\begin{proof} Consider case (1). The proof in case (2) is similar. Choose $\vartheta \in (0,1/8)$ such that $C \vartheta^{2\mu} < 1/4$ for $C$ and $\mu$ as in Lemma \ref{lemma1} and let $\varepsilon_0$ and $\delta_0$ be as in Lemma \ref{lemma1}.  We can assume $\varepsilon_0$ is smaller than $\varepsilon_0$ of Theorem~\ref{thm6_2} and Corollaries~\ref{cor6_3}-\ref{cor6_6} with $\gamma$ and $\sigma$ both equal to $1/2$ and $\tau = 1/100$.  Suppose $u \in \mathcal{F}^{\text{Dir}}_{\varepsilon}(\varphi^{(0)})$ for $C\varepsilon \leq \varepsilon_0$ for $C = C((n,m,\alpha,\varphi^{(0)}) \geq 1$ to be chosen.  Define 
\begin{equation*}
	\Sigma_u^* = \{ X \in B_{1/2}(0) : \mathcal{N}_u(X) \geq \alpha \}. 
\end{equation*}
We define sets $\Upsilon_0,\Upsilon_1,\Upsilon_2,\Upsilon_3,\ldots$ and $\Upsilon_{\infty}$ as follows: Define $\Upsilon_0$ to be the set of $Y \in \Sigma_u^*$ such that $u(Y+\rho X)$ satisfies alternative (i) to Lemma \ref{lemma1} for some $\rho \in [\vartheta,1]$.  For $j \geq 1$, let $\Upsilon_j$ be the set of $Y \in \Sigma_u^*$ such that $u(Y+\vartheta^i X)$ does not satisfy alternative (i) to Lemma \ref{lemma1} for $i = 0,1,\ldots,j$ and $u(Y+\theta^{j+1} X)$ satisfies alternative (i) to Lemma \ref{lemma1}.  Let $\Upsilon_{\infty}$ be the set of $Y \in \Sigma_u^*$ such that $u(Y+\vartheta^i X)$ does not satisfy alternative (i) to Lemma \ref{lemma1} for all $i = 0,1,2,\ldots$.  

If $\Upsilon_0 \neq \emptyset$, by Lemma \ref{coarse_graph}, $\Sigma_u^* \cap B_1(0) \subseteq B^2_{\delta(\varepsilon)}(0) \times \mathbb{R}^{n-2}$ for some $\delta(\varepsilon)$ such that $\delta(\varepsilon) \rightarrow 0$ as $\varepsilon \downarrow 0$, thus we trivially have Theorem \ref{theorem1} for $S = \emptyset$ and $T = \Sigma_u^*$.  From now on we may assume that $\Upsilon_0 = \emptyset$.  We want to show that Theorem \ref{theorem1} holds with $S = \Upsilon_{\infty}$ and $T = (\Sigma_u^* \setminus \Upsilon_{\infty}) \cup (\Sigma_u \cap B_1(0) \setminus B_{1/2}(0))$. 

Suppose $Y \in \Upsilon_{\infty}$.  By inductively applying alternative (ii) in Lemma \ref{lemma1}, we define $\varphi_i \in \widetilde{\Phi}_{\varepsilon_0}(\varphi^{(0)})$ for $i = 0,1,2,\ldots$ as follows.  Let $\varphi_0 = \varphi^{(0)}$.  For each $i = 1,2,\ldots$, choose $\varphi_i \in \widetilde{\Phi}_{\gamma \varepsilon_0}(\varphi^{(0)})$ such that 
\begin{equation} \label{thm1_eqn1}
	\vartheta^{-(n+2\alpha)i} \int_{B_{\vartheta^i}(0)} \mathcal{G}(u(Y+X),\varphi_i(X))^2 dX
	\leq \frac{1}{4} \vartheta^{-(n+2\alpha)(i-1)} \int_{B_{\vartheta^{i-1}}(0)} \mathcal{G}(u(Y+X),\varphi_{i-1}(X))^2 dX.    
\end{equation}
By the definition of $\widetilde{\Phi}_{\gamma \varepsilon_0}(\varphi^{(0)})$, there is a rotation $q_i$ of $\mathbb{R}^n$ given by $q_i = e^{A_i}$ for some skew symmetric matrix $A_i \in \mathcal{S}$ with $|A_i| \leq \gamma \varepsilon_0$ such that $\varphi_i(q_i X) \in \Phi_{\gamma \varepsilon_0}(\varphi^{(0)})$.  Note that by (\ref{thm1_eqn1}), 
\begin{equation} \label{thm1_eqn2}
	\vartheta^{-(n+2\alpha)i} \int_{B_{\vartheta^i}(0)} \mathcal{G}(u(Y+X),\varphi_i(X))^2 dX 
	\leq 4^{-i} \int_{B_1(0)} \mathcal{G}(u(Y+X),\varphi^{(0)}(X))^2 dX \leq 4^{-i} \varepsilon^2 
\end{equation}
and thus 
\begin{equation} \label{thm1_eqn3}
	\int_{B_1(0)} \mathcal{G}(\varphi_i(X),\varphi_{i-1}(X))^2 dX \leq C 4^{-i} \varepsilon^2
\end{equation}
for some constant $C = C(n,m,\alpha,\varphi^{(0)}) \in (0,\infty)$.  Thus 
\begin{equation} \label{thm1_eqn4}
	\int_{B_1(0)} \mathcal{G}(\varphi_i(X),\varphi^{(0)}(X))^2 dX \leq C 4^{-i} \varepsilon^2 < \varepsilon_0^2
\end{equation}
for some constant $C = C(n,m,\alpha,\varphi^{(0)}) \in (0,\infty)$, using our assumption that $C\varepsilon \leq \varepsilon_0$ for some constant $C \in (0,\infty)$, so $\varphi_i \in \widetilde{\Phi}_{\varepsilon_0}(\varphi^{(0)})$.  By (\ref{thm1_eqn3}), $\varphi_i$ converges in $L^2(B_1(0),\mathbb{R}^m)$ to some $\varphi_Y \in \widetilde{\Phi}_{\varepsilon_0}(\varphi^{(0)})$ and 
\begin{equation*}
	\int_{B_1(0)} \mathcal{G}(\varphi_i(X),\varphi_Y(X))^2 dX \leq C 4^{-i} \varepsilon^2
\end{equation*}
for some constant $C = C(n,m,\alpha,\varphi^{(0)}) \in (0,\infty)$.  Thus by (\ref{thm1_eqn2}),  
\begin{equation*}
	\vartheta^{-(n+2\alpha)i} \int_{B_{\vartheta^i}(0)} \mathcal{G}(u(Y+X),\varphi_Y(X))^2 dX 
	\leq C 4^{-i} \varepsilon^2
\end{equation*}
for some constant $C = C(n,m,\alpha,\varphi^{(0)}) \in (0,\infty)$ for $i = 1,2,3,\ldots$.  Given $\rho \in (0,1]$, choose $i$ such that $\vartheta^{i-1} < \rho \leq \vartheta^i$ to get 
\begin{equation} \label{thm1_eqn5}
	\rho^{-n-2\alpha} \int_{B_{\rho}(0)} \mathcal{G}(u(Y+X),\varphi_Y(X))^2 dX 
	\leq C \rho^{2\bar \mu} \varepsilon^2,
\end{equation}
where $\bar \mu = -\log \vartheta/\log 2$ and $C = C(n,m,\alpha,\varphi^{(0)}) \in (0,\infty)$.  Clearly by (\ref{thm1_eqn5}), $\varphi_Y$ is unique for each $Y \in \Upsilon_{\infty}$.  Observe that by (\ref{thm1_eqn4}) $\varphi_i \in \widetilde{\Phi}_{C\varepsilon}(\varphi^{(0)})$ there is a rotation $q_Y$ of $\mathbb{R}^n$ such that $\varphi_Y(q_Y X) \in \Phi_{\varepsilon_0}(\varphi^{(0)})$ and $|q_Y - I| \leq C\varepsilon$.  By the estimate on $|\xi|^2$ in Corollary~\ref{cor6_4} and by (\ref{thm1_eqn5}), 
\begin{equation*} 
	\rho^{-1-1\bar \mu} \op{dist} \, (q_Y^{-1}(\Sigma_u^* - Y) \cap B_{\rho}(0), \{0\} \times \mathbb{R}^{n-2})\leq C\varepsilon. 
\end{equation*}
Thus by a standard argument based on (\ref{thm1_eqn2}) and (\ref{thm1_eqn3}), 
\begin{equation*}
	|q_{Y_1} - q_{Y_2}| \leq C \varepsilon |Y_1 - Y_2|^{\bar \mu}
\end{equation*}
for every $Y_1,Y_2 \in \Upsilon_{\infty}$.  Thus $\Upsilon_{\infty} = \op{graph} f \cap B_{1/2}(0)$ is the graph of a function $f \in C^{1,\bar \mu}(B^{n-2}_{1/2}(0),\mathbb{R}^2)$ such that $\|f\|_{C^{1,\bar \mu}(B^{n-2}_{1/2}(0))} \leq C\varepsilon$. 

Now suppose $Y \in \Upsilon_j$ for some $1 \leq j < \infty$.  Now iteratively applying alternative (ii) in Lemma \ref{lemma1} only gives us $\varphi_i \in \widetilde{\Phi}_{\gamma \varepsilon_0}(\varphi^{(0)})$ for $i = 0,1,2,\ldots,j$ such that $\varphi_0 = \varphi^{(0)}$ and (\ref{thm1_eqn1}) holds for $i = 1,2,\ldots,j$.  Take $\varphi_Y = \varphi_j$.  Note that $\varphi_Y$ is no longer unique.  By the argument above, we obtain for $\rho \in [\vartheta^j,1]$ that 
\begin{equation*}
	\rho^{-1-\bar \mu} \op{dist}\, (q_Y(\Sigma_u^* - Y) \cap B_{\rho}(0), \{0\} \times \mathbb{R}^{n-2}) \leq C\varepsilon. 
\end{equation*}
where $C = C(n,m,\alpha,\varphi^{(0)}) \in (0,\infty)$ is a constant and $\bar \mu = -\log \vartheta/\log 2$ as above, and for some rotation $q_Y$ of $\mathbb{R}^n$, $\varphi_Y(q_Y X) \in \Phi_{\gamma \varepsilon_0}(\varphi^{(0)})$ and $|q_Y - I| \leq C\varepsilon$.  Hence 
\begin{equation} \label{thm1_eqn6}
	\rho^{-1} \op{dist} \, (\Sigma_u^* \cap B_{\rho}(Y), Y + \{0\} \times \mathbb{R}^{n-2})\leq C\varepsilon. 
\end{equation}
for every $Y \in \Upsilon_j$ and $\rho \in [\vartheta^j,1]$.  By the definition of $\Upsilon_j$, $u(Y+\vartheta^{j+1}X)$ satisfies alternative (i) in Lemma \ref{lemma1}, i.e. 
\begin{equation} \label{thm1_eqn7}
	\forall Y \in \Upsilon_j, \, \exists Z \in Y + \{0\} \times B_{\vartheta^{j+1}}^{n-2}(0) \text{ such that } B_{\delta_0 \vartheta^{j+1}}(Z) \cap \Sigma_u^* = \emptyset. 
\end{equation}
Now arguing exactly as in pp. 642-643 of ~\cite{SimonCTC} (using (\ref{thm1_eqn6}),(\ref{thm1_eqn7}) in place of (12), (13) on p. 642 of \cite{SimonCTC}) we obtain a covering of $\bigcup_{1 \leq i < \infty} \Upsilon_i$ by balls $B_{\rho_j}(X_j)$, $j = 1,2, 3, \ldots,$ such that $\sum_{j} \rho_j^{n-2} \leq 1-\delta_0$. 
\end{proof}

\begin{proof}[Proof of Theorem~\ref{no-gaps}]
Consider the case of two-valued Dirichlet energy minimizers. The proof in the case of two-valued $C^{1, \mu}$ harmonic functions is similar.  Let $\varphi^{(0)}$ be as in (\ref{varphi0}) and Dirichlet energy minimizing, with $\alpha = \alpha_{0}$ where $\alpha_{0}$ is the minimum possible degree of homogeneity for such $\varphi^{(0)}$  (thus $\alpha_{0} = 1/2$). We claim that for each $\d \in (0, 1/2)$, there exists $\varepsilon = \varepsilon(n, m, \d, \varphi^{(0)}) \in (0, 1/2)$ such that if  
$u \in W^{1, 2} \, (B_{2}(0); {\mathcal A}_{2}({\mathbb R}^{m}))$ is a symmetric Dirichlet energy minimizing function with 
$\int_{B_{2}(0)} {\mathcal G} (u, \varphi^{(0)})^{2} < \varepsilon$, then 
$\{X ,\ : \, {\mathcal N}_{u}(X) \geq \alpha_{0}\} \cap B_{\d}(0, y) \neq \emptyset$ for each $y \in B_{1}^{n-2}(0).$ To see this, first note that we may choose such $\varepsilon$ so that 
${\mathcal B}_{u} \cap B_{\d}(0,y) \neq \emptyset$ for each $y \in B_{1}^{n-2}(0).$ It is then automatic that  ${\mathcal H}^{n-2} \, ({\mathcal B}_{u} \cap B_{\d}(0,y)) > 0$ for each $y \in B_{1}^{n-2}(0),$ for if 
not, we can find $y \in B_{1}^{n-2}(0)$ and $Z \in {\mathcal B}_{u} \cap B_{\d}(0, y)$  such that 
for $\r \in (0, {\rm dist} \, (Z, \partial \, B_{\d}(0, y)),$ $B_{\r}(Z) \setminus B_{u}$ is simply connected
(see Appendix to \cite{SW}),  and hence $\left.u\right|_{B_{\r}(Z)}$ is given by two single-valued harmonic functions contradicting the assumption that $Z \in {\mathcal B}_{u}$. Thus 
${\mathcal H}^{n-2} \, ({\mathcal B}_{u} \cap B_{\d}(0,y)) > 0$ for each $y \in B_{1}^{n-2}(0)$ and 
hence by Lemma~\ref{stratification_lemma} and the definition of $\alpha_{0}$, $\{X ,\ : \, {\mathcal N}_{u}(X) \geq \alpha_{0}\} \cap B_{\d}(0, y) \neq \emptyset$ for each $y \in B_{1}^{n-2}(0)$ as claimed. 

Now given $Z \in B_{2}(0)$ with ${\mathcal N}_{v}(Z) = \alpha_{0}$, let $\varphi^{(Z)}$ be any blow-up of $v = u - u_{a}$ at $Z$. Since ${\mathcal N}_{\varphi^{(Z)}}(0)$ ( $= \a_{0}$) is not an integer, we have that 
$0 \in {\mathcal B}_{\varphi^{(Z)}}$ and hence (reasoning as above)  ${\mathcal H}^{n-2} \, ({\mathcal B}_{\varphi^{(Z)}} \cap B_{1}(0)) >0.$ By Lemma~\ref{stratification_lemma}, it then follows that 
${\mathcal H}^{n-2} \, (\{X \, : \, {\mathcal N}_{\varphi^{(Z)}}(X) \geq \a_{0}\} \cap B_{1}(0)) >0.$ On the other hand  
${\mathcal N}_{\varphi^{(Z)}}(X) \leq {\mathcal N}_{\varphi^{(Z)}}(0) = \a_{0}$ for any $X$, so this implies that 
${\mathcal H}^{n-2} \, (\{X \, : \, {\mathcal N}_{\varphi^{(Z)}}(X) = \a_{0}\} \cap B_{1}(0)) >0$;  in particular there are $(n-2)$ linearly independent vectors $X_{1}, X_{2}, \ldots, X_{n-2} \in B_{1}(0)$ with 
${\mathcal N}_{\varphi^{(Z)}}(X_{j}) = {\mathcal N}_{\varphi^{(Z)}}(0)$ implying that $\varphi^{(Z)}$ is cylindrical. Choosing now $\vartheta  = \vartheta(n, m, \varphi^{(Z)}) \in (0, 1/8)$ such that $C \vartheta^{2\mu} < 1/4$ and taking $\d = \d_{0}(n, m, \varphi^{(Z)})$ in the claim  of the preceding paragraph, where $C = C(n, m, \varphi^{(Z)})$, $\mu = \mu(n, m, \varphi^{(Z)})$ and $\d_{0}(n, m, \varphi^{(Z)})$ are as in Lemma~\ref{lemma1} (taken with $\varphi^{(0)} = \varphi^{(Z)}$),  we can then argue as in the proof of Proposition~\ref{theorem1}, 
using the claim in the preceding paragraph to rule out, in each application of Lemma~\ref{lemma1} in that argument, option (i) of  Lemma~\ref{lemma1}. This leads to the desired conclusion that ${\mathcal B}_{u}$ is an $(n-2)$-dimensional $C^{1,\a}$ submanifold near $Z$. 
\end{proof}

\begin{proof}[Proof of Theorem~\ref{theorem2}] Similar to the proof of Theorem~$2^{\prime}$ in~\cite{SimonCTC}, so we will only sketch the proof here.  Since $\op{dim} \Sigma^{(n-3)}_u \leq n-3$, it suffices to consider the set $\Sigma^*_u$ of points $Y \in \Sigma_u$ at which there is at least one blow up $\varphi$ with $\op{dim} S(\varphi) = n-2$.  Let $Y_0 \in \Sigma^*_u$ and $\varphi^{(0)}$ be a blow up of $u$ at $Y_0$.  By the definition of blow up we see that for every $\varepsilon > 0$ there exists a $\sigma = \sigma(\varepsilon) > 0$ such that $B_{R(\varepsilon)}(0) \subset B_{(1-|Y_0|)/\sigma}(Y_0)$ and 
\begin{equation*}
	\int_{B_1(0)} \mathcal{G}(u_{Y_0,\sigma},\varphi^{(0)})^2 < \varepsilon^2, \quad 
	N_{u_{Y_0,\sigma},0}(R(\varepsilon)) - \alpha < \delta(\varepsilon),
\end{equation*}
where $R(\varepsilon)$ and $\delta(\varepsilon)$ are as in Lemma \ref{lemma2_4}.  Let $\bar u = u_{Y_0,\sigma}$.  Given $\rho_0 > 0$, define the outer measure 
\begin{equation*}
	\mu_{\rho_0}(A) = \inf \sum_{i=1}^N \omega_{n-2} \sigma_i^{n-2}
\end{equation*}
where the infimum is taken over all finite covers of $A$ by balls $B^{n+m}_{\sigma_i}(q_i)$, $i = 1,\ldots,N$, with $\sigma_i \leq \rho_0$ for all $i$.  Choose a finite collection of balls $B^{n+m}_{\sigma_i}(q_i)$, $i = 1,\ldots,N$, that cover $\Sigma^{+}_{\a} \cap \overline{B_{1}(0)},$ where $\Sigma^+_{\alpha} = \Sigma_{\bar u} \cap \{ Y : \mathcal{N}_{\bar u}(Y) \geq \alpha \},$ such that $\sigma_i \leq \rho_0$ for all $i$ and 
\begin{equation*}
	\sum_{i=1}^N \omega_{n-2} \sigma_i^{n-2} \leq \mu_{\rho_0}(\Sigma^+_{\alpha}) + 1. 
\end{equation*}
By removing balls $B_{\sigma_i}(Y_i)$ that do not intersect $\Sigma^+_{\alpha}$, we may assume $B_{\sigma_i}(Y_i) \cap \Sigma^+_{\alpha} \neq \emptyset$ for all $i = 1,\ldots,N$.  Let $Z_i \in B_{\sigma_i}(Y_i) \cap \Sigma^+_{\alpha}$ for each $i \in \{1, 2, \ldots, N\}$.  By Lemma \ref{lemma2_4}, for suitably small $\varepsilon > 0$ and each $i$, either there is a homogeneous degree $\alpha$, symmetric Dirichlet energy minimizing function $\varphi_i \in W^{1,2}_{\rm loc}(\mathbb{R}^n;{\mathcal A}_{2}(\mathbb{R}^m))$ such that $\dim S(\varphi_{i}) = n-2$ and 
\begin{equation} \label{thm2_eqn1}
	\int_{B_1(0)} \mathcal{G}(\bar u_{Y_i,2\rho_i},\varphi_i)^2 < \varepsilon^2 
\end{equation}
or 
\begin{equation} \label{thm2_eqn2}
	\{ X \in \Sigma_{\bar u} \overline{B_{2\rho_i}(Y_i)} : \mathcal{N}_{\bar u}(X) \geq \alpha \} \subset \{ X : \op{dist}(X,Y_i+L) < \varepsilon \} 
\end{equation}
for some $(n-3)$-dimensional subspace $L$.  Note that this uses the fact that for every homogeneous degree $\alpha$ Dirichlet minimizing symmetric two-valued function $\varphi \in W^{1,2}(\mathbb{R}^n;\mathbb{R}^m)$ such that $\dim S(\varphi) = n-2$, $\alpha = k/2$ for some integer $k \geq 1$ and thus the set of all such $\alpha$ is discrete.  Using the fact that either (\ref{thm2_eqn1}) or (\ref{thm2_eqn2}) hold, iteratively apply Theorem \ref{theorem1} exactly as in the argument in~\cite{SimonCTC} to reach the desired conclusions.
\end{proof}

\begin{proof}[Proof of Theorem~\ref{theorem3}]
Argue as for the proof of Theorem~\ref{theorem2}, making use of the version of Proposition~\ref{theorem1} for two-valued $C^{1, \mu}$ harmonic functions. 
\end{proof}

\begin{proof}[Proof of Theorem A of the introduction]
Let $\Omega \subset {\mathbb R}^{n}$ be open and consider the case that $u \in W^{1, 2} \, (\Omega; {\mathcal A}_{2}({\mathbb R}^{m})$ is Dirichlet energy minimizing. The proof in case $u$ is $C^{1, \mu}$ and harmonic is similar. Let $B \subset \Omega$ be a closed ball. By Theorem~\ref{theorem2}, there exists a finite set $\{\a_{1}, \a_{2}, \ldots, \a_{k}\} \subset \{1/2, 1, 3/2, 2, \ldots\}$ such that if $Z \in B \cap \Sigma_{u}$ and $v$ has a cylindrical blow-up at $Z$, then ${\mathcal N}_{v}(Z)  = \a_{j}$ for some $j \in \{1, 2, \ldots, k\}.$ 
Let $V_{j} = V_{\a_{j}}$ for $j= 1, 2, \ldots, k$ where $V_{\a}$ is the open set given by Theorem~\ref{theorem2}, so that $V_{j} \supset \{Z \in \Omega \, : \, {\mathcal N}_{v}(Z) = \a_{j} \; \mbox{and $v$ has a cylindrical blow-up at $Z$}\}$. Set $\a_{0} = 1/2$, $\a_{k+1} = \infty$ and let, for $j=0, 1, 2, \ldots, k$, 
$$\G_{j} = \{Z \in B \cap \Sigma_{u} \, : \, \a_{j} \leq {\mathcal N}_{v}(Z) < \a_{j+1}\} \cap V_{j}$$ 
and 
$$\widetilde{\G}_{j} = \{Z \in B \cap \Sigma_{u} \, : \, \a_{j} \leq {\mathcal N}_{v}(Z) < \a_{j+1}\} \setminus V_{j}.$$
Then by Theorem~\ref{theorem2}, $\G_{j}$ is locally $(n-2)$-rectifiable (with locally finite measure), 
and by  Lemma~\ref{stratification_lemma}, ${\rm dim}_{\mathcal H} \, (\widetilde{\G}_{j}) \leq n-3.$ 
Moreover, by upper semi-continuity of ${\mathcal N}_{v}(\cdot)$, each of $\G_{j}$, $\widetilde{\G}_{j}$ is the intersection of an open set and a closed set and hence is locally compact.  And of course
$B \cap \Sigma_{u} = \cup_{j=0}^{k} \G_{j} \cup {\widetilde \G}_{j}.$ 
\end{proof}

\begin{proof}[Proof of Theorem B  of the introduction] Again consider the case of two-valued Dirichlet energy minimizing functions. The proof in the case of $C^{1, \mu}$ harmonic functions is similar. Set $h = u_{a}$, the average of the two values of $u$, which by [\cite{Almgren}, Theorem~2.6] is a single-valued harmonic function, and let $v = u - h.$ It follows from the argument of Theorem~\ref{theorem2} and Theorem~\ref{theorem3} that corresponding to  ${\mathcal H}^{n-2}$-a.e. point $Z \in {\mathcal B}_{u},$ there exists a symmetric two-valued homogeneous cylindrical energy minimizing function 
$\varphi^{(Z)}$ and  a number $\r_{Z}  \in (0, 1/2)$ such that for appropriate $\vartheta \in (0, 1/8)$ depending on $\varphi^{(Z)}$, the hypotheses of Lemma~\ref{lemma1} are satisfied with 
$\varphi^{(Z)}$ in place of $\varphi^{(0)}$ and $v_{Z, \r_{Z}}  = v(Z + \r_{Z} (\cdot))/\|v(Z + \r_{Z}(\cdot))\|_{L^{2}(B_{1}(0))}$ in place of $u$, but option (i) of the conclusion of Lemma~\ref{lemma1} fails with  $\varphi^{(Z)}$ in place of $\varphi^{(0)}$ and $v_{Z, \r_{Z}}(\vartheta^{j}(\cdot))$ in place of $u$  for each $j=0, 1, 2, \ldots.$ Hence, as in the argument of Proposition~\ref{theorem1}, Lemma~\ref{lemma1}  
leads to the desired unique blow-up $\widetilde{\varphi}^{(Z)}$ at any such $Z$, given by 
 $\widetilde{\varphi}^{(Z)}(Q_{Z}(X)) = \{\pm{\rm Re} \, \left(c_{Z}(x_{1} + ix_{2})^{k_{Z}/2}\right)\}$  for some orthogonal rotation $Q_{Z}$ of ${\mathbb R}^{n}$ and $c_{Z} \in {\mathbb C}^{m} \setminus \{0\};$ moreover, Lemma~\ref{lemma1} also gives the $L^{2}$ decay estimate, as stated in the conclusion of the theorem, for the error $\{\pm \e_{Z}\}$. 
 \end{proof}

\begin{proof}[Proof of Theorem C of the introduction] Let $Z_{0}$ be as in the theorem. By Theorem~\ref{no-gaps} and its proof, there exists $\r_{Z_{0}} >0$ such that ${\mathcal B}_{u} \cap B_{\r_{Z_{0}}}(Z_{0})$ is an $(n-2)$-dimensional $C^{1, \a}$ submanifold, and for each $Z \in {\mathcal B}_{u} \cap B_{\r_{Z_{0}}}(Z_{0})$, the conclusions of Theorem B hold with with $k_{Z} = k_{Z_{0}}$, $C_{Z} = C_{Z_{0}}$, $\g_{Z} = \g_{Z_{0}}$ and $\r_{Z} = \r_{Z_{0}}/4.$ In particular, for each $Z \in B_{u} \cap B_{\r_{Z_{0}}/2}(Z_{0})$, 
$$\s^{-n}\int_{B_{\s}(0)} |\e_{Z}|^{2} \leq C_{Z_{0}}\s^{k_{Z_{0}}+ \g_{Z_{0}}}$$ 
for $s \in (0, \r_{Z_{0}}/4).$ It only remains to establish the sup estimate on $\{\pm \e_{Z}\}$. Fix 
$Z \in {\mathcal B}_{u} \cap B_{\r_{Z_{0}}/2}(Z_{0})$ and let 
$u_{Z}(X) = u(Z + Q_{Z} \, X)$ and $\s \in (0, \r_{Z_{0}}/4),$ where $Q_{Z}$ is the orthogonal rotation of ${\mathbb R}^{n}$ as in Theorem B. Note that by Corollary~\ref{cor6_4} and the preceding estimate, 
$${\mathcal B}_{u_{Z}} \cap B_{\s/2}(0) \subset  \{X \, : \, {\rm dist} \, (X, \{0\} \times {\mathbb R}^{n-2}) \leq C_{Z_{0}}\s^{1+\g_{Z_{0}}/2}\},$$ 
so if $X \in B_{\s/2}(0)$ and ${\rm dist} \, (X, \{0\} \times {\mathbb R}^{n-2}) > \s^{1+\g_{Z_{0}}/(2n)} \equiv \d$, then 
${\mathcal B}_{u_{Z}} \cap B_{\d/2}(X) = \emptyset$  provided $\r_{Z_{0}}$ is sufficiently small;   hence by standard estimates for harmonic functions we have that $|\e_{Z}(X)|^{2} \leq C \d^{-n} \int_{B_{\d/2}(X)} |\e_{Z}|^{2} \leq C \left(\frac{\s}{\d}\right)^{n} \s^{-n}\int_{B_{\s}(0)}|\e_{Z}|^{2} \leq CC_{Z_{0}}\s^{k_{Z_{0}} +\g_{Z_{0}}/2}$ where $C = C(n).$  If on the other hand $X \in B_{\s/2}(0)$ and ${\rm dist} \, (X, \{0\} \times {\mathbb R}^{n-2}) \leq  \d$, then, with $r = \frac{1}{2}{\rm dist} \, (X, {\mathcal B}_{u_{Z}})$ and $v_{Z}(X) = u_{Z}(X) - h(Z + Q_{Z}X)$, we have by Theorem B, the triangle inequality and estimates for harmonic functions, $|\e_{Z}(X)|^{2} \leq 2|v_{Z}(X)|^{2} + C_{1}\left({\rm dist} \, (X, \{0\} \times {\mathbb R}^{n-2})\right)^{k_{Z}} \leq Cr^{-n}\int_{B_{r}(X)} |v_{Z}|^{2} + C_{1}\d^{k_{Z_{0}}}$ where $C_{1} = C_{1}(Z_{0})$ and $C = C(n).$ Choosing $Y \in {\mathcal B}_{u_{Z}}$ such that ${\rm dist} \, (X, {\mathcal B}_{u_{Z}}) = |X - Y|$, we see with the help of (\ref{doubling_estimate2}) that $r^{-n} \int_{B_{r}(X)} |v_{Z}|^{2} \leq r^{-n}\int_{B_{2r}(Y)}|v_{Z}|^{2} \leq 2^{n}\left(\frac{8r}{\r_{Z_{0}}}\right)^{k_{Z_{0}}} \left(\frac{\r_{Z_{0}}}{4}\right)^{-n} \int_{B_{\r_{Z_{0}}/2}(0)}|v_{Z}|^{2}$, and hence, since $\int_{B_{\r_{Z_{0}}/2}(0)} |v_{Z}|^{2} \leq 
\int_{B_{\r_{0}}(Z_{0})}|u - h|^{2}$ and $r \leq \frac{1}{2} \, {\rm dist} \, (X, \{0\} \times {\mathbb R}^{n-2}),$ it follows that $|\e_{Z}(X)|^{2} \leq C\d^{k_{Z_{0}}} = C\s^{k_{Z_{0}} + k_{Z_{0}}\g_{Z_{0}}/(2n)} \leq 
C\s^{k_{Z_{0}} + \g_{Z_{0}}/(2n)}$ where $C$ is independent of $\s$ and $Z$. The theorem is thus proved.
\end{proof}

\end{document}

%% file: mymacros.tex
\newcommand\beqn{\begin{equation}}
\newcommand\eeqn{\end{equation}}
\newcommand\beqny{\begin{eqnarray}}
\newcommand\eeqny{\end{eqnarray}}
\newcommand\beqnyn{\begin{eqnarray*}}
\newcommand\eeqnyn{\end{eqnarray*}}

\usepackage{amsmath,amsfonts,amssymb,amsthm}
\newtheorem{theorem}{Theorem}[section]
\newtheorem{lemma}[theorem]{Lemma}
\newtheorem{proposition}[theorem]{Proposition}
\newtheorem{corollary}[theorem]{Corollary}

\def\a{\alpha}           
\def\b{\beta}          
\def\g{\gamma}           
\def\d{\delta}         
\def\e{\epsilon}         
\def\z{\zeta}            
\def\th{\theta}          
\def\k{\kappa}           
        
\def\m{\mu}       

\def\t{\tau}

\def\r{\rho} 

\def\s{\sigma}

\def\G{\Gamma}